\documentclass[11pt,a4paper]{article}
\usepackage[latin1]{inputenc}
\usepackage[a4paper]{geometry}
\usepackage{amsfonts}
\usepackage{amsthm}
\usepackage{amsmath}
\usepackage{bbm}
\usepackage{parskip}
\usepackage{mathrsfs}
\usepackage{graphicx}
\usepackage{amssymb}
\usepackage{xtab}
\usepackage{setspace}
\usepackage{fancyhdr}
\usepackage[dvipsname]{xcolor}
\usepackage{color,soul}
\usepackage[nodayofweek]{datetime}
\usepackage{hyperref}
\newdateformat{monthyear}{\monthname[\THEMONTH], \THEYEAR}
\newcommand{\N}{\mathbb{N}}
\newcommand{\C}{\mathbb{C}}
\newcommand{\Z}{\mathbb{Z}}
\newcommand{\E}{\mathbb{E}}
\newcommand{\R}{\mathbb{R}}
\newcommand{\p}{\mathbb{P}}

\newcommand{\Partial}[2]{\frac{\partial #1}{\partial #2}}

\newcommand{\scndmxPartial}[3]{\frac{\partial^2 #1}{\partial #2 \partial #3}}

\newcommand{\Weylo}[1]{\mathbb{W}^{#1}}
\newcommand{\Weyl}[1]{\overline{\mathbb{W}^{#1}}}
\newcommand{\delWeyl}[2]{\mathbb{W}^{#1}_{#2}}
\newtheorem{proposition}{Proposition}[section]
\newtheorem{lemma}[proposition]{Lemma}
\newtheorem{conjecture}[proposition]{Conjecture}
\newtheorem{theorem}[proposition]{Theorem}

\newtheorem{corollary}[proposition]{Corollary}

\newtheorem{remark}[proposition]{Remark}

\newtheorem{definition}[proposition]{Definition}

\DeclareMathOperator{\sign}{sign}
\renewcommand{\epsilon}{\varepsilon}

\renewcommand{\Re}{\text{Re}}
\title{The Bethe Ansatz for Sticky Brownian Motions}
\author{Dom Brockington, Jon Warren}
\begin{document}
\maketitle
\begin{abstract}
    We consider a diffusion in $\R^n$ whose coordinates each behave as one-dimensional Brownian motions, that behave independently when apart, but have a sticky interaction when they meet. The diffusions in $\R^n$ can be viewed as the $n$-point motions of a stochastic flow of kernels. We derive the Kolmogorov backwards equation and show that for a specific choice of interaction it can be solved exactly with the Bethe ansatz. We then use our formulae to study the behaviour of the flow of kernels for the exactly solvable choice of interaction.
\end{abstract}
\section{Introduction}
In this paper we study a diffusion in $\R^n$, the coordinates of which evolve as independent one-dimensional Brownian motions when they are distinct, and have an attractive, so called sticky, interaction when they are equal. The diffusion can be interpreted as the evolving positions of $n$ particles on the real line, which interact when they meet. In particular the difference between two coordinates is described by a one dimensional sticky Brownian motion, recently studied as the weak solution to an SDE in $\cite{bass2014},$ and $\cite{EngelbertPeskir}$. Sticky Brownian motion, with parameter $\theta>0$, is a diffusion in $\R$ on natural scale with speed measure $m(dx)= 2 dx + \frac{2}{\theta}\delta_0(dx)$. The diffusion in $\R^n$ can visit the diagonal $\{x\in \R^n| \ x_1=...=x_n\}$ for a set times with positive Lebesgue measure, quite unlike a standard Brownian motion in $\R^n$. The interaction between coordinates at such times is not determined solely by specifying the parameter $\theta$ describing the stickiness. It was shown in \cite{HowittWarren} that the possible interactions can be specified by a finite measure on $[0,1]$ called the characteristic, or splitting, measure. The diffusions are consistent, in that for any $k<n$, any $k$ coordinates of the sticky Brownian motions in $\R^n$ with characteristic measure $\nu$, are sticky Brownian motions in $\R^k$ with the same characteristic measure, $\nu$. An example of such a diffusion was originally investigated by Le Jan and Raimond \cite{StickyFlowsCircle} using Dirichlet forms (on the torus rather than Euclidean space), and then the more general case was studied by Howitt and Warren \cite{HowittWarren}, via a martingale problem which we describe later.  \par{}

The consistency property means that we can also consider such systems of sticky Brownian motions to be the $n$-point motions of a stochastic flow of kernels. A flow of kernels $(K_{s,t}(x,dy))_{s\leq t}$ is essentially a random family of transition probability measures for a Markov process. Flows of kernels were introduced by Le Jan and Raimond in \cite{lejan2004} as a generalisation of flows of maps, to study stationary evolutions of turbulent fluids. The $n$-point motions can then be thought of as describing the behaviour of $n$ particles thrown into the fluid. Stochastic flows of kernels whose $n$-point motions are described by sticky Brownian motions are called Howitt-Warren flows in \cite{FlowsinWebandNet}, where their properties are studied in detail. Gawedzki and Horvai, \cite{Gawedzki2004}, discovered that, for two particles, sticky behaviour arose in certain limits of the Kraichnan model for turbulent advection. Warren then proved the convergence for $n$ particles and found the characteristic measure for the resulting sticky Brownian motions \cite{Warren2015}. Sun, Swart and Schertzer studied Howitt-Warren flows, constructing them directly as flows of mass in the Brownian web \cite{FlowsinWebandNet} by marking special separation points and attaching extra random variables to them that tells the mass following a path in the web how to split. The law of these additional random variables is described by the characteristic measure. Amongst other results they showed that the Howitt-Warren flows are almost surely purely atomic, at deterministic times. \par{}

The Howitt-Warren flow can be thought of as the continuum analogue to the random transition probabilities of the random walk in a random environment. Consistent with this sticky Brownian motions arise as scaling limits of the $n$-point motions of random walks in space-time i.i.d. random environments. A special case of the random walk in a random environment (RWRE) models, where the environment is Beta distributed, was shown by Barraquand and Corwin \cite{Barraquand2017} to be exactly solvable, they found exact solutions for the point to half line probabilities. This was shown using the Bethe ansatz and a non-commutative binomial formula from \cite{Povolotsky}. These exact solutions were then used to establish that there are GUE Tracey-Widom fluctuations in the large deviations of the random walk in a beta random environment. A straightforward calculation, see Section \ref{Section:RWRE}, shows that the scaling limit of random walks in a Beta random environment corresponds to the sticky Brownian motions with a uniform characteristic measure. Barraquand and Rychnovsky \cite{barraquand2019large}, working independently of us, derived exact solutions for the point to half-line probabilities of sticky Brownian motions with uniform characteristic measure by taking an appropriate scaling limit of the RWRE case. An asymptotic analysis then led to the discovery of GUE Tracey-Widom fluctuations in the large deviations of sticky Brownian motions as well. In this paper we will derive the Kolmogorov backwards equation from the martingale problem characterisation for the sticky Brownian motions with a uniform characteristic measure. Then we shall apply the Bethe ansatz to find an exact formula for the transition density of this process. The choice of uniform characteristic measure seems to be essential, only in this case is the diffusion exactly solvable by the Bethe ansatz. Further this seems to be the only case the diffusion is reversible, at least with respect to a measure we can write down explicitly. Our method is similar to that used by Tracey and Widom for the delta Bose gas \cite{DynamicsOfBoseGas}, however the importance of interactions between more than two particles adds significant complexity.

Before we introduce our main result we must define some terms. We use the notations $\Weylo{n}:= \{x\in \R^n| \ x_1 > x_2 > ...> x_n\}$ and $\Weyl{n}:= \{x\in \R^n| \ x_1\geq ...\geq x_n\}$ for the principal Weyl chamber, the images of this set under a permutation are called simply Weyl chambers, however we may sometimes refer to the principal Weyl chamber as just the Weyl chamber. By $C^2_0(\Weyl{n})$ we mean the set of functions $f:\Weyl{n} \to \R$ that have a $C^2$ extension to some open set containing $\Weyl{n}$, such that $f$ and all of its first and second partial derivatives vanish at infinity. Let $\Pi_n$ denote the collection of ordered partitions, $(\pi_1,...,\pi_k)$, of $\{1,..,n\}$ such that if $a\in \pi_j$, $b\in \pi_k$ and $j<k$ then $a<b$. That is the elements of the partition each consist of intervals intersected with $\Z$ and are indexed according to the size of their elements. \par{}

To each partition $\pi\in \Pi_n$ we associate a subset of $\Weyl{n}$ defined by 
\begin{equation*}
	\delWeyl{n}{\pi}:= \{x\in \Weyl{n}| \ x_\alpha= x_\beta \text{ if and only if there is a } \pi_i\in \pi \text{ such that } \alpha,\beta \in \pi_i \}.
\end{equation*}
In other words all the points in $\Weyl{n}$ whose coordinates are equal if and only if their indices are in the same element of $\pi$. Notice for $\pi= \{ \{1\},...,\{n\} \}$, $\delWeyl{n}{\pi}= \Weylo{n}$, in addition $\Weyl{n}= \cup_{\pi \in \Pi_n} \delWeyl{n}{\pi}$, and the sets $\delWeyl{n}{\pi}$ are disjoint. It's clear that there is a natural continuous bijection $I^\pi:\delWeyl{n}{\pi} \to \Weylo{|\pi|}$, given by $I^\pi(x)= (x_{p_1},...,x_{p_{|\pi|}})$ for some choice of $p_i \in \pi_i$. We can now define a Borel measure on $\delWeyl{n}{\pi}$ as the pushforward of the Lebesgue measure $\lambda$ on $\Weylo{|\pi|}$, $\lambda^\pi:= I^\pi_* \lambda$. This extends to a Borel measure on $\Weyl{n}$ via the formula $\lambda^\pi(A):=\lambda^\pi(A\cap \delWeyl{n}{\pi})$. \par{}

\begin{definition}\label{ReferenceMeasureOfTheTdensity}
	For $\theta>0$ the Borel measure $m^{(n)}_\theta$ on $\Weyl{n}$ is defined as
	\begin{equation*}
		m^{(n)}_\theta := \sum_{\pi \in \Pi_n} \theta^{|\pi|-n} \prod_{\pi_\iota \in\pi} \frac{1}{ |\pi_\iota|} \lambda^{\pi}.
	\end{equation*}
\end{definition}
Suppose $\theta>0$ and $X=(X(t))_{t\geq 0}$ is a solution to the Howitt-Warren martingale problem (sticky Brownian motions) in $\R^n$ with characteristic measure $\frac{\theta}{2} \mathbbm{1}_{[0,1]}dx$ and zero drift. Then we define $Y=(Y(t))_{t\geq 0}$ as the process obtained by ordering the coordinates of $X$, i.e. for each $t\geq 0$ $Y(t)= (Y^1(t),...,Y^n(t))=(X^{\sigma(1)}(t),...,X^{\sigma(n)}(t))$ for some $\sigma\in S_n$ such that $Y^1(t)\geq...\geq Y^n(t)$.
\begin{theorem}\label{Intro:MainResult}
	For every bounded Lipschitz continuous function $f:\Weyl{n}\to \R$, $x \in \Weyl{n}$ and $t>0$
	\begin{equation*}
		\E_x[f(Y_t)]= \int u_t(x,y) f(y) m^{(n)}_\theta(dy).
	\end{equation*}
	Where $u_t:\R^{n}\times \R^n \to \R$ is defined for each $t>0$ by
	\begin{equation*}
		u_t(x,y) := \frac{1}{(2\pi)^n}\int_{\R^n} e^{-\frac{1}{2} t |k|^2} \sum_{\sigma\in S_n} e^{ik_\sigma \cdot (x - y_\sigma)} \prod_{\substack{\alpha <\beta : \\ \sigma(\beta) < \sigma(\alpha)}}\tfrac{i\theta(k_{\sigma(\alpha)}- k_{\sigma(\beta)}) + k_{\sigma(\beta)} k_{\sigma(\alpha)}}{i\theta(k_{\sigma(\alpha)}- k_{\sigma(\beta)}) - k_{\sigma(\beta)} k_{\sigma(\alpha)}} dk,
	\end{equation*}
	where $S_n$ denotes the group of permutations on $\{1,...,n\}$ and $k_\sigma = (k_{\sigma(1)},...,k_{\sigma(n)})$.
\end{theorem}\par{}
\begin{remark}
	Note that the function $u$ is well defined (the integral always converges), because for every $x,y, k \in \R^n$, and every permutation $\sigma \in S_n$
	\begin{equation*}
		\left|e^{ik_\sigma \cdot (x - y_\sigma)} \prod_{\substack{\alpha <\beta : \\ \sigma(\beta) < \sigma(\alpha)}}\tfrac{i\theta(k_{\sigma(\alpha)}- k_{\sigma(\beta)}) + k_{\sigma(\beta)} k_{\sigma(\alpha)}}{i\theta(k_{\sigma(\alpha)}- k_{\sigma(\beta)}) - k_{\sigma(\beta)} k_{\sigma(\alpha)}} \right| = 1.
	\end{equation*}
	The function above is not continuous at points where there are distinct $\alpha, \beta$ such that $k_\alpha=k_\beta=0$ (where the denominator vanishes), but since the modulus is constant the value can simply chosen to be $1$ here, and it does not affect the integral because such points have measure zero. It is easily seen that we can pass derivatives under the integral, and thus we have $u_t(\cdot, y) \in C^2_0(\R^{n})$ for all $t>0$ and $y\in \R^n$. In particular  $u_t(\cdot, y) \in C^2_0(\Weyl{n})$ for all $t>0$ and $y\in \Weyl{n}$, when restricted to $\Weyl{n}$. It is also the case, as we will show later, that $u_t(x,y)=u_t(y,x)$ for all $t>0$ and $x,y\in \R^n$.
\end{remark}

\begin{remark}
    Another representation $u_t(x,y)$ is in terms of a product of eigenfunctions of the generator of the process $Y$. For each $k\in \R^n$ we have an eigenfunction given by
    \begin{align*}
        E_k( x) := \sum_{\sigma\in S_n} e^{k_\sigma\cdot x} \prod_{\substack{\alpha <\beta : \\ \sigma(\beta) < \sigma(\alpha)}}\tfrac{i\theta(k_{\sigma(\alpha)}- k_{\sigma(\beta)}) + k_{\sigma(\beta)} k_{\sigma(\alpha)}}{i\theta(k_{\sigma(\alpha)}- k_{\sigma(\beta)}) - k_{\sigma(\beta)} k_{\sigma(\alpha)}}.
    \end{align*}
    The transition density is given by
    \begin{align*}
        u_t(x,y) = \frac{1}{(2\pi)^n}\int_{\Weylo{n}} e^{-\frac{1}{2}t|k|^2} E_k(x) \overline{E_k(y)} dk.
    \end{align*}
    The proof the above expression for $u_t(x,y)$ agrees with the one previously given is straightforward, and so omitted.
\end{remark}

We further prove that $m^{(n)}_\theta$ is in fact the stationary measure of the process $Y$, and reversibility of the process $Y$ with respect to $m^{(n)}_\theta$. \par{}

The Howitt-Warren flows are almost surely purely atomic; it is possible to interpret the values of the transition densities of the ordered $n$-point motions, the process $Y$ above, as the moments of the size of the atom at a given location. Using this interpretation we consider the fluctuations of the sizes of the atoms as $t\to \infty$ and find them to be exponentially distributed when taken to be $\sim \sqrt{t}$ away from the origin, with parameter determined by $\theta$. This is similar to the Gamma fluctuations found in the same regime for the point to point probabilities of the Beta random walk in a random environment by Thierry and Le Doussal \cite{ExactSolutionRWRE1dThieryLeDoussal}. The same authors also found 
that in the large deviation regime the fluctuations have Tracey-Widom GUE fluctuations, just as for the point to half-line probabilities. It thus seems reasonable to conjecture the same fluctuations appear in the size of atoms of the Howitt-Warren flows, but we do not pursue the necessary asymptotic analysis here.

The outline of the paper is as follows: In Section \ref{A Consistent Family of Sticky Brownian Motions} we define the diffusion via a martingale problem, in Section \ref{Section: The Backwards Equation} we derive the Kolmogorov Backwards equation for the ordered $n$-point motions, and show that the generator of the process is symmetric with respect to the measure $m^{(n)}_\theta$ when restricted to a certain class of $C^2$ functions. In Section \ref{Bethe ansatz for Sticky Brownian motions} we show that the backwards equation is solvable by the Bethe ansatz, and as a consequence we show that the ordered $n$ point motions are reversible with respect to $m^{(n)}_\theta$. Finally in Section \ref{Stochastic flows of kernels} we introduce stochastic flows of kernels, and apply our results to Howitt-Warren flows .

\section{A Consistent Family of Sticky Brownian Motions}\label{A Consistent Family of Sticky Brownian Motions}
We introduce the \textit{Howitt-Warren martingale problem} in $\R^n$ with drift $\beta\in \R$ and characteristic measure $\nu$ (a finite measure
on $[0,1]$), as formulated in \cite{HowittWarren}.
Solutions are processes in $\R^n$, representing the positions of $n$ particles, each moving as one dimensional Brownian motions with drift $\beta$, such that two or more particles undergo sticky interactions, determined by $\nu$, when they meet. The solutions are consistent, in the sense that if $X$ is the solution to martingale problem in $\R^n$ with characteristic measure $\nu$ and drift $\beta$, then for any choice of distinct $i_1,...,i_k \in \{1,...,n\}$ with $k<n$, $(X^{i_j})_{j=1}^k$ is a solution to the martingale problem in $\R^k$ with characteristic measure $\nu$ and drift $\beta$.\par{}
To each point $x\in \R^n$ we associate a partition of the set $\{1,...,n \}$, $\pi(x)$, where $i,j \in \{1,...,n\}$ are in the same component of $\pi(x)$ if and only if $x_i = x_j$.
 Next we define, for each pair of disjoint subsets $I,J \subset \{1,..., n\}$, the vectors $v_{I,J} \in \R^n$ given by
\begin{equation*}
(v_{I,J})_i = \begin{cases}
1, \quad \text{if } i \in I; \\
-1, \quad \text{if } i \in J;\\
0, \quad \text{otherwise.}
\end{cases}
\end{equation*}
Note that $I$ and $J$ are allowed to be empty. Then we define the set of vectors $\mathcal{V}(x)$ as
\begin{equation*}
	\mathcal{V}(x):= \{ v_{I,J}: \quad I\cup J \in \pi(x), \ I\cap J = \emptyset. \}.
\end{equation*}
\par
$\mathcal{V}(x)$ keeps track of the directions in which the process can infinitesimally move. We'll use this to describe the interactions. Define the parameters $\theta(k, l)$ for $k,l\geq 1$ by
\begin{equation}\label{thetaDefinition}
\theta(k,l) := \int_0^1 x^{k-1} (1-x)^{l-1}\nu(dx),
\end{equation}
and for $k,l\geq 0$ by $\theta(1,0) - \theta(0,1) = \beta$, $\theta(0,0)=0$, requiring the consistency property $\theta(k,l) = \theta(k+1, l) + \theta(k,l+1)$ for all $k,l\geq 0$ gives definition to all $k,l\geq 0$. \par{}
\begin{definition}
	Let $D_n$ be the collection of functions $f:\R^n\to \R$  which are continuous and are such that for all Weyl chambers $A\subset \R^n$ the restriction of $f$ to $A$ is linear, so that  if $A\subset \R$ is a Weyl chamber and $x,y \in A$ then $f(x+y)= f(x)+ f(y)$.
\end{definition}
For functions $f\in D_n$ we define the operator $\mathcal{A}^\theta_n$ by
\begin{equation*}
\mathcal{A}^\theta_n f(x) := \sum_{v_{I,J}\in \mathcal{V}(x)} \theta(|I|,|J|) \nabla_{\substack{v_{I,J}}} f(x).
\end{equation*}
Where $\nabla_{v_{I,J}}$ denotes the one sided derivative in direction $v_{I,J}$.
\begin{definition}\label{Def:HWMartProb}
	Let $(X(t))_{t\geq 0} = \big(\big( X^1(t) ,... , X^n(t) \big)\big)_{t\geq 0} \subset \R^n$ be a continuous square integrable semi-martingale with initial condition $X(0) = x \in \R^n$, defined on a filtered probability space $(\Omega, \mathcal{F}, (\mathcal{F}_t)_{t\geq 0}, \p)$. Then $(X(t))_{t\geq 0}$ is a solution to the \textbf{Howitt-Warren martingale problem} with drift $\beta$ and characteristic measure $\nu$ if for any $i,j \in \{1,...,n\}$:
\begin{equation*}
	\langle X^i, X^j \rangle(t) = \int_0^t \mathbbm{1}_{\{X^i(s)=X^j(s)\}} ds,
\end{equation*}
and the following process is a martingale, for every function $F \in D_n$,
\begin{equation*}
	F(X(t)) - \int_{0}^{t} \mathcal{A}^\theta_n F(X(s)) ds.
\end{equation*}	
\end{definition}
Note that the first condition implies that $\langle X^i, X^i \rangle(t) = t$, and it follows from the second condition and the definition of $\mathcal{A}^\theta_n$ that $X^i(t)- \beta t$ is a martingale for each $i$. Hence each coordinate must be a Brownian motion with drift $\beta$. Well posedness of this martingale problem, and that solutions do indeed form a consistent family of Feller processes is shown in \cite{HowittWarren}. \par

\section{The Backwards Equation} \label{Section: The Backwards Equation}
\subsection{The Generator of Ordered Sticky Brownian Motions }
Define the functions $F^{(i)}:\R^n\to \R$ by $F^{(i)}(x)=x_j$ where $x_j$ is the $i^{th}$ largest coordinate of $x$, and $F:\R^n \to \Weyl{n}$ by $F(x) := (F^{(1)}(x),..., F^{(n)}(x))$. Note that these functions are in $D_n$. Further, suppose $X=(X(t))_{t\geq 0}$ is a solution to the Howitt-Warren martingale problem in $\R^n$ with characteristic measure $\nu$, and drift $\beta = 0$, with initial condition $x\in \Weyl{n}$. Then define the process $Y=(Y(t))_{t\geq 0}$ by $Y(t) :=F(X(t))$, with initial condition $x\in \Weyl{n}$. Note that we defined $Y$ from $x$ started inside the Weyl chamber. This process lies entirely in the Weyl chamber $\Weyl{n}$, making it admissible to the Bethe ansatz.  The object of this section is to identify the Kolmogorov Backwards equation for $Y$, and from it the invariant measure for $Y$. \par{}

\begin{remark}
	Before talking about its Kolmogorov Backward equation we need to know $Y$ is a Markov process. For this we refer to Dynkin's criterion \cite{rogers1981}. In particular, we only need to show that $\E_{x}\left[ f\circ F(X_t) \right] = \E_{F(x)}[f(Y_t)]$ for every $x\in \R^n$. This holds by definition for $x\in \Weyl{n}$, for $x\in \R^n \setminus \Weyl{n}$ we need to show that for any permutation $\sigma \in S_n$ $\sigma(X(t)):= (X^{\sigma(1)}(t),..., X^{\sigma(n)}(t))$ remains a solution to the same Howitt-Warren martingale problem, but with initial condition $\sigma(x)$. It's clear $\sigma(X)$ remains a continuous square integrable semi-martingale, and has initial condition $\sigma(x)$. Further it's clear that $\sigma(X)$ has the correct quadratic variations. Finally, because the function $\sigma$ is a continuous, linear, and maps Weyl chambers to Weyl chambers, $\{ F\circ \sigma: F\in D_n\}= D_n$, so that the martingale problem is still satisfied by $\sigma(X)$. It's clear there exists a permutation $\sigma \in S_n$ such that $\sigma(x) \in \Weyl{n}$, and by definition $\sigma(x) = F(x)$. By uniqueness of solutions to the martingale problem we have $\E_x[ f\circ F(X_t)] = \E_x[ f \circ F \circ \sigma^{-1} \circ \sigma(X_t)] = \E_{\sigma(x)} [f\circ F \circ \sigma^{-1} (X_t)]$ but clearly $F \circ \sigma^{-1} = F$. Hence $\E_x[ f\circ F(X_t)] = \E_{\sigma(x)} [f\circ F(X_t)] = \E_{F(x)}[f(Y_t)]$ as required. Thus $Y= F(X)$ is a Markov process.
\end{remark}

We proceed by considering the action of the generator of $Y$ on certain $C^2$ functions.
\begin{definition}\label{Domain of the Generator}
	Let $\mathcal{D}_\theta$ denote the set of functions $f\in C^2_0(\Weyl{n})$ such that for any $a,b \in \{1,...,n\}$ with $a<b$, $x_a=x_b$ implies
	\begin{align}
		& \frac{1}{2}\sum_{\substack{a\leq i,j \leq b: \\ i\neq j}} \frac{\partial^2 f}{ \partial x_i \partial x_j} (x) \nonumber \\
		& \quad =  \sum_{i =a}^b \frac{\partial f}{\partial x_i}(x) \sum_{k=0}^{b-a+1} \binom{b-a+1}{k} \theta(k, b-a+1-k) \sign(k-i+a-1). \label{Complicated BCs-1}
	\end{align}
	Where $\sign(0)$ is taken to be $1$ here.
\end{definition} 
\begin{proposition}\label{Prop:TheGenerator}
	Suppose $f\in \mathcal{D}_\theta$ then, denoting the generator of the process $Y$ by $\mathcal{G}_\theta$ (in the sense of \cite{Revuz&Yor}), we have
	\begin{equation*}
		\mathcal{G}_\theta f = \frac{1}{2}\Delta f.
	\end{equation*}
\end{proposition}
The same calculations will also give us a backwards equation for the process.
\begin{proposition}\label{Prop:TheBackwardsEquation}
    Suppose $g\in C^2(\R_{> 0}\times \ \Weyl{n})$, and $g(t,\cdot)\in \mathcal{D}_\theta$ for all $t> 0$. Further suppose that $g$ satisfies the PDE
    \begin{equation}\label{HeatEq}
        \Partial{g}{t} = \frac{1}{2}\Delta g, \ \text{for all } t> 0, \ x\in \Weyl{n}.
    \end{equation}
    With initial condition $g(0, x)=f(x)$ with $f\in C_b(\Weyl{n})$, i.e. $g(t,\cdot)\to f$ uniformly as $t\to 0$. Then for each $t>0$  $\left(g(t-s, Y(s))\right)_{s\in[0,t]}$ is a continuous local martingale.
\end{proposition}
\begin{proof}[Proof of Proposition \ref{Prop:TheGenerator}]
	 Since $X$ solves the martingale problem, and $F^{(i)}\in D_n$, $Y$ is a semi-martingale. For $f\in C^2_0(\Weyl{n})$, It\^o's formula gives 
	\begin{align*}
		&\E_x[f(Y(t))]=\\
		&\quad f(x) + \sum_{i=1}^n\E_x\left[\int_0^t \frac{\partial f}{\partial x_i} (Y(s)) dY^i(s)\right] +\frac{1}{2} \sum_{i,j=1}^n \E_x\left[\int_{0}^{t} \frac{\partial^2 f}{\partial x_i \partial x_j}(Y(s)) d\langle Y^i,Y^j \rangle(s)\right].
	\end{align*}
	We need to calculate the quadratic covariations for $Y$. Denoting $P_i = \{A \subset \{1,..., n\} | \ |A|= n-i+1\}$, we  can define $f_A: \R^n \to \R$ as $f_A(x)= \max_{a \in A} x_a$ and $g_i: \R^{P_i}\to \R$ as $g((y_A)_{A\in P_i})= \min_{A \in P_i} y_A$. Then $F^{(i)}(x) = g_i \left( (f_A(x))_{A \in P_i} \right)$, in $f_A$ is a convex function and $g_i$ is a concave function. Referring to \cite[Proposition 8]{SemiMartingaleDecompForConvexFunctions} we can write the local martingale part of $F^{(i)}(X)$ in terms of a linear combination of stochastic integrals with respect to the $X^i$. In particular We can write
	\begin{align*}
		f_A(X_t) = f_A(x) + \sum_{a \in A} \int_0^t \mathbbm{1}_{B_a^A} (X_s) dX^a_s + C_t.
	\end{align*}
	Where $C_t$ has finite variation, and $B_a^A = \{ x: \ \min_{k\in A}\{ k: \ \max_{j\in A} x_j = x_k \} = a \}$. Notice that for a fixed $x$ and $A$ there is only one $a$ such that $\mathbbm{1}_{B_a^A}(x)$ is non zero.\par{}

	Now we put an ordering on the set $P_i$, the specific ordering does not matter we just need to be able to minimise over the indices of elements in $\R^{P_i}$. Suppose $A, B \in P_i$ are distinct, we first define $(a_j)_{j=1}^{n-i+1}$ and $(b_j)_{j=1}^{n-i+1}$ to be the elements of $A$ and $B$ respectively in increasing order. We say $A<B$ if for $l:= \min \{k \in\N: \ b_k\neq a_k, \ 1\leq k \leq n-i+1\}$, we have $a_l< b_l$, if instead $b_l< a_l$ then $B<A$ so this is a total ordering for $P_i$. Supposing $Z$ is a semi-martingale taking values in $\R^{P_i}$, with decomposition $Z_t = Z_0 + M_t + K_t$, where $M$ is a local martingale, and $K$ a process with finite variation. Then using that for $y\in \R^{P_i}$ $-g_i(-y)= -\max_{A\in P_i} (- y_A)$, we have
	\begin{align*}
		-g_i (-Z_t) = -g_i(-Z_0) + \sum_{A\in P_i}\int_0^t \mathbbm{1}_{B_A} (Z_s) dZ^A_s + D_t.
	\end{align*}
	Where $D$ has finite variation, $B_A:= \{ z \in \R^{P_i}: \ \min\{B\in P_i: \ \inf_{C\in P_i} z_C = z_B \} = A  \}$ with the minimum understood in terms of the ordering we just defined on $P_i$. Notice that for a fixed $z$ there is only one $A$ such that $\mathbbm{1}_{B_A}(z)$ is non zero. The local martingale part of $Y^i = g_i((f_A(X))_{A\in P_i})$ is given by
	\begin{align*}
		\sum_{A \in P_i} \sum_{a \in A} \int_0^t \mathbbm{1}_{B^A_a}(X_s) \mathbbm{1}_{B_A}((f_C(X_s))_{C\in P_i}) dX^a_s. 
	\end{align*}
	Giving that the quadratic covariations are
	\begin{align*}
		\langle Y^i, Y^j \rangle_t = \sum_{\substack{A \in P_i,\\ B\in P_j}} \sum_{\substack{a \in A,\\ b\in B}} \int_0^t  \mathbbm{1}_{B^A_a}(X_s) \mathbbm{1}_{B_A}((f_C(X_s))_{C\in P_i}) \mathbbm{1}_{B^B_b}(X_s) \mathbbm{1}_{B_B}((f_C(X_s))_{C\in P_i}) \mathbbm{1}_{\{X^a_s = X^b_s\}}ds.
	\end{align*}
	Recall $f_C(x) = \max_{c \in C} x_c$ so that $\mathbbm{1}_{B_A}((f_C(x))_{C\in P_i})$ is non zero precisely when $A$ is the subset of $\{1,...,n\}$ with indices corresponding to the first $i-1$ largest coordinates of $X_s$ removed, call this set $A_i(X_s)$. Then $\mathbbm{1}_{B_a^{A_i(X_s)}}(X_s)$ is non zero if and only if $a$ is the smallest element of $\{1,..., n\}$ such that $X^a_s$ is equal to the $i$th largest coordinate of $X_s$, i.e. $Y^i_s$. Hence we have
	\begin{align*}
		\langle Y^i, Y^j \rangle_t =& \int_0^t \mathbbm{1}_{\{Y^i_s = Y^j_s\}} ds.
	\end{align*}

	The martingale problem also tell us that for each $i$
	\begin{equation*}
		Y^i(t) - \int_0^t \mathcal{A}^\theta_n F^{(i)}(X(s))ds
	\end{equation*}
	is a martingale. Recall $f\in C^2_0(\Weyl{n})$, thus $\Partial{f}{x_i}$ is bounded on $\Weyl{n}$ so that the stochastic integral with respect to the martingale part of $Y$ is a true martingale. Thus we can rewrite the expectation as
	\begin{align}
		\E_x[f(Y(t))]= f(x) &+ \sum_{i=1}^n\E_x\left[\int_0^t \frac{\partial f}{\partial x_i}(Y(s))\mathcal{A}^\theta_n F^{(i)}(X(s))ds\right] \nonumber \\
		&+ \frac{1}{2} \sum_{i,j=1}^n \E_x\bigg[\int_{0}^{t} \frac{\partial^2 f}{\partial x_i \partial x_j}(Y(s)) \mathbbm{1}_{\{ Y^i(s)=Y^j(s) \}}ds\bigg]. \label{Mean of f(X)}
	\end{align}
	By evaluating $\mathcal{A}^\theta_n F^{(i)}$, and then differentiating equation (\ref{Mean of f(X)}) in time, we can determine the generator of $Y$.\par{}
	Let $x\in \R^n$ and denote $y= F(x)\in \Weyl{n}$. We have
	\begin{align}
		\mathcal{A}^\theta_n F^{(i)}(x)&=\sum_{v\in\mathcal{V}(x)} \theta(v) \nabla_v F^{(i)}(x). \label{MartingaleProblemGenerator}
	\end{align}
	Where $\nabla_v$ is the directional derivative in direction $v$. Recall $v\in \mathcal{V}(x)$ is defined by the disjoint subsets $I,J \subset \{1,..., n\}$ such that $I\cup J \in \pi(x)$. With $v_i = 1$ if $i\in I$, $-1$ if $i \in J$, and $0$ otherwise. For each element, $B$, of the partition $\pi(x)$ there is a corresponding element, $C$, of the partition $\pi(y)$ such that for each $i \in B$ there is a $j_i \in C$ with $x_i=y_{j_i}$, and the $j_i$ can be chosen so that the mapping $i\mapsto i_j$ is injective. Letting $C$ denote the element of $\pi(y)$ corresponding to $I\cup J\in \pi(x)$, it's clear that if $i \notin C$ then $\nabla_v F^{(i)}(x) =0$, and for $i\in C$ the derivative is either $1$ or $-1$ depending only on the sizes of $I$ and $J$. Since $y\in \Weyl{n}$ there is an $a \in \{1,..., n\}$ and $m>0$ such that $C=\{a,..., a+ m-1\}$. Hence line (\ref{MartingaleProblemGenerator}) is equal to
	\begin{align*}
		&\sum_{k=0}^{m} \binom{m}{k} \theta(k, m-k) \sign(k-i+a-1).
	\end{align*}
	Where $\sign(0)$ is taken to be $1$ here. In particular this means that when $y_i$ is distinct from all other coordinates, the above equals $\theta(1, 0) - \theta(0, 1) = \beta =0$.
	\begin{align}\label{The MP Generator}
		\sum_{i=1}^n& \frac{\partial f}{\partial y_i} (y) \mathcal{A}^\theta_n F^{(i)}(x) \nonumber\\
		&=\sum_{C\in \pi(y)} \sum_{i \in C} \frac{\partial f}{\partial y_i}(y) \sum_{k=0}^{|C|} \binom{|C|}{k} \theta(k, |C|-k) \sign(k-i+\inf C-1).
	\end{align}
	Where each of the partial derivatives are evaluated at $y$. Putting (\ref{The MP Generator}) into (\ref{Mean of f(X)}) we can compute the limit
	\begin{align*}
		&\lim\limits_{t\to 0}\frac{1}{t} \left( \E_x \left[f(Y(t))\right] - f(x) \right)\\
		=& \lim\limits_{t\to 0}\frac{1}{2t}\int_0^t \E_x[ \Delta f(Y(s))]ds \\
		&+ \frac{1}{t}\int_0^t\E_x \bigg[\sum_{C\in \pi(y)} \sum_{i \in C} \frac{\partial f}{\partial y_i}(y) \sum_{k=0}^{|C|} \binom{|C|}{k} \theta(k, |C|-k) \sign(k-i+\inf C-1) \bigg] \\
		& \quad \quad + \frac{1}{2} \sum_{i\neq j} \E_x \left[ \frac{\partial^2 f}{\partial y_i \partial y_j} (Y(s)) \mathbbm{1}_{\{Y^i(s)=Y^j(s)\}}\right] ds.
	\end{align*}

	In particular if we have $f\in \mathcal{D}_\theta$ then the last two lines cancel, leaving $\frac{1}{2t}\int_0^t \E_x[\Delta f(Y(s))] ds$. Recalling that $F:\R^n \to \Weyl{n}$ is continuous and $Y(t)=F(X(t))$, we can use the Feller property of $X$. Since $\Delta f\in C_0(\Weyl{n})$, $\Delta f\circ F \in C_0(\R^n)$ (since $F(x) \to \infty$ as $|x| \to \infty$). Hence $\frac{1}{2t}\int_0^t \E_x[\Delta f(Y(s))] ds$ converges uniformly to $\frac{1}{2}\Delta f(x)$ as $t\to 0$ and thus for $f\in \mathcal{D}_\theta$
	\begin{align*}
		\lim\limits_{t\to 0}\frac{1}{t} \left( \E_x \left[f(Y(t))\right] - f(x) \right)= \frac{1}{2} \Delta f(x), \ \text{with respect to the uniform norm.}
	\end{align*}
	Hence if $f\in \mathcal{D}_\theta$ then it is in the domain of the generator of $Y$ and $\mathcal{G}_\theta f = \frac{1}{2}\Delta f$.
\end{proof}
We now use the above calculations to prove Proposition \ref{Prop:TheBackwardsEquation}.
\begin{proof}{Proof of Proposition \ref{Prop:TheBackwardsEquation}} 
	By applying It\^o's formula as we did in the preceding proof, we see that for any function $g$ satisfying the assumptions of the proposition, there is an adapted process $(M(u))_{u\in [0,t]}$ that is a continuous local martingale on $[0,s]$ for each $s<t$ such that,
	\begin{align*}
		g(t-s, Y(s)) &= -\int_0^s \Partial{g}{t} (t-u, Y(u)) du + \int_0^s \Delta g(t-u, Y(u)) du + M(s),\\
		&= M(s).
	\end{align*}
	Now we just need to show that $M(s)$ is a local martingale on $[0, t]$. Since $g(t, \cdot) \to f$ uniformly as $t\to 0$ we have
	\begin{align*}
		|M(s)| = | g(t-s, Y(s)) | &\leq \underbrace{\| g(t-s, \cdot) - f \|_{\infty}}_{\to 0 \text{ as } s\to t} + \|f\|_{\infty}.
	\end{align*}
	Thus there is an $\epsilon >0$ such that $M(s)$ is bounded on $[t-\epsilon, t]$, giving that $M(s)- M(t-\epsilon)$ is a martingale on $[t-\epsilon, t]$. Hence $M(s)$ is a local martingale on $[0, t]$. Clearly $M(0) = g(t, x)$, and $M(t)= f(Y(t))$ since
	\begin{align*}
		|M(s) - f(Y(t))| = |g(t-s, Y(s)) - f(Y(t))| \leq \|g(t-s, \cdot) - f \|_{\infty} + |f(Y(s))- f(Y(t))|.
	\end{align*}
	The first term vanishes as $s\to t$ due to the uniform convergence of $g$ to $f$, and the second almost surely due to the continuity of $f$ and $Y$.
\end{proof}
Hence we can find the transition probabilities of $Y$ by looking for the Green's function for (\ref{HeatEq}), providing solutions are sufficiently regular to make $g(t-s, Y(s))$ a true martingale. In general it's not clear that there should be solutions to (\ref{HeatEq}), it's not even clear whether $\mathcal{D}_\theta$ is non-trivial. In the rest of the paper we focus on the case of a uniform characteristic measure: $\nu = \frac{1}{2} \theta \mathbbm{1}_{[0,1]}dx$.  Since we know $\nu$ we can calculate the constants $\theta(k,l)$, by definition we have
\begin{align}
\theta(k,l)&= \frac{\theta}{2} \int_0^1 x^{k-1}(1-x)^{l-1} dx,\nonumber \\
&= \frac{\theta}{2} \frac{(l-1)!(k-1)!}{(k+l-1)!}. \label{ThetaFormula}
\end{align}
In this case we also have $\theta(k,0)= \theta(0,k)$ for all $k \in \N$. Hence, for the characteristic measure $\nu = \frac{1}{2} \theta \mathbbm{1}_{[0,1]}dx$, (\ref{Complicated BCs-1}) can be rewritten as
\begin{align}
	&\frac{1}{2}\sum_{\substack{a\leq i,j \leq b: \\ i\neq j}} \frac{\partial^2 f}{ \partial x_i \partial x_j} (x) = -\frac{\theta}{2}\sum_{i = a}^{b} \frac{\partial f}{\partial x_i}(x) a(b-a+1, i), \quad \text{whenever }x_a=x_b. \label{Complicated BCs}
\end{align}
Where the coefficients are defined
\begin{align}
		a(b-a+1, i-a+1) := \sum_{k=1}^{b-a} \frac{b-a+1}{k(b-a+1-k)} \sign(k-i+a-1). \label{BCs Coefficients}
\end{align}
In the following section this particular form of the constants $\theta(k,l)$ will allow us to replace the conditions in line (\ref{Complicated BCs}) with much simpler conditions. In particular find each of the second derivatives in terms of the first derivatives.
\begin{remark}
	If we try to derive the Kolmogorov Backwards equation for the original process $X$, we run into problems. Namely that the action of the generator of $X$ within the set of $C^2_0$ functions does not determine the process. We can see this by considering a pair of sticky Brownian motions with parameter $\theta>0$ $X^1, X^2$. We have by It\^o's formula for all $f\in C^2_0(\R^2)$
	\begin{align*}
		\E_x[f(X^1(t), X^2(t))] =& f(x_1, x_2) + \frac{1}{2} \int_0^t \E_x[\Delta f(X^1(s), X^2(s))]ds\\
		&+ \int_0^t \E_x[\mathbbm{1}_{\{X^1(s)= X^2(s)\}} \scndmxPartial{f}{x_1}{x_2}(X^1(s), X^2(s)) ] ds.
	\end{align*}
	So that $f$ is in the domain of the generator if $\scndmxPartial{f}{x_1}{x_2}(x_1, x_2) = 0$ whenever $x_1 =x_2$. But this does not depend on the parameter $\theta$, and thus the generator restricted to this set cannot determine the law of the sticky Brownian motions.
\end{remark}
\subsection{Rearranging the Boundary Conditions}

Henceforth we consider the case where the characteristic measure is uniform, i.e. $\nu(dx)= \frac{\theta}{2}\mathbbm{1}_{[0,1]}dx$. Let's first note that if we set $|C|=2$ in (\ref{Complicated BCs}) we see $f\in \mathcal{D}_\theta$ satisfies
\begin{equation*}
	\frac{\partial^2 f}{\partial x_a \partial x_{a+1}} = \theta \left( \frac{\partial f}{\partial x_{a+1}}-\frac{\partial f}{\partial x_{a}} \right), \quad \text{whenever } x_a=x_{a+1}.
\end{equation*}
We will show that we can replace the full boundary conditions with equivalent ones of the above form, that is
\begin{lemma}
	\begin{equation*}
	\mathcal{D}_\theta = \left\{ f\in C^2_0(\Weyl{n})| \ \forall \ 1\leq a < b \leq n, \ \text{if } x_a=x_b \text{ then}\ \frac{\theta}{b-a}\left( \Partial{f}{x_b}- \Partial{f}{x_a}\right) = \scndmxPartial{f}{x_a}{x_b} \right\}.		
	\end{equation*}
\end{lemma}
\begin{remark}
	Essentially we are solving for the second derivatives of functions in $\mathcal{D}_\theta$, given their first derivatives. Whilst this should be possible for any characteristic measure, our method relies on the special form of the parameters $\theta(k, l)$ in the case of the uniform characteristic measure.
\end{remark}
\begin{proof}
Note that because we are in the Weyl chamber, $x_a=x_b$ implies $x_a=x_{a+1}=\dots=x_b$. Thus the condition for $x_a=...=x_{b-1}$ must also hold when $x_a=...=x_{b}$ etc. We prove the original conditions (\ref{Complicated BCs}) are equivalent to the new conditions, using an inductive argument. That is we prove that the new condition for $x_a=x_{b}$ is equivalent to the old conditions, assuming the new conditions for $x_c= x_d$ are satisfied for all $a\leq c< d \leq b$ such that $d-c < b-a$. \par{}

Hence we assume that the boundary conditions (\ref{Complicated BCs}) for $x_c = x_d$ are satisfied for all $a\leq c< d\leq b$, and that for all $a\leq c< d\leq b$ with $d-c< b-a$
\begin{equation}\label{BCs:assumption}
		\frac{\partial^2 f}{\partial x_c \partial x_{d}} (x) = \frac{\theta}{d-c} \left( \frac{\partial f}{\partial x_{d}}(x)-\frac{\partial f}{\partial x_{c}}(x) \right), \quad \text{if } x_c=...=x_{d}.
\end{equation}
Without loss of generality we can relabel $(x_a,...,x_{b})$ as $(x_1,...,x_{m})$, for $m=b-a+1$. Then for $u\in \mathcal{D}_\theta$ we can rewrite the sum over mixed derivatives
\begin{equation*}
	\frac{1}{2}\sum_{i\neq j} \scndmxPartial{f}{x_i}{x_j} = \frac{1}{2}\sum_{\substack{i\neq j \\ i,j \neq m}}\scndmxPartial{f}{x_i}{x_j} + \sum_{k=2}^{m-1}\scndmxPartial{f}{x_k}{x_m} +\scndmxPartial{f}{x_1}{x_m}.
\end{equation*}

Using equations (\ref{Complicated BCs}) and (\ref{BCs:assumption}) we have the equality, when $x_1=...=x_{m}$,
\begin{align}
 \frac{\partial^2 f}{\partial x_1 \partial x_m} =& -\frac{\theta}{2}\sum_{j=1}^m \frac{\partial f}{\partial y_j}(y) \sum_{k=1}^{m-1} \frac{m}{k(m-k)} \sign(k-j) -\sum_{i<j} \frac{\theta}{j-i}\left( \Partial{f}{x_j} - \Partial{f}{x_i} \right)\nonumber \\
 & + \frac{\theta}{m-1}\left( \Partial{f}{x_m} - \Partial{f}{x_1} \right). \label{Boundary Condition Induction}
\end{align}
We have the following equalities
\begin{align*}
	\sum_{i<j} \frac{\theta}{j-i}\left( \Partial{f}{x_j} - \Partial{f}{x_i} \right) =& \sum_{j=2}^m \sum_{i=1}^{j-1} \frac{\theta}{j-i} \Partial{f}{x_j} - \sum_{j=2}^{m-1} \sum_{i=1}^{j-1} \frac{\theta}{j-i} \Partial{f}{x_i} \\
	=& \sum_{j=2}^m \sum_{i=1}^{j-1} \frac{\theta}{j-i} \Partial{f}{x_j} + \sum_{j=1}^{m-1} \sum_{i=j+1}^{m} \frac{\theta}{j-i} \Partial{f}{x_j}\\
	=& \theta \sum_{j=1}^m \Partial{f}{x_j} \sum_{i \neq j} \frac{1}{j-i}.
\end{align*}
So that we are finished if for each $j\in \{1,..., n\}$
\begin{align*}
	\frac{1}{2}\sum_{k=1}^{m-1} \frac{m}{k(m-k)} \sign(k-j) + \sum_{i \neq j} \frac{1}{j-i} = 0.
\end{align*}
Noting that we have $\frac{m}{k(m-k)} = \frac{1}{k} + \frac{1}{m-k}$, we get
\begin{align}
	\frac{1}{2}\sum_{k=1}^{m-1} \frac{m}{k(m-k)} \sign(k-j) =& \frac{1}{2} \sum_{k=j}^{m-j} \frac{m}{k(m-k)} \nonumber\\
	=& \frac{1}{2} \sum_{k=j}^{m-j} \left(\frac{1}{k} + \frac{1}{m-k} \right) \nonumber\\
	=& \sum_{k=j}^{m-j} \frac{1}{k}. \label{Equation: Complicated BCs Coefficients}
\end{align}
In addition
\begin{align*}
	\sum_{i\neq j} \frac{1}{j-i} =& \sum_{i=1}^{j-1} \frac{1}{j-i} - \sum_{i=j+1}^{m} \frac{1}{i-j} \\
	=& \sum_{k=1}^{j-1} \frac{1}{k} - \sum_{k=1}^{m-j} \frac{1}{k} = -\sum_{k=j}^{m-j}\frac{1}{k}.
\end{align*}
With the convention that, when $a<b$, $\sum_{k=b}^a c_k = - \sum_{k=a}^b c_k$. Putting this into line (\ref{Boundary Condition Induction}) we see
\begin{align*}
	\scndmxPartial{f}{x_1}{x_m} = \frac{\theta}{m-1}\left( \Partial{f}{x_m} - \Partial{f}{x_1} \right).
\end{align*}
As noted previously for $m=2$ both conditions are equivalent, so by induction the old conditions imply the new conditions. Finally it's easy to see that assuming the new conditions hold on $x_c= x_d$ for all $a\leq c<d\leq b$ and the old conditions on $x_c=x_d$ for all $a\leq c< d\leq b$ such that $d-c< b-a$, we can follow the above argument in reverse to prove the new conditions imply the old ones.  Hence the equivalence of the two sets of conditions is proven.
\end{proof}
As a consequence we can reframe proposition \ref{Prop:TheBackwardsEquation}. For $g\in C^2_0(\R_{> 0}\times \Weyl{n})$ satisfying the PDE
\begin{equation*}
	\begin{cases}
		\Partial{g}{t} = \frac{1}{2}\Delta g, \ \text{for } x\in \Weyl{n};\\
		\scndmxPartial{u}{x_a}{x_b}=\frac{\theta}{b-a} \left(\Partial{g}{x_b}- \Partial{g}{x_a}\right), \ \text{if } b>a \ \text{and } x_a=x_b.
	\end{cases}
\end{equation*}
with initial condition $g(t,\cdot)\to f$ uniformly as $t\to0$, where $f\in C_b(\Weyl{n})$, we have $g(t,x)= \E_x\left[f(Y(t))\right].$ This rearrangement will simplify the combinatorics required to show that we can solve the PDE with the Bethe ansatz.

\subsection{Invariant Measure}\label{Section:StationaryMeasure}
In this section we prove an integration by parts formula for the generator of the ordered $n$-point motion of the Howitt-Warren flow with uniform characteristic measure. First we introduce some useful notation. \par{}

Recall that each $\pi \in \Pi_n$ there is a natural bijection between $\Weylo{|\pi|}$ and $\delWeyl{n}{\pi}$, $I^{\pi}:\delWeyl{n}{\pi} \to\Weylo{|\pi|}$. For a function $u: \Weyl{n} \to \R$ denote by $u_\pi : \Weylo{|\pi|}\to \R$ the function defined by $u_{\pi}(x):=u\circ (I^\pi)^{-1}(x)$ for all $x \in \Weylo{|\pi|}$. For $u,v\in C^1(\Weyl{n})$ such that the below integrals converge we define
\begin{equation}\label{StickyInnerProduct}
	\left(u, v \right)_\theta := \sum_{\pi \in \Pi_n} \theta^{|\pi|-n} \left(\prod_{\pi_\iota \in \pi} \frac{1}{|\pi_\iota|} \right) \int_{\Weyl{|\pi|}} \nabla u_\pi \cdot \nabla v_\pi dx.
\end{equation}
Now we can state the integration by parts formula, recalling $m^{(n)}_\theta$ from definition \ref{ReferenceMeasureOfTheTdensity}
\begin{proposition}\label{Lemma:Stationary}
Suppose $u\in\mathcal{D}_\theta$ and $v\in C^1_b(\Weyl{n})$, such that there exists $a, c>0$ such that $|\nabla u(x)| \leq a e^{-c|x|}$. We have
\begin{equation}\label{Energy form}
	\int_{\Weyl{n}} \Delta u(x) v(x) m^{(n)}_\theta(dx)	=-\left(u , v\right)_\theta,
\end{equation}
whenever the above integrals are finite.
\end{proposition}
\begin{proof}
	 Since $u \in\mathcal{D}_\theta$ we can relate $\Delta u_\pi$ and $(\Delta u)_\pi$. Clearly we have
	\begin{align*}
		\Delta u_\pi = \sum_{\pi_\iota \in \pi} \sum_{j,k \in \pi_\iota} \left( \scndmxPartial{u}{x_j}{x_k} \right)_\pi. 
	\end{align*}
	Hence
	\begin{align*}
		\Delta u_\pi - (\Delta u)_\pi =& \sum_{\pi_\iota \in \pi} \sum_{\substack{j,k \in \pi_\iota \\ j\neq k}} \left( \scndmxPartial{u}{x_j}{x_k} \right)_\pi \\
		=& 2\sum_{\pi_\iota \in \pi} \sum_{\substack{j,k \in \pi_\iota \\ j< k}} \left( \scndmxPartial{u}{x_j}{x_k} \right)_\pi.
	\end{align*}
	Clearly the second sum is empty whenever $|\pi_\iota|=1$, so we can exclude those terms from the first sum. Using equations (\ref{Complicated BCs}), (\ref{BCs Coefficients}), and the notations $\underline{\pi_\iota}:= \inf{\pi_\iota}$, $\overline{\pi_\iota}:= \sup{\pi_\iota} = |\pi_\iota|+ \underline{\pi_\iota} -1$, this is equal to
	\begin{align*}
		& -\theta \sum_{\substack{\pi_\iota \in \pi:\\ |\pi_\iota|>1}} \sum_{j \in \pi_\iota} \left( \Partial{u}{x_j} \right)_\pi a(|\pi_\iota|, j- \underline{\pi_\iota}+ 1).
	\end{align*}
	Now we consider the left hand side of equation (\ref{Energy form}). Using Definition \ref{ReferenceMeasureOfTheTdensity}, this is equal to
	\begin{align*}
		\sum_{\pi \in \Pi_n}  \theta^{|\pi|-n} \left(\prod_{\pi_\iota \in \pi} \frac{1}{|\pi_\iota|} \right) \int_{\delWeyl{n}{\pi}} \Delta u(x) v(x) \lambda^{\pi}(dx)
	\end{align*}
	Rewriting each integral in the sum in terms of Lebesgue integrals on lower dimensional spaces, we find the above is equal to
	\begin{align}
		& \sum_{\pi \in \Pi_n} \theta^{|\pi|-n} \left(\prod_{\pi_\iota \in \pi} \frac{1}{|\pi_\iota|} \right) \int_{\Weylo{|\pi|}} (\Delta u)_\pi(x) v_\pi(x) dx \nonumber \\
		=& \sum_{\pi \in \Pi_n} \theta^{|\pi|-n} \left(\prod_{\pi_\iota \in \pi} \frac{1}{|\pi_\iota|} \right) \int_{\Weylo{|\pi|}} \bigg( \Delta u_\pi(x) + \theta \sum_{\substack{\pi_\iota \in \pi:\\ |\pi_\iota|>1}} \sum_{j \in \pi_\iota} \left( \Partial{u}{x_j} \right)_\pi a(|\pi_\iota|, j- \underline{\pi_\iota}+ 1) \bigg) v_\pi(x) dx. \label{Equation: Integration by Parts 2}
	\end{align}
	Since the Weyl chamber has a piecewise smooth boundary we can apply Green's identity to the first term in each integral. Applying it on $\Weyl{|\pi|}\cap \{x \in \Weyl{n}: \ |x|<R\}$ and then taking $R\to \infty$, the exponential bound on $|\nabla u|$ together with the boundedness of $v$ ensures the only boundary term to survive in the limit will be the integral over $\partial \Weyl{|\pi|}$. \par{}
	The smooth part of the boundary of the Weyl chamber $\Weylo{|\pi|}$ can be written in terms of the disjoint union of $\delWeyl{|\pi|}{\tilde{\pi}}$ over the set $M_\pi:= \{\tilde{\pi}\in \Pi_{|\pi|}: \ |\tilde{\pi}|=|\pi|-1\}$. Note that if $|\pi|=1$ this union is empty, and the boundary integral vanishes. Each $\tilde{\pi}$ in $M_\pi$ consists of $|\pi|-2$ singletons and one set $\{l, l+1\}$ for some $l\in\{1,..., |\pi|\}$. Further the outward unit normal on $\delWeyl{|\pi|}{\tilde{\pi}}$ is given by
	\begin{equation*}
		\underline{n}(x)_r = \begin{cases}
				-\frac{1}{\sqrt{2}}, \quad \text{if } r=l;\\
				\frac{1}{\sqrt{2}}, \quad \text{if } r=l+1;\\
				0, \quad \text{otherwise}.
		\end{cases}
	\end{equation*}
	Finally the boundary measure is given by $\sum_{\tilde{\pi}\in M_\pi} \sqrt{2} \lambda^{\tilde{\pi}}$, so that (\ref{Equation: Integration by Parts 2}) equals
	\begin{align*}
		& \sum_{\pi \in \Pi_n} \theta^{|\pi|-n} \left( \prod_{\pi_\iota \in \pi} \frac{1}{ |\pi_\iota|} \right) \left(\sum_{\tilde{\pi}\in M_\pi}\int_{\delWeyl{|\pi|}{\tilde{\pi}}} \left( \Partial{u_\pi}{y_{l+1}} - \Partial{u_\pi}{y_l} \right) \ v_\pi d\lambda^{\tilde{\pi}}\right) - \int_{\Weylo{|\pi|}} \nabla u_\pi(x) \cdot \nabla v_\pi(x) dx \\
		&+ \theta \sum_{\substack{\pi_\iota \in \pi:\\ |\pi_\iota|>1}} \sum_{j \in \pi_\iota} \int_{\Weylo{|\pi|}} \left( \Partial{u}{x_j} \right)_\pi a(|\pi|, j- \underline{\pi_\iota}+ 1) v_\pi(x) dx.
	\end{align*}
	Where $l$ is depends on $\tilde{\pi}$ and is defined as above. We have written the partial derivatives of $u_\pi$ with respect to $y$ to emphasise the fact that $u_\pi$ is a function on $\Weyl{|\pi|}$ rather than $\Weyl{n}$. Hence to complete the proof it is enough to show the first and third terms cancel. Rewriting the integrals with respect to $\lambda^{\tilde{\pi}}$, the first is equal to
	\begin{align*}
		&\sum_{\pi \in \Pi_n} \theta^{|\pi|-n} \left( \prod_{\pi_\iota \in \pi} \frac{1}{ |\pi_\iota|} \right) \sum_{\tilde{\pi}\in M_\pi}\int_{\Weylo{|\tilde{\pi}|}} \left( \Partial{u_\pi}{y_{l+1}} - \Partial{u_\pi}{y_l} \right)_{\tilde{\pi}} \ (v_\pi)_{\tilde{\pi}} d\lambda^{\tilde{\pi}}.
	\end{align*}
	Clearly this is equal to
	\begin{align*}
		&\sum_{\pi \in \Pi_n} \theta^{|\pi|-n} \left(\prod_{\pi_\iota \in \pi} \frac{1}{ |\pi_\iota|} \right) \sum_{\tilde{\pi}\in M_\pi}\int_{\Weylo{|\tilde{\pi}|}} \left( \left(\sum_{j \in \pi_{l+1}} \Partial{u}{x_j} \right)_\pi - \left( \sum_{j \in \pi_l}\Partial{u}{x_j} \right)_\pi \right)_{\tilde{\pi}}(x) \ (v_\pi)_{\tilde{\pi}}(x) dx.
	\end{align*}
	Which can be rewritten as
	\begin{align*}
		=&\sum_{\pi \in \Pi_n}\sum_{\tilde{\pi}\in M_\pi} \theta^{|\pi|-n} \left(\prod_{\pi_\iota \in \pi} \frac{1}{ |\pi_\iota|} \right) \int_{\Weylo{|\pi|-1}}\sum_{j \in \pi_{l+1}\cup \pi_l} \left( \left( \Partial{u}{x_j} \right)_\pi \right)_{\tilde{\pi}}(x) \sign(j -\underline{\pi_{l+1}}) \ (v_\pi)_{\tilde{\pi}}(x) dx.
	\end{align*}	
	Notice that, for each $\pi\in \Pi_n$ and $\tilde{\pi}\in M_\pi$, we can rewrite the summand in terms of a new partition, $\hat{\pi}$, formed from $\pi$ by merging two adjacent blocks to form the $\pi_{l+1}\cup \pi_l$ block. Further, because the partitions are in $\Pi_n$, there are exactly $|\pi_{l+1}\cup \pi_l|-1$ partitions that yield $\hat{\pi}$ by merging two blocks to form $\pi_{l+1}\cup \pi_l$. Rewriting the sum in terms of $\hat{\pi}$ we get
	\begin{align*}
		&=\sum_{\hat{\pi} \in \Pi_n} \theta^{|\hat{\pi}|+1-n} \left(\prod_{\hat{\pi}_\iota \in \hat{\pi}} \tfrac{1}{ |\hat{\pi}_\iota|} \right) \sum_{\substack{\hat{\pi}_\iota \in \hat{\pi}: \\ |\hat{\pi_\iota}|>1}}\int_{\Weylo{|\hat{\pi}|}} \sum_{k = 1}^{|\hat{\pi}_\iota|-1} \tfrac{|\hat{\pi}_\iota|}{k(|\hat{\pi}_\iota|- k)} \sum_{j\in \hat{\pi}_\iota} \left( \Partial{u}{x_j} \right)_{\hat{\pi}}(x) \sign(j -\underline{\hat{\pi}_\iota} - k) \ v_{\hat{\pi}}(x) dx.
	\end{align*}
	Here the sum over $j$ is over the partitions whose blocks have been merged to get $\hat{\pi}$, with $k$ corresponding to the size of the lower block. The extra factor $\frac{|\hat{\pi}_\iota|}{k(|\hat{\pi}_\iota|- k)}$ is simply a correction to the product to write it in terms of $\hat{\pi}$ rather than the $\pi$ partition whose blocks we merged. \par{}

	Recalling that $\sign(0)=1$ here, equation (\ref{BCs Coefficients}) yields that the above is precisely equal to
	\begin{align*}
		-\sum_{\pi \in \Pi_n}\theta^{|\pi|+1-n} \left(\prod_{\pi_\iota \in \pi} \frac{1}{ |\pi_\iota|} \right) \sum_{\substack{\pi_\iota \in \pi:\\ |\pi_\iota|>1}} \sum_{j \in \pi_\iota} \int_{\Weylo{|\pi|}} \left( \Partial{u}{x_j} \right)_\pi (x) a(|\pi|, j- \underline{\pi_\iota}+ 1) v_\pi(x) dx.
	\end{align*}
	Hence (\ref{Equation: Integration by Parts 2}) is equal to
	\begin{align*}
		&  -\sum_{\pi \in \Pi_n} \theta^{|\pi| -n} \prod_{\pi_\iota \in \pi} \frac{1}{ |\pi_\iota|} \int_{\Weylo{|\pi|}} \nabla u_\pi(x) \cdot \nabla v_\pi(x) dx \\
		=& - (u, v)_\theta.
	\end{align*}
\end{proof}
Thus,  if we denote by $L^2(m^{(n)}_\theta)$ the $L^2$ space on $\Weyl{n}$, with respect to the measure $m^{(n)}_\theta$, with the standard $L^2$ inner product. Then the generator is symmetric on $\mathcal{D}_\theta \cap L^2(m^{(n)}_\theta)$, suggesting the process is reversible with respect to this measure. But because our calculations are only done for $u\in \mathcal{D}_\theta$, and we do not know how rich the set $\mathcal{D}_\theta$ is, this is not enough for a proof. However taking $v=1$, the right hand side of (\ref{Energy form}) vanishes, giving us the following useful corollary.
\begin{corollary}\label{StationarityOfm}
	For $u\in \mathcal{D}_\theta$ such that there are $a, c>0$ with $|\nabla u(x)| \leq ae^{-c|x|}$ we have
	\begin{equation*}
		\frac{1}{2}\int \Delta u (x) m^{(n)}_\theta (dx) =0.
	\end{equation*}
\end{corollary}

In the next section we find the Green's function for the backwards equation, and thus the transition density for the process (with respect to the measure $m^{(n)}_\theta$). Using this we can prove that $m^{(n)}_\theta$ is the stationary measure, and that $Y$ is reversible with respect to $m^{(n)}_\theta$.

\section{Bethe Ansatz for Sticky Brownian Motions}\label{Bethe ansatz for Sticky Brownian motions}
Essentially we are trying to find a solution to the following PDE, for each fixed $y\in \Weylo{n}$ and $\theta$ some positive constant, with initial condition $u_0(x,y)= \delta(x-y)$, where $\delta$ is the Dirac delta distribution.
\begin{equation}\label{StickBackwardEquation}
	\begin{cases}
		\Partial{u_t}{t} = \frac{1}{2} \Delta u_t, \quad \text{for all } x\in \Weyl{n}; \\
		\theta\left( \frac{\partial u}{\partial x_b} - \frac{\partial u}{\partial x_a} \right) = (b-a)\frac{\partial^2 u}{\partial x_a \partial x_b},
		\quad \text{when } x_a=x_b, \ \text{for some } a<b.
	\end{cases}
\end{equation}

 The Bethe ansatz suggest that if we define
\begin{equation}\label{InversionFactor}
	S_{\alpha, \beta}(k) := \frac{i\theta\left( k_\beta - k_\alpha \right) +k_\alpha k_\beta}{i \theta\left( k_\beta - k_\alpha \right) -k_\alpha k_\beta }.
\end{equation}
Then the solution is given by the following equation,
\begin{equation}\label{BetheAnsatz}
	u_t(x,y) = \frac{1}{(2\pi)^n}\int_{\R^n} e^{-\frac{1}{2} t |k|^2} \sum_{\sigma\in S_n} e^{ik_\sigma \cdot (x - y_\sigma)} \prod_{\substack{\alpha <\beta : \\ \sigma(\beta) < \sigma(\alpha)}} S_{\sigma(\beta), \sigma(\alpha)}(k) dk,
\end{equation}
where $S_n$ denotes the group of permutations on $\{1,...,n\}$ and $k_\sigma = (k_{\sigma(1)},...,k_{\sigma(n)})$. \par 

The idea here is similar to that used to find the transition density of a reflected Brownian motion. Since we are considering a process with ordered coordinates we combine solutions to the interior equation with permuted coordinates, the permutations representing possible orderings of the original process. The more complicated boundary conditions require us to combine our solutions in a more complicated way, in particular we take linear combinations in Fourier space, in such a way that the boundary conditions where $b-a=1$ are satisfied. This is how we find the form of (\ref{InversionFactor}). In fact it forces this ansatz onto us, leaving no freedom to deal with the additional conditions which correspond to $b-a>1$ in (\ref{StickBackwardEquation}). \par

In fact Barraquand and Rychnovsky conjectured in \cite{barraquand2019large} that the Backwards equation for the system of sticky Brownian motions was the heat equation with the boundary conditions corresponding to $b-a=1$ in (\ref{StickBackwardEquation}), by looking at the Bethe ansatz answer for the system. It's important to note that for any other choice of characteristic measure $\nu$ with $\nu([0,1]) = \frac{\theta}{2}$ the boundary conditions corresponding to $b-a=1$ would be the same, so we do not expect these boundary conditions alone to give uniqueness of the PDE. However in order to rewrite the PDE as (\ref{StickBackwardEquation}), we assume the solution to be $C^2$ in space. It is possible the $b-a=1$ boundary conditions do determine the solution under this additional regularity assumption and the transition densities for all of the other systems of sticky Brownian motions are not $C^2$ in space. \par{}

It's clear that (\ref{BetheAnsatz}) satisfies the first condition in (\ref{StickBackwardEquation}), and our choice of (\ref{InversionFactor}) guarantees the second condition holds when $b-a=1$. However when $b-a>1$ it is not clear that they are still satisfied. Fortunately, and surprisingly, the second condition turns out to be satisfied in its entirety. We can also show the initial condition holds, hence we obtain our main result which we restate here:
\begin{theorem}\label{transition probabilities}
	Suppose $\theta>0$, and $X=(X(t))_{t\geq 0}$ is a solution to the Howitt-Warren martingale problem in $\R^n$ with characteristic measure $\frac{\theta}{2} \mathbbm{1}_{[0,1]}dx$ and zero drift. Let $Y=(Y(t))_{t\geq 0}$ be the process obtained by ordering the coordinates of $(X(t))_{t\geq 0}$. Then for every bounded and Lipschitz continuous function $f:\Weyl{n}\to \R$, $x \in \Weyl{n}$ and $t>0$
	\begin{equation*}
		\E_x[f(Y_t)]= \int u_t(x,y) f(y) m^{(n)}_\theta(dy).
	\end{equation*}
	Where $u$ is as in (\ref{BetheAnsatz}), $m^{(n)}_\theta$ is defined in definition \ref{ReferenceMeasureOfTheTdensity}.
\end{theorem}

In the following section we shall prove Theorem (\ref{transition probabilities}), first we show the boundary conditions are satisfied and then the initial condition. To ensure we can perform the necessary exchanges of integral and derivative we start with some bounds for the Bethe ansatz.
\subsection{Bounds for Dominated Convergence}
\begin{lemma}\label{uIsIntegrable}
	For every $x\in \Weyl{n}$ and $t>0$ we have $u_t(x, \cdot) \in L^1(m^{(n)}_\theta)$, where $u_t(x, \cdot)$ is defined as in (\ref{BetheAnsatz}). Further, for each $x\in \Weyl{n}$ and $t>0$, there exist $a, c>0$ such that $|\nabla_y u_t(x, y)| \leq a e^{-c|y|}$ for all $y\in \Weyl{n}$. The same statement holds if we instead consider the $x$ derivative and vary $x$ with $y$ being fixed. Similarly for each $x \in \Weyl{n}$ and $s>0$ we can find $a, c>0$ such that $|u_t(x,y)|, \ |\partial_t u_t(x, y) | \leq a e^{-c|y|}$ for all $t>s$ and $y\in \Weyl{n}$.
\end{lemma}
\begin{proof}
	The proof for this lemma is a simplified version of the methods we apply in subsection \ref{Section: Initial Condition}, as such we omit the main details to avoid repetition and instead sketch the proof. Following the arguments used to prove Proposition \ref{Prop: k integral bound}, with $\pi=\{\{1\},\{2\},...,\{n\}\}$, we can derive a Gaussian bound on the summand in (\ref{BetheAnsatz}). We can then adapt the arguments in Lemma \ref{Lemma: shifted k integral bound} to bound the resulting contour integrals, which will have additional factors of $k$ due to the derivatives. In fact the proof can be greatly simplified in this case as we do not need to consider the $t\to0$ limit, and therefore we don't need to ensure we get the optimal exponent for $t$. The above arguments give us a bound in the form of a finite sum of Gaussian kernels, multiplied by a negative power of $t$, from which the above bounds follow easily (note that for the bound on the $x$ derivatives we can simply apply the bound on the $y$ derivatives, as $u_t(x,y)=u_t(y,x)$ which we prove later in Lemma \ref{uisSymmetric}).
\end{proof}

The second part of the above lemma provides the necessary bounds to justify passing derivatives through the first integral in $\int u_t(x,y) f(y) m^{(n)}_\theta(dy)$. Further it is easy to see we can apply Dominated convergence to find 
\begin{align*}
	&\Partial{u_t}{x_a} = \frac{1}{(2\pi)^n} \int_{\R^n} e^{-\frac{1}{2}t|k|^2} \sum_{\sigma\in S_n} i k_{\sigma(a)} e^{ik_\sigma \cdot (x-y_\sigma)} \prod_{\substack{\alpha< \beta: \\ \sigma(\beta) < \sigma(\alpha)}} S_{\sigma(\beta), \sigma(\alpha)} (k) dk, \\
	&\scndmxPartial{u_t}{x_a}{x_b} = -\frac{1}{(2\pi)^n} \int_{\R^n} e^{-\frac{1}{2}t|k|^2} \sum_{\sigma\in S_n} k_{\sigma(a)} k_{\sigma(b)} e^{ik_\sigma \cdot (x-y_\sigma)} \prod_{\substack{\alpha< \beta: \\ \sigma(\beta) < \sigma(\alpha)}} S_{\sigma(\beta), \sigma(\alpha)} (k) dk.
\end{align*}
This allows us to not only confirm that $\int u_t(x,y) f(y) m^{(n)}_\theta(dy)$ solves the heat equation but also to reduce the boundary conditions to a combinatorial problem.

\subsection{Boundary Conditions}\label{Section:BoundaryConditions}
\begin{proposition}\label{Proposition: u satisfies boundary conditions}
	\begin{equation*}
		\int u_t(x,y) f(y) m^{(n)}_\theta (dy) \in \mathcal{D}_\theta.
	\end{equation*}
\end{proposition}
Using the same ideas as in the previous subsection we can derive sufficient bounds to show $\int u_t(x,y) f(y) m^{(n)}_\theta (dy) \in C^2_0(\Weyl{n})$. Hence we just need to show it satisfies the correct boundary conditions, from the PDE (\ref{StickBackwardEquation}). Fix $a,b\in\{1,...,n\}$ with $a<b$, then for $t>0$ we can differentiate under the integral, as noted in the previous subsection, to see that the corresponding boundary condition is satisfied if for all $a<b$, $x_a=x_b$ implies
\begin{equation*}
	\int_{\R^n} e^{-\frac{1}{2} t |k|^2} \sum_{\sigma\in S_n} \left( i\theta (k_{\sigma(b)} - k_{\sigma(a)}) + (b-a)k_{\sigma(b)} k_{\sigma(a)} \right) e^{ik_\sigma \cdot (x - y_\sigma)} \prod_{\substack{\alpha <\beta : \\ \sigma(\beta) < \sigma(\alpha)}} S_{\sigma(\beta), \sigma(\alpha)}(k) dk =0.
\end{equation*}
This can be simplified by splitting the summand into parts dependent on $\sigma(a),...,\sigma(b)$ and on the remaining values $\sigma$ takes. Noting that we have $x_a=...=x_b$
\begin{equation*}
\prod_{c=a}^b e^{ik_{\sigma(c)}(x_c - y_{\sigma(c)})} = \prod_{c=a}^b e^{ i k_{\sigma(c)}(x_a- y_{\sigma(c)})} = \prod_{\tilde{c}\in \{\sigma(a),...,\sigma(b) \}}e^{ ik_{\tilde c}(x_a - y_{\tilde c})}.
\end{equation*}
Notice that $\{\sigma(a),...,\sigma(b)\}= \{\sigma(1),...,\sigma(a-1),\sigma(b+1),...,\sigma(n)\}^c$, and thus the exponential factor of the summand only depends on $\sigma(\{a,...,b\})$ and not $\sigma(a),...,\sigma(b)$ themselves. Now we split the product
\begin{align*}
	\prod_{\substack{\alpha <\beta : \\ \sigma(\beta) < \sigma(\alpha)}} S_{\sigma(\beta), \sigma(\alpha)}(k) =& 	\prod_{\substack{\alpha<a\leq \beta \leq b : \\ \sigma(\beta) < \sigma(\alpha)}} S_{\sigma(\beta), \sigma(\alpha)}(k)
	\prod_{\substack{a\leq \alpha \leq b <\beta : \\ \sigma(\beta) < \sigma(\alpha)}} S_{\sigma(\beta), \sigma(\alpha)}(k) \\
	&\prod_{\substack{\alpha, \beta \in \{a,...,b\}^c :\\ \alpha<\beta, \\ \sigma(\beta) < \sigma(\alpha)}} S_{\sigma(\beta), \sigma(\alpha)}(k) 
	\prod_{\substack{a\leq\alpha <\beta\leq b : \\ \sigma(\beta) < \sigma(\alpha)}} S_{\sigma(\beta), \sigma(\alpha)}(k).
\end{align*}
Note that $S_{\sigma(\beta),\sigma(\alpha)}$ does not depend on $\alpha$ and $\beta$, but on $\sigma(\alpha)$ and $\sigma(\beta)$. Suppose, for a given permutation $\sigma$,  $S_{\sigma(\beta), \sigma(\alpha)}$ appears in the first product. Then for any permutation $\tau$ with $\sigma(c)=\tau(c)$ for every $c\in \{a,...,b\}^c$ we have $\sigma(\beta)\in \{\sigma(a),...,\sigma(b)\}= \{\tau(a),...,\tau(b)\}$. Thus there exists $\gamma\in\{a,...,b\}$ such that $\tau(\gamma)= \sigma(\beta)$, and so we have $\tau(\alpha)=\sigma(\alpha)>\sigma(\beta)=\sigma(\gamma)$ and $\alpha<a\leq\gamma$. Hence $S_{\tau(\gamma), \tau(\alpha)}=S_{\sigma(\beta),\sigma(\alpha)}$ appears in the product for $\tau$. This shows the first product doesn't depend on $\{\sigma(a),...,\sigma(b)\}$, and similarly the second doesn't either. The third product clearly doesn't depend on them, leaving only the fourth product. Finally we note that the fourth product doesn't depend on the values $\sigma$ takes outside $\{a,...,b\}$. Hence we can split the sum into a sum over possibilities for the permutation outside $\{a,..., b\}$ and then a sum over possibilities inside $\{a,.., b\}$. Pulling the parts depending only on the values of $\sigma $ outside $\{a,...,b\}$ out of the second sum we see that the second sum will always vanish, and thus our condition will hold, if
\begin{equation*}
	\sum_{\sigma\in S_{b-a+1}} \left( i\theta (k_{\sigma(b-a+1)} - k_{\sigma(1)}) + (b-a)k_{\sigma(b-a+1)} k_{\sigma(1)} \right) \prod_{\substack{1\leq \alpha<\beta \leq b-a+1: \\ \sigma(\beta) <\sigma(\alpha)}} S_{\sigma(\beta), \sigma(\alpha)}(k) =0.
\end{equation*}
Where we have relabelled $k_a,..., k_b$ to $k_1,..., k_{b-a+1}$. Hence it is enough to prove the following
\begin{proposition}\label{Prop: BoundaryConditionEquation1}
	For every $n\in \N$ we have the identity
	\begin{equation}\label{BoundaryConditionEquation1}
		\sum_{\sigma \in S_n} \left( i\theta \left( k_{\sigma(n)} - k_{\sigma(1)} \right) + (n-1) k_{\sigma(n)} k_{\sigma(1)} \right) \prod_{\substack{\alpha <\beta : \\ \sigma(\beta) < \sigma(\alpha)}} S_{\sigma(\beta), \sigma(\alpha)} = 0, \quad \text{for every } n\in \N.
	\end{equation}
\end{proposition}
First we simplify the left hand side by pulling out the common denominator. Recalling (\ref{InversionFactor})
\begin{align*}
	&\prod_{\sigma(\beta)<\sigma(\alpha)} \left( i\theta(k_{\sigma(\alpha)}-k_{\sigma(\beta)}) -k_{\sigma(\beta)}k_{\sigma(\alpha)} \right) \prod_{\substack{\alpha <\beta : \\ \sigma(\beta) < \sigma(\alpha)}} S_{\sigma(\beta), \sigma(\alpha)} \\
	= &\prod_{\sigma(\beta)<\sigma(\alpha)} \left( i\theta(k_{\sigma(\alpha)}-k_{\sigma(\beta)}) -k_{\sigma(\beta)}k_{\sigma(\alpha)} \right) \prod_{\substack{\alpha <\beta : \\ \sigma(\beta) < \sigma(\alpha)}} \frac{i \theta (k_{\sigma(\alpha)} - k_{\sigma(\beta)}) + k_{\sigma(\alpha)} k_{\sigma(\beta)}}{i \theta (k_{\sigma(\alpha)} - k_{\sigma(\beta)}) - k_{\sigma(\alpha)} k_{\sigma(\beta)}} \\
	= &\prod_{\substack{\beta<\alpha : \\ \sigma(\beta) < \sigma(\alpha)}} \left( i \theta (k_{\sigma(\beta)} - k_{\sigma(\alpha)}) - k_{\sigma(\alpha)} k_{\sigma(\beta)} \right) \prod_{\substack{\alpha <\beta : \\ \sigma(\beta) < \sigma(\alpha)}} \left(i \theta (k_{\sigma(\alpha)} - k_{\sigma(\beta)}) + k_{\sigma(\alpha)} k_{\sigma(\beta)}\right).
\end{align*}
Since permutations are bijections, this denominator doesn't depend on $\sigma$. Thus multiplying both sides of (\ref{BoundaryConditionEquation1}) by it gives the equivalent equation
\begin{align}
	\sum_{\sigma \in S_n} &\left( i\theta \left( k_{\sigma(n)} - k_{\sigma(1)} \right) + (n-1) k_{\sigma(n)} k_{\sigma(1)} \right) \nonumber \\ & \prod\limits_{\substack{\beta <\alpha : \\ \sigma(\beta) < \sigma(\alpha)}} \left(i \theta (k_{\sigma(\alpha)} - k_{\sigma(\beta)}) - k_{\sigma(\alpha)} k_{\sigma(\beta)}\right) \prod\limits_{\substack{\alpha <\beta : \\ \sigma(\beta) < \sigma(\alpha)}} \left(i \theta (k_{\sigma(\alpha)} - k_{\sigma(\beta)}) + k_{\sigma(\alpha)} k_{\sigma(\beta)}\right) =0. \nonumber
\end{align}
The next simplification we can make is to notice that since $\theta>0$ the transformations $k_j \to i\theta k_j$ are invertible, and so the above equation is equivalent to
\begin{align}
	\sum_{\sigma \in S_n} &\left( \left( k_{\sigma(n)} - k_{\sigma(1)} \right) + (n-1) k_{\sigma(n)} k_{\sigma(1)} \right) \nonumber \\ &\prod\limits_{\substack{\alpha <\beta : \\ \sigma(\alpha) < \sigma(\beta)}} \left( (k_{\sigma(\beta)} - k_{\sigma(\alpha)}) - k_{\sigma(\alpha)} k_{\sigma(\beta)}\right) \prod\limits_{\substack{\alpha <\beta : \\ \sigma(\beta) < \sigma(\alpha)}} \left( (k_{\sigma(\alpha)} - k_{\sigma(\beta)}) + k_{\sigma(\alpha)} k_{\sigma(\beta)}\right) =0. \nonumber
\end{align}
Where we have cancelled off $(i\theta)^{2\left(\binom{n}{2} +1\right)}$. We'll now split the equation into two parts and simplify before showing they cancel. Making the following rearrangements, and defining the polynomial $B$
\begin{align}
	& \quad\prod\limits_{\substack{ \beta< \alpha : \\  \sigma(\beta) < \sigma(\alpha)}} \left( ( k_{\sigma(\alpha)} - k_{\sigma(\beta)}) - k_{\sigma(\alpha)} k_{\sigma(\beta)}\right) \prod\limits_{\substack{\alpha <\beta : \\ \sigma(\beta) < \sigma(\alpha)}} \left( (k_{\sigma(\alpha)} - k_{\sigma(\beta)}) + k_{\sigma(\alpha)} k_{\sigma(\beta)}\right). \label{Bsigma} \\
	&= \prod_{\alpha < \beta}  \sign(\sigma(\beta)-\sigma(\alpha))\left( k_{\sigma(\beta)} - k_{\sigma(\alpha)} - k_{\sigma(\alpha)}  k_{\sigma(\beta)} \right)\nonumber\\ 
	&= \sign(\sigma)\prod_{\alpha < \beta}\left(k_{\sigma(\beta)}-k_{\sigma(\alpha)} - k_{\sigma(\alpha)}k_{\sigma(\beta)}\right) =: \sign(\sigma) B(k_\sigma). \nonumber
\end{align}
We proceed by considering the expressions
\begin{gather}
	\sum_{\sigma \in S_n} \sign(\sigma) (n-1)k_{\sigma(n)} k_{\sigma(1)} B(k_\sigma); \label{BoundaryConditionDiff} \\
	\sum_{\sigma \in S_n} \sign(\sigma) \left( k_{\sigma(n)} - k_{\sigma(1)} \right) B(k_\sigma). \label{BoundaryConditionDrift}
\end{gather}
It's clear that both (\ref{BoundaryConditionDiff}) and (\ref{BoundaryConditionDrift}) are polynomials in the $k_j$, we will now make some more general statements about polynomials of this form. \\
It's clear that if $f:\R^n\to \R$ is a polynomial, then
\begin{equation}\label{polynomial9}
	\sum_{\sigma \in S_n} \sign(\sigma) f(k_\sigma) B(k_\sigma)
\end{equation}
is an alternating polynomial. To see this suppose $a<b$ and we exchange $k_a$ and $k_b$ in the above expression. Then $k_\sigma$ becomes $k_{(a,b)\circ\sigma}$ giving
\begin{align*}
	\sum_{\sigma \in S_n} \sign(\sigma) f(k_{(a,b)\circ\sigma}) B(k_{(a,b)\circ\sigma})&=-\sum_{\sigma \in S_n} \sign((a,b)\circ\sigma) f(k_{(a,b)\circ\sigma}) B(k_{(a,b)\circ\sigma})\\
	&=-\sum_{\sigma\in S_n} \sign(\sigma) f(k_\sigma) B(k_\sigma).
\end{align*}
In particular whenever we have $k_{\alpha} = k_{\beta}$, for $\alpha \neq \beta$, any such polynomial must vanish. Hence we must be able to take out as a factor the Vandermonde determinant, $\prod_{\alpha < \beta}(k_\beta - k_\alpha)$, since this is itself alternating whatever remains must be symmetric. Thus for any polynomial $f:\R^n\to \R^n$ there exists a symmetric polynomial $g:\R^n\to \R$ such that
\begin{equation}\label{VandermondeFactor}
\sum_{\sigma\in S_n}\sign(\sigma)f(k_\sigma) B(k_\sigma) = g(k)\prod_{\alpha < \beta} (k_\beta - k_\alpha).
\end{equation}
In the case of (\ref{BoundaryConditionDiff}) and (\ref{BoundaryConditionDrift}) the polynomial $f$ is also multilinear (no variable appears with exponent higher than one), and depends only on two variables. The following lemma will allow us to make further statements about $g$ based on these assumptions.
\begin{lemma}\label{BsubDegree}
	If $i,j\in\{2,...,n-1\}$ with $i\neq j$, and $\kappa \in \R^n$ such that we fix  $\kappa_i=-1$ and $\kappa_j=1$. Then $B(\kappa)$ has degree at most $n-2$ when considered as a polynomial of $\kappa_1$ or $\kappa_n$.
\end{lemma}
\begin{proof}
	Recalling the formula for $B(k)$, (\ref{Bsigma}), we have
	\begin{align*}
		B(\kappa) = \prod_{\substack{\alpha < \beta:\\ \alpha,\beta\neq i,j}} \left( \kappa_\beta - \kappa_\alpha - \kappa_\alpha \kappa_\beta \right) &\prod_{\alpha\neq i,j} \left( \sign(j-\alpha)(1 - \kappa_\alpha) - \kappa_\alpha \right) \\
		\times&\prod_{\alpha\neq i,j} \left( \sign(i-\alpha)(-1-\kappa_\alpha) + \kappa_\alpha \right) \ \left( 2\sign(j-i) +1 \right).
	\end{align*}
	The first product contains $(n-3)$ factors with $\kappa_1$ and $\kappa_n$ each. The second and third contribute the factor of the form:
	\begin{equation*}
		(1-2\kappa_1)(-1).
	\end{equation*}
	For $\kappa_1$, and
	\begin{equation}
		(-1)(2\kappa_n+1).
	\end{equation}
	For $\kappa_n$. Leaving a total of $n-2$ factors involving $\kappa_1$ and $\kappa_n$ each, which proves the statement.
\end{proof}
Now we can apply the above lemma to the expressions we are interested in.
\begin{lemma}\label{multilinearpolys}
	If $f:\R^2\to \R$ is a multilinear polynomial, then there exists constants $C_0,C_1$ and $C_2$ such that
	\begin{align*}
	\sum_{\sigma\in S_n}& \sign(\sigma)f(k_{\sigma(1)},k_{\sigma(n)}) B(k_\sigma) = \\
	&\left(C_0 + \sum_{m=1}^{\lfloor n/2\rfloor} \left( C_1 \sum_{\alpha_1<...<\alpha_{2m}} k_{\alpha_1}...k_{\alpha_{2m}} +C_2\sum_{\alpha_1<...<\alpha_{2m+1}}k_{\alpha_1}... k_{\alpha_{2m+1}} \right) \right)\prod_{\alpha < \beta} (k_\beta - k_\alpha).
	\end{align*}
\end{lemma}
\begin{proof}
	The discussion preceding Lemma \ref{BsubDegree} shows that we at least have equation (\ref{VandermondeFactor}), and that $g$ must be symmetric. To get the form given in the statement we will show that $g$ is also multilinear, which tells us we can write it as a linear combination of elementary symmetric polynomials, and then that the coefficients in this combination are of the form given above. Both of these arguments proceed by considering the exponents of the variables $k_j$. \par{}
	To show multilinearity we note that for each $k_j$, $\prod_{\alpha<\beta}(k_\beta - k_\alpha)$ contains $n-1$ linear factors of $k_j$. Furthermore each $B(k_\sigma)$ also contains exactly $n-1$ linear factors of $k_j$. But $f$ is multilinear so in the summand $\sign(\sigma)f(k_\sigma)B(k_\sigma)$ the largest possible power of $k_j$ is $n$. Hence the largest possible power of $k_j$ in $g(k)$ is $1$. This holds for each $j$ so $g(k)$ is multilinear. \par{}
	Since $g(k)$ is multilinear and symmetric it must be of the form
	\begin{equation*}
		g(k)=C_0 + \sum_{m=1}^n C_m \sum_{\alpha_1<...<\alpha_m} k_{\alpha_1}... k_{\alpha_m}.
	\end{equation*}
	Now we show that the constants $C_m$ satisfy $C_1 = C_{2m+1}$ and $C_2= C_{2m}$ for all $m\leq n/2$. Setting $\kappa = (k_1,...,k_{n-2},-1,1)$, we have the equality
	\begin{equation}\label{kappaEquality}
		\sum_{\sigma\in S_n} \sign(\sigma)f(\kappa_{\sigma(1)},\kappa_{\sigma(n)}) B (\kappa_\sigma) = 2 g(\kappa) \prod_{\alpha < \beta<n-1} (k_\beta - k_\alpha) \prod_{\gamma=1}^{n-2}(1-k_\gamma) (-1-k_\gamma).
	\end{equation}
	Since $g$ is symmetric polynomial, if one of its terms contains $k_{n-1}$ but not $k_n$, there is a term otherwise equal, where $k_{n-1}$ is replaced with $k_n$, and vice versa. Using $\kappa$ as defined in the previous proof, in $g(\kappa)$ these terms cancel, leaving only the terms that contain both or neither. For $\kappa$ we have set $k_{n-1}k_n=-1$ so we have the following.
	\begin{equation*}
		g(\kappa) = C_0 + \sum_{m=1}^{n-2} (C_m- C_{m+2}) \sum_{\alpha_1<...<\alpha_m<n-1} k_{\alpha_1}...k_{\alpha_m}.
	\end{equation*}
	The next step is to consider the exponents on the left hand side of (\ref{kappaEquality}) for each term in the sum, and show $g(\kappa)$ must be constant. First $B(k_\sigma)$ contains $(n-1)$ linear factors of each $k_j$, so the only way a $k_j$ with exponent $n$ can appear is if it also occurs in $f(\kappa_{\sigma(1)}, \kappa_{\sigma(n)})$, hence only if $j=\sigma(n)$ or $\sigma(1)$. But the previous lemma tells us that $B(\kappa_\sigma)$ has degree $n-2$ as a polynomial of $\kappa_{\sigma(1)}$ or $\kappa_{\sigma(n)}$. Thus the highest possible power of any of the $k_j$ on the left hand side of (\ref{kappaEquality}) is $n-1$. However the right hand side still contains $n-1$ linear factors of each $k_j$ outside of $g(\kappa)$, so $g(\kappa)$ must be constant. Hence for every $m>0$ $C_m = C_{m+2}$, proving the result.
\end{proof}
\begin{remark}\label{DeterminantOfBCs}
	Using the general formula for the sum of elementary symmetric polynomials on $n$ variables, $\prod_{j=1}^n(1+x_j)$, together with the above lemma, gives us that for a multilinear polynomial $f:\R^2\to \R$, there are constants $C_m$ and $D_m$ such that 
	\begin{align*}
		&\sum_{\sigma\in S_n} \sign(\sigma)f(k_{\sigma(1)},k_{\sigma(n)}) B(k_\sigma)\\
		 =& \prod_{\alpha<\beta} (k_\beta - k_\alpha) \Bigg(C_0 + \frac{1}{2}C_1\left(\prod_{j=1}^{n} (1+k_j) + \prod_{j=1}^n(1-k_j) -2 \right)\\
		 & \qquad \qquad \qquad+ \frac{1}{2}C_2\left(\prod_{j=1}^n(1+k_j) - \prod_{j=1}^n(1-k_j) \right) \Bigg) \\
		=& \prod_{\alpha<\beta} (k_\beta - k_\alpha) \left( D_0 + D_1\prod_{j=1}^n(1+k_j) +D_2\prod_{j=1}^n(1-k_j) \right)\\
		=&\det\left( k_i^{j-1} \right)\left( D_0+D_1\det\left((1+k_j)\delta_{ij}\right)+D_2\det\left((1-k_j)\delta_{ij}\right) \right).
	\end{align*} 
\end{remark}
Now we can return to our original expressions (\ref{BoundaryConditionDiff}) and (\ref{BoundaryConditionDrift}). These two lemmas imply that we have constants $C^{(n)}_0,\tilde{C}^{(n)}_0,C_1, \tilde{C}^{(n)}_1, C^{(n)}_2$ and $\tilde{C}^{(n)}_2$ such that
\begin{align}
	&\sum_{\sigma\in S_n} \sign(\sigma)(k_{\sigma(n)} - k_{\sigma(1)}) B(k_\sigma) = \label{bcDriftg}\\
	& \prod_{\alpha < \beta}(k_\beta - k_\alpha)\left(C^{(n)}_0 + \sum_{m=1}^{\lfloor n/2\rfloor} \left( C^{(n)}_1 \sum_{\alpha_1<...<\alpha_{2m}} k_{\alpha_1}...k_{\alpha_{2m}} +C^{(n)}_2\sum_{\alpha_1<...<\alpha_{2m+1}}k_{\alpha_1}... k_{\alpha_{2m+1}} \right)\right). \nonumber
\end{align}
\begin{align}
	&\sum_{\sigma\in S_n}\sign(\sigma) (n-1)k_{\sigma(n)}k_{\sigma(1)} B(k_\sigma)=  \label{bcDiffg}\\
	&\prod_{\alpha < \beta}(k_\beta - k_\alpha)\left(\tilde{C}^{(n)}_0 + \sum_{m=1}^{\lfloor n/2\rfloor} \left( \tilde{C}^{(n)}_1 \sum_{\alpha_1<...<\alpha_{2m}} k_{\alpha_1}...k_{\alpha_{2m}} +\tilde{C}^{(n)}_2\sum_{\alpha_1<...<\alpha_{2m+1}}k_{\alpha_1}... k_{\alpha_{2m+1}} \right)\right). \nonumber
\end{align}
The next lemma provides a link between these constants for different values of $n$, that will allow us to find their value inductively.
\begin{lemma}
 	For $m=0,1,2$ we have that $C_m^{(n)} = (n-1)C_m^{(n-1)}$ and $\tilde{C}_m^{(n)} = (n-1)\tilde{C}^{(n-1)}_m$.
\end{lemma}
\begin{proof}
	Take $k_n=0$ in (\ref{bcDriftg}) we get the equality
	\begin{align*}
		&\prod_{\alpha=1}^{n-1} (-k_{\alpha}) \prod_{\alpha<\beta<n} (k_\beta - k_\alpha) \\
		&\times  \left(C^{(n)}_0 + \sum_{m=1}^{\lfloor n/2\rfloor} \bigg( C^{(n)}_1 \sum_{\alpha_1<...<\alpha_{2m}<n} k_{\alpha_1}...k_{\alpha_{2m}} +C^{(n)}_2\sum_{\alpha_1<...<\alpha_{2m+1}<n}k_{\alpha_1}... k_{\alpha_{2m+1}} \bigg) \right) \\
		=&\sum_{\substack{\sigma\in S_n:\\ \sigma(1),\sigma(n)\neq n }} (k_{\sigma(n)} - k_{\sigma(1)}) \prod_{\alpha=1}^{n-1}(-k_{\alpha}) \prod_{\alpha<\beta<n} \left(k_\beta - k_\alpha - \sign\left( \sigma^{-1}(\beta) - \sigma^{-1}(\alpha) \right) k_\beta k_\alpha \right) \\
		&+ \sum_{\substack{\sigma\in S_n:\\ \sigma(1) = n }} k_{\sigma(n)} \prod_{\alpha=1}^{n-1}(-k_{\alpha}) \prod_{\alpha<\beta<n} \left(k_\beta - k_\alpha - \sign\left( \sigma^{-1}(\beta) - \sigma^{-1}(\alpha) \right) k_\beta k_\alpha \right) \\
		&- \sum_{\substack{\sigma\in S_n:\\ \sigma(n) = n }} k_{\sigma(1)} \prod_{\alpha=1}^{n-1}(-k_{\alpha}) \prod_{\alpha<\beta<n} \left(k_\beta - k_\alpha - \sign\left( \sigma^{-1}(\beta) - \sigma^{-1}(\alpha) \right) k_\beta k_\alpha \right).
 	\end{align*}
 	Here we have rewritten $\sign(\sigma)B(k_\sigma)$ as (\ref{Bsigma}) once again. Note that $\sigma^{-1}(n)$ plays no role in the terms of the first sum (on the second line). Thus we can relabel each permutation to one in $S_{n-1}$, with each one occurring $n-2$ times. For example, when $n=4$, we would replace the permutations $(\substack{1 \ 2 \ 3 \ 4 \\ 1 \ 4 \ 3 \ 2})$ and $(\substack{1 \ 2 \ 3 \ 4 \\ 1 \ 2 \ 4 \ 3})$ with $(\substack{1 \ 2 \ 3 \\ 1 \ 3 \ 2})$ and $(\substack{1 \ 2 \ 3 \\ 1 \ 2 \ 3})$ respectively. Note that this replacement does not change $\sign\left(\sigma^{-1}(\beta)-\sigma^{-1}(\alpha)\right)$, and thus does not change the summand. We can do the same with the third and fourth lines, these have no repeats as $\sigma^{-1}(n)$ must be $1$ or $n$ depending on the sum, this gives
	\begin{gather*}
		(n-2)\prod_{\alpha=1}^{n-1}(-k_{\alpha}) \sum_{\sigma\in S_{n-1}} (k_{\sigma(n-1)} - k_{\sigma(1)}) \prod_{\alpha<\beta<n} \left(k_\beta - k_\alpha - \sign\left( \sigma^{-1}(\beta) - \sigma^{-1}(\alpha) \right) k_\beta k_\alpha \right) \\
		+ \prod_{\alpha=1}^{n-1}(-k_{\alpha}) \sum_{\sigma\in S_{n-1}} k_{\sigma(n-1)} \prod_{\alpha<\beta<n} \left(k_\beta - k_\alpha - \sign\left( \sigma^{-1}(\beta) - \sigma^{-1}(\alpha) \right) k_\beta k_\alpha \right) \\
		- \prod_{\alpha=1}^{n-1}(-k_{\alpha}) \sum_{\sigma\in S_{n-1}} k_{\sigma(1)} \prod_{\alpha<\beta<n}\left(k_\beta - k_\alpha - \sign\left( \sigma^{-1}(\beta) - \sigma^{-1}(\alpha) \right) k_\beta k_\alpha \right).
	\end{gather*} 
	Which is exactly $(n-1)\prod_{\alpha=1}^{n-1}(-k_{\sigma(\alpha)})$ times the $n-1$ case, and thus is equal to
	\begin{align*}
		(n&-1)\prod_{\alpha=1}^{n-1}(-k_{\alpha}) \prod_{\alpha<\beta<n}(k_\beta - k_\alpha) \\
		\times&\left(C^{(n-1)}_0 + \sum_{m=1}^{\lfloor (n-1)/2\rfloor} \left( C^{(n-1)}_1 \sum_{\alpha_1<...<\alpha_{2m}} k_{\alpha_1}...k_{\alpha_{2m}} +C^{(n-1)}_2\sum_{\alpha_1<...<\alpha_{2m+1}}k_{\alpha_1}... k_{\alpha_{2m+1}} \right) \right).
	\end{align*}
	Comparing coefficients with what we started with, it's clear that $C^{(n)}_m = (n-1)C^{(n-1)}_m$ for $m=0,1,2$ as required. \par
	The proof for the $\tilde{C}^{(n)}_m$ follows the same lines as above.
\end{proof}
Finally we just need to establish the values $C^{(2)}_0,C^{(2)}_1, C^{(2)}_2, \tilde{C}^{(2)}_0, \tilde{C}^{(2)}_1$ and $\tilde{C}^{(2)}_2$ to find all the remaining values by induction. (\ref{BoundaryConditionDiff}) in the $n=2$ case is
\begin{gather*}
	k_1 k_2 (k_2 - k_1 - k_1k_2) + k_1k_2 (k_2 - k_1 + k_1 k_2) = 2(k_2 - k_1) k_1k_2.
\end{gather*}
Thus $C^{(2)}_0=0$, $C^{(2)}_1=0$ and $C^{(2)}_2=2$. Combining the two lemmas above this implies for $m=0,1$ $C^{(n)}_0 = 0$ for every $n$, and $C^{(n)}_2 = 2(n-1)!$ for every $n$. (\ref{BoundaryConditionDrift}) in the $n=2$ case is
\begin{gather*}
	(k_2 - k_1) (k_2 - k_1 - k_1k_2) + (k_1 - k_2) (k_2 - k_1 + k_1 k_2) = -2(k_2 - k_1) k_1k_2.
\end{gather*}
Thus $\tilde{C}^{(2)}_0=0$, $\tilde{C}^{(2)}_1=0$ and $\tilde{C}^{(2)}_2=-2$. Combining the two lemmas above this implies for $m=0,1$ $\tilde{C}^{(n)}_m = 0$ for every $n$, and $\tilde{C}^{(n)}_2 = -2(n-1)!$ for every $n$. In particular this shows that the sum of (\ref{BoundaryConditionDiff}) and (\ref{BoundaryConditionDrift}) is $0$, proving Proposition \ref{Prop: BoundaryConditionEquation1}. As a consequence we have proven Proposition \ref{Proposition: u satisfies boundary conditions}, concluding this subsection.
\subsection{Initial Condition}\label{Section: Initial Condition}
\begin{proposition}\label{Proposition: u satisfies initial condition}
For any bounded Lipschitz continuous function $f:\Weyl{n}\to \R$ we have
\begin{equation*}
	\int u_t(\cdot,y) f(y) m^{(n)}_\theta(dy) \to f \text{ uniformly, as } t\to 0.
\end{equation*}
Where the definitions of $m^{(n)}_\theta$ and $u_t$ are given in definition \ref{ReferenceMeasureOfTheTdensity} and Definition (\ref{BetheAnsatz}) respectively.	
\end{proposition}
First we'll show
\begin{lemma}\label{Lemma: u integrates to one}
	\begin{equation*}
	\int u_t(x,y) m^{(n)}_\theta(dy) =1 \quad \text{for all } x\in \Weyl{n}, \ t>0.	
	\end{equation*}
\end{lemma}
\begin{proof}
	Lemma \ref{uIsIntegrable} allows us to calculate the time derivative by passing it through the integral
	\begin{align*}
		\Partial{}{t}\int u_t(x,y) m^{(n)}_\theta(dy) &= \int \frac{1}{2} \Delta u_t(x,y) m^{(n)}_\theta(dy) \\
		&=0.
	\end{align*}
The first equality is clear from the definition of $u$. The second equality follows from Corollary \ref{StationarityOfm} and Lemma \ref{uIsIntegrable}. This shows the integral is constant, to finish we shall show convergence to $1$ as $t\to \infty$. Scaling $k$ by $t^{-\frac{1}{2}}$ and $y$ by $t^{\frac{1}{2}}$ we see the following
\begin{align*}
	\int u_t(x,y) m^{(n)}_\theta (dy) &= \int \frac{1}{(2\pi)^n} \int_{\R^n} e^{-\frac{1}{2}t|k|^2} \sum_{\sigma\in S_n} e^{ik_\sigma \cdot (x-y_\sigma)} \prod_{\substack{\alpha<\beta :\\ \sigma(\beta)< \sigma(\alpha)}} S_{\sigma(\beta),\sigma(\alpha)}(k) dk \ m^{(n)}_\theta(dy), \\
	&= \sum_{\pi\in \Pi_n} \theta^{|\pi|-n} \left(\prod_{\pi_\iota \in \pi} \frac{1}{|\pi_\iota|} \right) \frac{1}{(2\pi)^n t^{\frac{1}{2}(n-|\pi|)}} \int \int_{\R^n} e^{-\frac{1}{2}|k|^2} \sum_{\sigma\in S_n} e^{ik_\sigma \cdot (x/\sqrt{t}-y_\sigma)} \\
	&\qquad \qquad \prod_{\substack{\alpha<\beta :\\ \sigma(\beta)< \sigma(\alpha)}} S_{\sigma(\beta),\sigma(\alpha)}(k/\sqrt{t}) dk \ \lambda^\pi(dy).
\end{align*}
We can justify applying Dominated convergence to this by referring to lemma \ref{uIsIntegrable}, to take $t\to \infty$. It is clear that $S_{\sigma(\beta),\sigma(\alpha)}(\tfrac{k}{\sqrt{t}}) \to 1$ as $t\to \infty$ for almost every $k$. Notice that all terms with $|\pi|<n$ in the sum over partitions vanish in the limit, leaving only the partition consisting exclusively of singletons. For this partition $\lambda^\pi$ is just the Lebesgue measure on the Weyl chamber. Thus we have
\begin{align*}
	\int u_t(x,y) m^{(n)}_\theta (dy) &= \frac{n!}{(2\pi)^n} \int_{\Weylo{n}} \int_{\R^n} e^{-\frac{1}{2}|k|^2- ik \cdot y} dk dy, \\
	&=1.
\end{align*}
The $n!$ comes from the sum over permutations, the resulting integral in $k$ is just the Fourier transform of a Gaussian and hence the integral over the Weyl chamber is easily calculated.
\end{proof}
Now we can write
\begin{equation*}
	\int u_t(x,y) f(y) m^{(n)}_\theta(dy) -f(x) = \int u_t(x,y) \left( f(y) -f(x)	\right) m^{(n)}_\theta(dy).
\end{equation*}

It follows directly from the definition of $m^{(n)}_\theta$ that
\begin{align}
	&\left| \int u_t(x,y) \left( f(y) -f(x)	\right) m^{(n)}_\theta(dy)\right|\nonumber \\
	\leq& \sum_{\pi \in \Pi_n} \theta^{|\pi|-n} \prod_{\pi_i \in \pi} \tfrac{1}{|\pi_i|} \left| \int u_t(x,y) \left( f(y) -f(x)	\right) \lambda^\pi(dy)\right|.\label{EquationLambdaPi}
\end{align}
Thus we can restrict our attentions to the integral with respect to $\lambda^\pi$ for a fixed $\pi \in \Pi_n$. \par{}
Let us briefly outline the proof. We begin by rearranging $u_t(x,y)$ into a more convenient form, and then split the sum over permutations so that we first sum over permutations $\sigma$ for which the images $(\sigma(\pi_\iota))_{\iota=1}^{|\pi|}$ is fixed. We then bound $u_t(x,y)$ by making contour shifts, following the same idea used to calculate the Fourier transform of the Gaussian density. This step is complicated by the presence of poles in the integral defining $u_t(x,y)$, however our previous step gives us some control over where the poles appear and we can further use that $x$ and $y$ are both in the Weyl chamber to derive Gaussian bounds on $u_t(x,y)$. In the final step we combine these bounds with the Lipschitz property for $f$ to derive the desired uniform convergence. This requires bounding of the contribution from $\Weyl{\pi}$ to $\int |u_t(x,y)| m^{(n)}_\theta(dy)$ and some care in considering what happens when $x$ is near but not in $\Weyl{\pi}$ to ensure we get uniform convergence. \par{}
To begin our rearrangements we prove that $u_t(x,y)$ is symmetric under swaps of $x$ and $y$.
\begin{lemma}\label{uisSymmetric}
	For every $x,y \in \Weyl{n}$ and $t>0$
	\begin{equation*}
		u_t(x,y)=u_t(y,x).
	\end{equation*}
\end{lemma}
\begin{proof}
Recall that $u$ is defined in (\ref{BetheAnsatz}) as
\begin{equation*}
	u_t(x, y) =\tfrac{1}{(2\pi)^n} \int_{\R^n} e^{-\frac{1}{2}t|k|^2} \sum_{\sigma \in S_n} e^{ik_\sigma \cdot (x- y_\sigma)} \prod_{\substack{\alpha<\beta: \\ \sigma(\alpha) > \sigma(\beta)}} \tfrac{i\theta( k_{\sigma(\alpha)}-k_{\sigma(\beta)})+ k_{\sigma(\alpha)}k_{\sigma(\beta)}}{i\theta( k_{\sigma(\alpha)}-k_{\sigma(\beta)})- k_{\sigma(\alpha)}k_{\sigma(\beta)}} dk.
\end{equation*}
If we first take the sum outside the integral, then perform the change of variables in the $k$ integral, $k\to -k_{\sigma^{-1}}$, this becomes
\begin{equation*}
	\tfrac{1}{(2\pi)^n} \sum_{\sigma \in S_n} \int_{\R^n} e^{-\frac{1}{2}t|k|^2+ ik_{\sigma^{-1}} \cdot (x_{\sigma^{-1}}- y)} \prod_{\substack{\alpha<\beta: \\ \sigma(\alpha) > \sigma(\beta)}} \tfrac{i\theta( k_{\beta}-k_{\alpha})+ k_{\alpha}k_{\beta}}{i\theta( k_{\beta}-k_{\alpha})- k_{\alpha}k_{\beta}} dk.
\end{equation*}
Notice that we can relabel the product as follows
\begin{equation*}
	\prod_{\substack{\alpha<\beta: \\ \sigma(\alpha) > \sigma(\beta)}} \tfrac{i\theta( k_{\beta}-k_{\alpha})+ k_{\alpha}k_{\beta}}{i\theta( k_{\beta}-k_{\alpha})- k_{\alpha}k_{\beta}} = \prod_{\substack{\alpha<\beta: \\ {\sigma^{-1}}(\alpha) > {\sigma^{-1}}(\beta)}} \tfrac{i\theta( k_{{\sigma^{-1}}(\alpha)}-k_{{\sigma^{-1}}(\beta)})+ k_{{\sigma^{-1}}(\alpha)}k_{{\sigma^{-1}}(\beta)}}{i\theta( k_{{\sigma^{-1}}(\alpha)}-k_{{\sigma^{-1}}(\beta)})- k_{{\sigma^{-1}}(\alpha)}k_{{\sigma^{-1}}(\beta)}}.
\end{equation*}
Hence by relabelling the sum to be over $\sigma^{-1}\in S_n$ we see that we get $u_t(y,x)$ as desired.
\end{proof}
Now we proceed with the proof of the proposition, we can rewrite the summand of (\ref{EquationLambdaPi}) (ignoring the constants) as
\begin{align}
	&\left| \int u_t(y,x) \left( f(y) -f(x)	\right) \lambda^\pi(dy)\right| \nonumber\\
	=&\left| \int \tfrac{1}{(2\pi)^n} \int_{\R^n} e^{-\frac{1}{2}t|k|^2} \sum_{\sigma \in S_n} e^{ik_\sigma \cdot (y- x_\sigma)} \prod_{\substack{\alpha<\beta: \\ \sigma(\alpha) > \sigma(\beta)}} S_{\sigma(\beta),\sigma(\alpha)}(k) \ dk \left( f(y) -f(x)	\right) \lambda^\pi(dy)\right|. \label{IC:FixedPartition}
\end{align}
For a partition $\pi\in \Pi_n$ and permutation $\sigma\in S_n$ define the set of ordered pairs $\sigma(\pi): = \{(\pi_1, \sigma(\pi_1)),...,(\pi_{|\pi|}, \sigma(\pi_{|\pi|}))\}$ (Where $\sigma(A)$ denotes the image of $A$ under $\sigma$). We can rewrite the sum in the above integral as follows 
\begin{align*}
	\sum_{\substack{\tau\in S_n: \\ \tau|_{\pi_\iota} \text{ is increasing } \forall \iota}} \sum_{\substack{\sigma \in S_n: \\ \sigma(\pi) = \tau(\pi)}} e^{ik_\sigma \cdot (y-x_\sigma)} \prod_{\substack{\alpha<\beta: \\ \sigma(\alpha) > \sigma(\beta)}} S_{\sigma(\beta),\sigma(\alpha)}(k).
\end{align*}
Let's consider $e^{ik_\sigma \cdot (y - x_\sigma)}= e^{-ik\cdot x}\prod_{j=1}^n e^{ik_{\tau(j)}y_j}$. We'll use the notation $\overline{\pi_\iota}:=\sup\pi_\iota$, and $\underline{\pi_\iota}:=\sup\pi_\iota$. We know that for each $\pi_\iota \in \pi$, $\alpha,\beta \in \pi_\iota$ implies $y_\alpha = y_\beta$ $\lambda^\pi$-a.e. Hence $\prod_{j=1}^n e^{ik_{\sigma(j)}y_j} = \prod_{\pi_\iota \in \pi} \prod_{\alpha \in \pi_\iota} e^{i k_{\sigma(\alpha)}y_{\overline{\pi_\iota}}}$ $\lambda^\pi$-a.e. But since $\sigma(\pi)=\tau(\pi)$ this is just equal to  $\prod_{\pi_\iota \in \pi} \prod_{\alpha \in \pi_\iota} e^{i k_{\tau(\alpha)}y_{\overline{\pi_\iota}}}$ which equals $e^{ik_\tau \cdot y}$. Hence we can pull the exponential out of the second sum to make the previous expression equal $\lambda^\pi$-a.e. to
\begin{align*}
	\sum_{\substack{\tau \in S_n:\\ \tau|_{\pi_\iota} \text{ is increasing } \forall \iota}} e^{ik_\tau \cdot (y-x_\tau)}  \sum_{\substack{\sigma \in S_n: \\ \sigma(\pi) = \tau(\pi)}} \prod_{\substack{\alpha<\beta: \\ \sigma(\alpha) > \sigma(\beta)}} S_{\sigma(\beta),\sigma(\alpha)}(k).
\end{align*}
Now we can consider the product, in particular we can show that when $\alpha$ and $\beta$ are in different elements of $\pi$ then the appearance of $S_{\sigma(\beta),\sigma(\alpha)}(k)$ in the product depends only on $\tau$ and not the specific $\sigma$. Suppose $\alpha<\beta$ are in different elements of $\pi$ and that $\sigma(\beta)<\sigma(\alpha)$. Since $\pi$ is an ordered partition there exists $\iota<j$ such that $\alpha\in \pi_\iota$ and $\beta\in \pi_j$. Now since $\sigma(\pi)=\tau(\pi)$ there must exist $\gamma \in \pi_\iota$, and $\delta \in \pi_j$ (thus $\gamma<\delta$) such that $\tau(\gamma) = \sigma(\alpha) >\sigma(\beta)= \tau(\delta)$. Hence for each such $\alpha < \beta$ such that $\sigma(\beta) < \sigma(\alpha)$, where $\alpha$ and $\beta$ are in different elements of $\pi$, there are $\gamma< \delta$ in different elements of $\pi$ such that $\tau(\delta)< \tau(\gamma)$. Similarly we can go in the other direction, so that if $\alpha$ and $\beta$ are in different elements of $\pi$ then $(\sigma(\beta),\sigma(\alpha))$ is an inversion for $\sigma$ if and only if it is an inversion for $\tau$ (That is if $\alpha$ and $\beta$ are in different elements of $\pi$ then $\alpha< \beta$ with $\sigma(\beta)< \sigma(\alpha)$ occurs if and only if $\tau^{-1}(\sigma(\alpha))< \tau^{-1}(\sigma(\beta))$ with $\sigma(\beta)<\sigma(\alpha)$). Hence we can split off the part of the product where $\alpha$ and $\beta$ are in different elements of $\pi$ and rewrite entirely in terms of $\tau$. Hence the previous expression is equal to
\begin{align}\label{IC: split permutation sum}
	\sum_{\substack{\tau \in S_n:\\ \tau|_{\pi_\iota} \text{ is increasing } \forall \iota}} e^{ik_\tau \cdot (y-x_\tau)} \left( \prod_{\iota <j} \prod_{\substack{\alpha\in \pi_\iota, \ \beta \in \pi_j:\\ \tau(\beta)< \tau(\alpha)}} S_{\tau(\beta),\tau(\alpha)}(k) \right) \sum_{\substack{\sigma \in S_n: \\ \sigma(\pi) = \tau(\pi)}} \prod_{\pi_\iota \in\pi} \prod_{\substack{\alpha<\beta: \\ \sigma(\alpha) > \sigma(\beta), \\ \alpha,\beta \in \pi_\iota }} S_{\sigma(\beta),\sigma(\alpha)}(k).
\end{align}
We can calculate the second sum using the formula in the following lemma.
\begin{lemma}\label{Lemma: sum over permutations}
Suppose $m\in \N$ and $\theta>0$ then for all $k\in \R^m$ such that $k_\alpha \neq 0$ for all $\alpha\in \{1,\dots, m\}$
	\begin{align*}
		\sum_{\sigma \in S_m} \prod_{\substack{\alpha< \beta: \\ \sigma(\alpha)>\sigma(\beta)}} S_{\sigma(\beta),\sigma(\alpha)}(k) = m! \prod_{\alpha< \beta} \frac{i\theta(k_\beta - k_\alpha)}{ i\theta(k_\beta - k_\alpha) - k_\alpha k_\beta}.
	\end{align*}	
\end{lemma}
\begin{proof}
	First we prove the following equality holds for all $\xi \in \C^m$
	\begin{equation}\label{IC: Combinatorial formula}
		\sum_{\sigma \in S_m} \left( \prod_{\substack{\alpha< \beta: \\ \sigma(\beta)< \sigma(\alpha)}} (-1) \right) \left( \prod_{\alpha< \beta} \xi_{\sigma(\alpha)} - \xi_{\sigma(\beta)} - 1 \right) = m! \prod_{\alpha< \beta} (\xi_\alpha - \xi_\beta).
	\end{equation}
	It's clear that the left hand side is a degree $\binom{m}{2}$ polynomial, which we shall denote $P(\xi)$. Thus if we can prove that $P(\xi)$ is also alternating, it must be a constant multiple of the right hand side. We then just need to check the constant to finish the proof. \par{}
	To prove the left hand side is alternating it is enough to consider swaps of consecutive variables, e.g. $\xi_j$ and $\xi_{j+1}$ for some $j\in \{1,..., m-1\}$. Let $s_j = (j, j+1)\in S_n$, i.e. the permutation that swaps $j$ and $j+1$ leaving everything else fixed. Clearly for all $\sigma\in S_n$
	\begin{align}
		\prod_{\substack{\alpha< \beta : \\ \sigma(\beta)< \sigma(\alpha)}} (-1) = - \prod_{\substack{\alpha< \beta : \\ \sigma\circ s_j(\beta)< \sigma\circ s_j(\alpha)}} (-1).
	\end{align}
	It follows simply by relabelling the sum in its definition that $P(\xi_{s_j}) = - P(\xi)$. Hence $P$ is an alternating polynomial and there is a $c\in \R$ such that
	\begin{align*}
		\sum_{\sigma \in S_m} \left( \prod_{\substack{\alpha< \beta: \\ \sigma(\beta)< \sigma(\alpha)}} (-1) \right) \left( \prod_{\alpha< \beta} \xi_{\sigma(\alpha)} - \xi_{\sigma(\beta)} - 1 \right) = c \prod_{\alpha< \beta} (\xi_\alpha - \xi_\beta).
	\end{align*}
	To finish we just have to note that if we expand the bracket on the left hand side we get $m! \prod_{\alpha< \beta} (\xi_\alpha - \xi_\beta)$ plus additional terms of lower degree. But we know that the left hand side, $P$, is a constant multiple of $\prod_{\alpha< \beta} (\xi_\beta - \xi_\alpha)$, thus the lower degree terms must cancel. This proves (\ref{IC: Combinatorial formula}). To finish the proof we just need to divide both sides of (\ref{IC: Combinatorial formula}) by $\prod_{\substack{\alpha< \beta}} (\xi_{\beta}- \xi_{\alpha} -1)$ then set $\xi_j = i\theta/k_j$ for each $j$ to get the desired identity.
\end{proof}
Hence we get that (\ref{IC: split permutation sum}) is equal to
\begin{align*}
	&\sum_{\substack{\tau \in S_n:\\ \tau|_{\pi_\iota} \text{ is increasing } \forall \iota}} e^{ik_\tau \cdot (y-x_\tau)} T^{\tau,\pi}(k).
\end{align*}
Where $T^{\tau,\pi}:\C^n \to \C^n$ is defined (for a.e. $k\in \C^n$) as follows
\begin{align*}
	T^{\tau,\pi}(k):= \left(\prod_{\iota <j} \prod_{\substack{\alpha\in \pi_\iota, \ \beta \in \pi_j:\\ \tau(\beta)< \tau(\alpha)}}S_{\tau(\beta),\tau(\alpha)}(k)\right) \left( \prod_{\pi_\iota \in\pi} |\pi_\iota|!\prod_{\substack{\alpha<\beta: \\ \alpha,\beta \in \pi_\iota }} \tfrac{i\theta( k_{\tau(\beta)}-k_{\tau(\alpha)})}{i\theta( k_{\tau(\beta)}-k_{\tau(\alpha)})- k_{\tau(\alpha)}k_{\tau(\beta)}} \right).
\end{align*}
This rearrangement, together with the triangle inequality gives us that (\ref{IC:FixedPartition}) is bounded above by
\begin{align}
	\tfrac{\prod_{\iota=1}^{|\pi|}|\pi_\iota|! }{(2\pi)^n} \sum_{\substack{\tau \in S_n:\\ \tau|_{\pi_\iota} \text{ is increasing } \forall \iota}}& \int\bigg|\int_{\R^n} e^{-\frac{1}{2}t|k|^2+ik_\tau \cdot (y-x_\tau)} T^{\tau,\pi}(k) dk\bigg| | f(y)- f(x) |\lambda^\pi(dy) .\label{IC: PreContourShift}
\end{align}
Now we can move on to the next step, which we briefly motivate. We want to get control on the $k$ integral in the above expression, and we need the bound to be integrable in $y$ with respect to $\lambda^\pi$ and to be vanishing as $t\to 0$ whenever $y\neq x$. Note that we can rewrite the exponent appearing in the integrand as follows
\begin{equation*}
	-\frac{1}{2}t|k|^2+ik_\tau \cdot (y-x_\tau) = -\frac{1}{2}t\sum_{\alpha=1}^n (k_{\tau(\alpha)} - \frac{i}{t}(y_\alpha - x_{\tau(\alpha)}))^2 - \frac{(y_\alpha - x_{\tau(\alpha)})^2}{2t}.
\end{equation*}
Suggesting that we should use Cauchy's residue theorem to shift the $k_{\tau(\alpha)}$ contour from $\R$ to $C_\alpha:=\{z\in \C : \ z-\frac{i}{t}(y_\alpha - x_{\tau(\alpha)})\in \R\}$ for each $\alpha\in \{1,..., n\}$, and then parametrise the resulting contour integral as an integral over $\R$. Supposing we can do this without encountering any poles the exponent becomes
\begin{equation*}
	-\frac{1}{2}t\sum_{\alpha=1}^n \tilde{k}_{\tau(\alpha)}^2 - \frac{(y_\alpha - x_{\tau(\alpha)})^2}{2t}.
\end{equation*}
Where $\tilde{k}_{\tau(\alpha)}\in \R$ is our new integration variable. The second term of the summand gives us the necessary control in the $y$ variable, and the first term should allow us to control the resulting $k$ integral. However this approach is complicated by $T^{\tau,\pi}$, which contribute poles that hinder our contour shifting. We end up not being able to shift the integration contours for all of the $k$ variables without encountering poles, however we are still able to make some of the desired contour shifts. To see which shifts can be made we need check where these poles occur, so first note that by definition
\begin{align} \label{IC: Two products}
	T^{\tau,\pi}(k) =\left(\prod_{\iota <j} \prod_{\substack{\alpha<\beta:\\ \tau(\beta)< \tau(\alpha), \\ \alpha\in \pi_\iota, \ \beta \in \pi_j}} \tfrac{i\theta(k_{\tau(\alpha)}- k_{\tau(\beta)}) + k_{\tau(\alpha)} k_{\tau(\beta)}}{i\theta(k_{\tau(\alpha)}- k_{\tau(\beta)}) - k_{\tau(\alpha)} k_{\tau(\beta)}} \right) \left( \prod_{\pi_\iota \in \pi}\prod_{\substack{\alpha<\beta: \\ \alpha,\beta \in \pi_\iota }} \tfrac{i\theta( k_{\tau(\beta)}-k_{\tau(\alpha)})}{i\theta( k_{\tau(\beta)}-k_{\tau(\alpha)})- k_{\tau(\alpha)}k_{\tau(\beta)}} \right)
\end{align}
The following lemma provides us with the desired information.
\begin{lemma}\label{Lemma: poles of the factors in T}
	We'll use the notation $\mathbb{H}= \{x+iy \in \C| \ x\in \R, \ y\in\R_{>0}\}$ for the upper half complex plane. The function $(z, w) \mapsto i\theta(z-w)- zw$ has no zeroes in the set $\mathbb{H}\times -\mathbb{H}$.
\end{lemma}
\begin{proof}
	For $w\in -\mathbb{H}$ there are $a\in \R$ and $b\in \R_{>0}$ such that $w= a- bi$. It's easily checked that $i\theta(z - w) - zw= 0$ if and only if we have
	\begin{align*}
		z = \frac{\theta^2 a - i \theta ((\theta+ b)b + a^2)}{(\theta + b)^2 + a^2} \in -\mathbb{H}.
    \end{align*}
    Thus there are no zeroes inside $\mathbb{H}\times -\mathbb{H}$ as claimed.
\end{proof}
Observing the structure of the products in (\ref{IC: Two products}) we define the set $E^{\tau, \pi} \subset \C^n$ defined as $\times_{k=1}^n E^{\tau, \pi}_k$ where $E^{\tau, \pi}_k$ is the upper half complex plane if there is a $\pi_\iota \in \pi$ such that $k = \sup \pi_\iota$ and $\tau(\alpha)< \tau(k)$ for all $\alpha< k$, the lower half complex plane if there is a $\pi_\iota$ such that $k= \inf \pi_\iota$ and $\tau(\beta)>\tau(k)$ for all $\beta>k$, the whole complex plane if both of these conditions are satisfied, and the real line if neither are satisfied. Lemma \ref{Lemma: poles of the factors in T} shows the denominator of $T^{\tau, \pi}$ as in (\ref{IC: Two products}) has no zeroes in the set $E^{\tau, \pi}$ (\ref{IC: Two products}), and thus we can perform our contour shifts as long as the contours remain within this set. To simplify our notation slightly we'll henceforth write $\overline{\pi_\iota}:= \sup\pi_\iota$ and $\underline{\pi_\iota}:= \inf \pi_\iota$ for each $\pi_\iota\in \pi$.

We'll now use these ideas to get a following bound on the $k$ integral in line (\ref{IC: PreContourShift}), first we need to find a family of indices for which contour shifts can be made, that is a collection of $\alpha$ such that $E^\tau_\alpha$ contains at least one complex half plane. \par{}

\begin{lemma}\label{IC: Lemma a b}
    Suppose $\pi \in \Pi_n$ and $\tau\in S_n$ such that $\tau|_{\pi\iota}$ is increasing for every $\pi_\iota \in \pi$. For each $\pi_\iota \in \pi$ there are $a_\iota \leq \iota \leq b_\iota$ such that $\tau(\underline{\pi_{b_\iota}})\leq \tau(\overline{\pi_{a_\iota}})$, and the following properties hold
    \begin{itemize}
        \item $\tau(\underline{\pi_{b_\iota}})< \tau(\beta)$ for every $\beta> \underline{\pi_{b_\iota}}$;
        \item  and $\tau(\overline{\pi_{a_\iota}})> \tau(\alpha)$ for all $\alpha< \overline{\pi_{\iota}}$.
    \end{itemize}
    Further, given such a $(a_\iota, b_\iota)$, we can define $m_\iota:= \sup\{\tau(\alpha)| \ \overline{\pi_{a_\iota}}\leq \alpha \leq \underline{\pi_{b_\iota}}\}$ and $l_\iota:= \inf\{\tau(\beta)| \ \overline{\pi_{a_\iota}}\leq \beta \leq \underline{\pi_{b_\iota}}\}$ if $\overline{\pi_{a_\iota}}< \underline{\pi_{b_\iota}}$; and $m_\iota:= \tau(\overline{\pi_{a_\iota}})$ and $l_\iota:= \tau(\underline{\pi_{b_\iota}})$ if $\overline{\pi_{a_\iota}}\geq \underline{\pi_{b_\iota}}$. The following properties hold for $m_\iota$ and $l_\iota$:
    \begin{itemize}
        \item there are $\pi_c, \pi_d \in \pi$ such that $\tau^{-1}(m_\iota)= \overline{\pi_d}$ and $\tau^{-1}(l_\iota)= \underline{\pi_c}$;
        \item for all $\alpha< \tau^{-1}(m_\iota)$ we have $\tau(\alpha)< m_\iota$;
        \item and for all $\beta> \tau^{-1}(l_\iota)$ we have $\tau(\beta)> l_\iota$.
    \end{itemize}
\end{lemma}
\begin{proof}
    First we define $\mu_\iota:=\overline{\pi_{a_\iota}}$ where $a_\iota:= \inf\{a \leq \iota:\ \tau(\overline{\pi_a})\geq \tau(\underline{\pi_\iota})\}$, and then from it we define $\nu_\iota := \underline{\pi_{b_\iota}}$ where $b_\iota:= \sup\{b\geq \iota: \ \tau(\mu_\iota) \geq \tau(\underline{\pi_b})\}$. $\mu_\iota$ and $\nu_\iota$ are introduced for convenience and will be used throughout this section. It is easy to see that the $a_\iota$ and $b_\iota$ satisfy the first two properties we claimed for them, namely that $a_\iota \leq \iota \leq b_\iota$ and $\tau(\nu_\iota)= \tau(\underline{\pi_{b_\iota}})\leq \tau(\overline{\pi_{a_\iota}})= \tau(\mu_\iota)$. \par{}

    Now we show $\tau(\nu_\iota)< \tau(\beta)$ for all $\beta> \nu_\iota$, and $\tau(\mu_\iota)>\tau(\alpha)$ for all $\alpha< \mu_\iota$. Starting with $\mu_\iota$, if there is an $\alpha< \mu_\iota$ such that $\tau(\mu_\iota)< \tau(\alpha)$ then by definition of $\mu_\iota$ $\alpha$ must be in a different element of $\pi$ to $\mu_\iota$, say $\pi_c$, with $c< a_\iota$. Since $\tau$ is increasing on every element of $\pi$ this means we must have $\tau(\overline{\pi_c})> \tau(\alpha)> \tau(\mu_\iota) = \tau(\overline{\pi_{a_\iota}})$ which contradicts the definition of $a_\iota$, so no such $\alpha$ exists. By a similar argument there is no $\beta> \nu_\iota$ such that $\tau(\nu_\iota)< \tau(\beta)$. \par{}

    It just remains to prove the second set of statements, those about $m_\iota$ and $l_\iota$. Suppose we are given $(a_\iota, b_\iota)$ as in the first part of the lemma and once more define $\mu_\iota:=\overline{\pi_{a_\iota}}$ and $\nu_\iota := \underline{\pi_{b_\iota}}$. The first property for $m_\iota$ and $l_\iota$ follows immediately from the fact that $\tau|_{\pi_j}$ is increasing for all $\pi_j\in \pi$, the definitions of $m_\iota$ and $l_\iota$, and from $\pi\in \Pi_n$. For the second and third statements we consider two cases separately: $\mu_\iota < \nu_\iota$ and $\nu_\iota \leq \mu_\iota$. For the latter case we have $m_\iota= \tau(\nu_\iota)$ and $l_\iota= \tau(\mu_\iota)$, so the statements are the same as those we just proved. If instead we have $\mu_\iota< \nu_\iota$ we can argue the second statement as follows. Clearly for all $\alpha$ such that $\mu_\iota\leq \alpha \leq \nu_\iota$ we have $\tau(\alpha)< m_\iota$, thus we only need to check that $\alpha<\mu_\iota$ implies $\tau(\alpha)< m_\iota$. Suppose this is false, i.e. there is an $\alpha< \mu_\iota$ such that $\tau(\alpha)> m_\iota$. Since $m_\iota > \tau(\mu_\iota)$ this implies $\tau(\alpha) > \tau(\mu_\iota)$, since we also have $\alpha<\mu_\iota$ this is a contradiction as we know from previously that $\tau(\mu_\iota)> \tau(\alpha)$ whenever $\mu_\iota> \alpha$. A similar argument proves the third statement, thereby proving the lemma.
\end{proof}
In the following proposition we will assume we have a $\pi \in \Pi_n$ with a family $(a_\iota, b_\iota)_{\pi_\iota \in \pi}$ given by the above lemma, and adopt the notation of the above proof, namely $\mu_\iota:= \overline{\pi_{a_\iota}}$ and $\nu_\iota:= \underline{\pi_{b_\iota}}$. The above lemma ensures that whenever $\alpha= \mu_\iota, \tau^{-1}(m_\iota)$ the set $E^{\tau, \pi}_{\alpha}$ contains the upper half complex plane, and if $\beta= \nu_\iota, \tau^{-1}(l_\iota)$ then $E^{\tau, \pi}_\beta$ contains the lower half complex plane.
\begin{proposition}\label{Prop: k integral bound}
	Suppose $\pi \in \Pi_n$ and $\tau\in S_n$ such that $\tau|_{\pi\iota}$ is increasing for every $\pi_\iota \in \pi$, and for each $\pi_\iota \in \pi$ we have $a_\iota \leq \iota \leq b_\iota$ as in the above lemma. There is a constant $C>0$, depending only on $\pi$ and $n$, such that the following bound holds for all $x,y \in \Weyl{n}$ 
	\begin{align}
		&\left|\int_{\R^n} e^{-\frac{1}{2}t|k|^2+ik_\tau \cdot (y-x_\tau)} T^{\tau,\pi}(k) dk\right| \leq C t^{-\frac{1}{2}|\pi|} |\log(t)|^{|\pi|} e^{-\frac{|y - \chi|^2}{12nt}}\prod_{\pi_\iota \in \pi} e^{- \frac{1}{24nt}\left((x_{m_\iota} - \chi^\iota)^2 + (x_{l_\iota} - \chi^\iota)^2 \right)}. \label{Prop: k integral bound inequality}
    \end{align}
    Where $\chi=\chi(x)\in \R^n$ is defined by $\chi_\alpha:= \chi^\iota:= \frac{1}{2}(x_{\tau(\mu_\iota)} + x_{\tau(\nu_\iota)})$ for all $\alpha \in \pi_\iota $.
\end{proposition}

We begin the proof with the following intermediate bound
\begin{lemma}\label{Lemma: intermediate k bound}
	Let $\Gamma_{\alpha, x, y}= C_\alpha$ if $x,y\in \Weyl{n}$ are such that the $C_{\alpha}$ contour lies in $E^{\tau, \pi}_\alpha$, and $\R$ otherwise.
	\begin{align}
		\left|\int_{\R^n} e^{-\frac{1}{2}t|k|^2+ik_\tau \cdot (y-x_\tau)} T^{\tau,\pi}(k) dk\right|	\leq&  e^{-\frac{|y - \chi|^2}{12nt}} \left(\prod_{\pi_\iota \in \pi} e^{- \frac{1}{24nt}\left((x_{m_\iota} - \chi^\iota)^2 + (x_{l_\iota} - \chi^\iota)^2 \right)}\right)\nonumber\\
		&\qquad \int_{\times_{\alpha=1}^n \Gamma_{\alpha, x, y}} e^{-\frac{1}{2}t\sum_{\alpha=1}^n \Re(k_{\tau(\alpha)})^2} \left| T^{\tau,\pi}(k) \right| dk.
	\end{align}
\end{lemma}
\begin{proof}
We note that from the definition of the contours $\Gamma_{\alpha,x,y}$ and the set $E^{\tau, \pi}= \times_{\alpha=1}^n E^{\tau, \pi}_\alpha$ the following equality follows by Cauchy's residue theorem
\begin{align}
	&\int_{\R^n} e^{-\frac{1}{2}t|k|^2+ik_\tau \cdot (y-x_\tau)} T^{\tau,\pi}(k) dk \nonumber\\
	=& e^{-\frac{1}{2t}\sum_{\alpha: \Gamma_{\alpha,x,y}\neq \R} (y_\alpha - x_{\tau(\alpha)})^2} \int_{\times_{\alpha=1}^{n} \Gamma_{\alpha, x, y}} e^{-\frac{1}{2}t\sum_{\alpha=1}^n \Re(k_{\tau(\alpha)})^2 +i \sum_{\alpha: \Gamma_{\alpha, x, y}= \R} k_{\tau(\alpha)}(y_\alpha- x_{\tau(\alpha)})} T^{\tau, \pi}(k) dk.\label{IC: Contour shifted k integral}
\end{align}
From which we see that to prove lemma \ref{Lemma: intermediate k bound} we must bound the exponential appearing in front of the integral, which means we need to consider which contour shifts have been made. We want to check when the condition for $\Gamma_{\alpha, x, y} = C_\alpha$ is true, for $\alpha= \mu_\iota, \nu_\iota$, thus we want to check when $C_\alpha$ lies inside $E^{\tau, \pi}_\alpha$. We know from lemma \ref{IC: Lemma a b} that $E^{\tau, \pi}_{\mu_\iota}$ contains the upper half complex plane, and $E^{\tau, \pi}_{\nu_\iota}$ contains the lower half complex plane. Thus $C_{\mu_\iota}\subset E^{\tau}_{\mu_\iota}$ when $y_{\mu_\iota} \geq x_{\tau(\mu_\iota)}$, and $C_{\nu_\iota}\subset E^{\tau}_{\nu_\iota}$ when $y_{\nu_\iota} \leq x_{\tau(\nu_\iota)}$. Hence we have the following inequalities:
\begin{align*}
	&-\frac{1}{2t}\mathbbm{1}_{\{ \Gamma_{\mu_\iota ,x,y}\neq \R\}} (y_{\mu_\iota} - x_{\tau(\mu_\iota)})^2 \leq -\frac{1}{2t}\mathbbm{1}_{\{ (y_{\mu_\iota} \geq x_{\tau(\mu_\iota)}) \}} (y_{\mu_\iota} - x_{\tau(\mu_\iota)})^2;\\
	&-\frac{1}{2t}\mathbbm{1}_{\{ \Gamma_{\nu_\iota ,x,y}\neq \R\}} (y_{\nu_\iota} - x_{\tau(\nu_\iota)})^2 \leq -\frac{1}{2t}\mathbbm{1}_{\{ (y_{\nu_\iota} \leq x_{\tau(\nu_\iota)}) \}} (y_{\nu_\iota} - x_{\tau(\nu_\iota)})^2.
\end{align*}

There are two cases of interest, the first is when $\mu_\iota< \nu_\iota$ in this case the two indices are in different elements of $\pi$, the second is when $\mu_\iota\geq \nu_\iota$ for which the two indices are in the same element of $\pi$. Let us deal now with the first case. \par{}

By definition we have $\tau(\mu_\iota) \geq \tau(\nu_\iota)$, thus since $x, y \in \Weyl{n}$ it follows that we always have $y_{\nu_\iota}\leq y_{\mu_\iota}$. Hence if we have both $y_{\nu_\iota}> x_{\tau(\nu_\iota)}$ and $y_{\mu_\iota}< x_{\tau(\mu_\iota)}$ it follows that $x_{\tau(\nu_\iota)} < x_{\tau(\mu_\iota)}$ but since $x\in \Weyl{n}$ this is a contradiction. Hence for all $x,y \in \Weyl{n}$ at least one of $y_{\mu_\iota} \geq x_{\tau(\mu_\iota)}$ and $y_{\nu_\iota} \leq x_{\tau(\nu_\iota)}$ must be true. This means we have the following equality
\begin{align} \label{IC: The exponents}
	&-\frac{1}{2t}\mathbbm{1}_{\{ (y_{\mu_\iota} \geq x_{\tau(\mu_\iota)}) \}} (y_{\mu_\iota} - x_{\tau(\mu_\iota)})^2 -\frac{1}{2t}\mathbbm{1}_{\{ (y_{\nu_\iota} \leq x_{\tau(\nu_\iota)}) \}} (y_{\nu_\iota} - x_{\tau(\nu_\iota)})^2 \nonumber\\
	=&\begin{cases}
		-\frac{1}{2t} (y_{\mu_\iota}- x_{\tau(\mu_\iota)})^2, \ \text{if } y_{\mu_\iota} \geq x_{\tau(\mu_\iota)} \text{ and } y_{\nu_\iota}> x_{\tau(\nu_\iota)}; \\
		-\frac{1}{2t} (y_{\nu_\iota}- x_{\tau(\nu_\iota)})^2, \ \text{if } y_{\mu_\iota} < x_{\tau(\mu_\iota)} \text{ and } y_{\nu_\iota} \leq x_{\tau(\nu_\iota)}; \\
		 -\frac{1}{2t} (y_{\mu_\iota}- x_{\tau(\mu_\iota)})^2 -\frac{1}{2t} (y_{\nu_\iota}- x_{\tau(\nu_\iota)})^2, \ \text{if } y_{\nu_\iota} \leq x_{\tau(\nu_\iota)} \ \text{and } y_{\mu_\iota} \geq x_{\tau(\mu_\iota)}.
	\end{cases}
\end{align}
Let $\chi^\iota:= \frac{1}{2}(x_{\mu_\iota}+ x_{\nu_\iota})$ we can rewrite the first line as
\begin{equation*}
	-\frac{1}{2t}\left((y_{\mu_\iota} - \chi^\iota)^2 + \frac{1}{2}(x_{\tau(\mu_\iota)}- x_{\tau(\nu_\iota)})^2 \right) + \frac{1}{2t}(y_{\mu_\iota}- \chi^\iota)(x_{\tau(\mu_\iota)}- x_{\tau(\nu_\iota)}).
\end{equation*}
We have $\tau(\mu_\iota) > \tau(\nu_\iota)$ and $x\in \Weyl{n}$, so that $(x_{\tau(\mu_\iota)}- x_{\tau(\nu_\iota)})< 0$. From $y\in \Weyl{n}$ and $\mu_\iota<\nu_\iota$ it follows that $y_{\mu_\iota}\geq \frac{1}{2}(y_{\mu_\iota} + y_{\nu_\iota})$ which, under the conditions of the first line, is bounded below by $\chi^\iota = \frac{1}{2}(x_{\mu_\iota}+ x_{\nu_\iota})$. Thus $y_{\mu_\iota} - \chi^\iota > 0$, and the last term above is negative. We also have $y_{\nu_\iota} - \chi_\iota \geq y_{\nu_\iota} - x_{\tau(\nu_\iota)}>0$ in this case, thus using $y_{\nu_\iota} \leq y_{\mu_\iota}$ we get $-(y_{\mu_\iota} - \chi^\iota)^2 \leq -(y_{\nu_\iota} - \chi^\iota)^2$.  It follows that the above expression is bounded above by
\begin{equation*}
	-\frac{1}{4t}\left((y_{\mu_\iota} - \chi^\iota)^2 + (y_{\nu_\iota} - \chi^\iota)^2 + (x_{\tau(\mu_\iota)}- x_{\tau(\nu_\iota)})^2 \right).
\end{equation*}
The same ideas yield the same bound on the cases of the second and third lines, so that the above expression is a bound for (\ref{IC: The exponents}). \par{}

Now we need to look at the contour shifts for $m_\iota$ and $l_\iota$, recall $m_\iota= \sup\{\tau(\alpha)| \ \mu_\iota\leq \alpha \leq \nu_\iota\}$ and $l_\iota= \inf\{\tau(\beta)| \ \mu_\iota\leq \beta \leq \nu_\iota\}$. Note that it is quite possible for $m_\iota= \tau(\mu_\iota)$ or for $l_\iota= \tau(\nu_\iota)$. We need to check when $C_{\tau^{-1}(m_\iota)}\subset E^{\tau, \pi}_{\tau^{-1}(m_\iota)}$. From Lemma \ref{IC: Lemma a b} we know $ E^{\tau, \pi}_{\tau^{-1}(m_\iota)}$ contains the upper half complex plane. Therefore $\Gamma_{\tau^{-1}(m_\iota), x, y}= C_{\tau^{-1}(m_\iota)}$ if $y_{\tau^{-1}(m_\iota)}\geq x_{m_\iota}$. Therefore we have
\begin{align*}
	-\frac{1}{2t}\mathbbm{1}_{\{ \Gamma_{\tau^{-1}(m_\iota) ,x,y}\neq \R\}} (y_{\tau^{-1}(m_\iota)} - x_{m_\iota})^2 \leq -\frac{1}{2t}\mathbbm{1}_{\{ (y_{\tau^{-1}(m_\iota)} \geq x_{m_\iota}) \}} (y_{\tau^{-1}(m_\iota)} - x_{m_\iota})^2.
\end{align*}
We can combine this with our previous bound to get the following
\begin{align}
	&-\frac{1}{2t}\mathbbm{1}_{\{ \Gamma_{\mu_\iota ,x,y}\neq \R\}} (y_{\mu_\iota} - x_{\tau(\mu_\iota)})^2 -\frac{1}{2t}\mathbbm{1}_{\{ \Gamma_{\nu_\iota ,x,y}\neq \R\}} (y_{\nu_\iota} - x_{\tau(\nu_\iota)})^2\nonumber\\
	&\qquad -\frac{1}{2t}\mathbbm{1}_{\{ \Gamma_{\tau^{-1}(m_\iota) ,x,y}\neq \R, \ m_\iota\neq \tau(\mu_\iota)\}} (y_{\tau^{-1}(m_\iota)} - x_{m_\iota})^2  \nonumber\\
	\leq& -\frac{1}{4t}\left((y_{\mu_\iota} - \chi^\iota)^2 + (y_{\nu_\iota} - \chi^\iota)^2 + (x_{\tau(\mu_\iota)}- x_{\tau(\nu_\iota)})^2 \right)\nonumber\\
	&\qquad- \frac{1}{2t}\mathbbm{1}_{\{ y_{\tau^{-1}(m_\iota)} \geq x_{m_\iota}, \ m_\iota\neq \tau(\mu_\iota) \}}(y_{\tau^{-1}(m_\iota)} - x_{m_\iota})^2.\label{IC: Shifted exponent}
\end{align}
We aim to show this is bounded above, for some positive constants $C_1, C_2$, by
\begin{equation*}
	-\frac{C_1}{t}\left((y_{\mu_\iota} - \chi^\iota)^2 + (y_{\nu_\iota} - \chi^\iota)^2 + (x_{\tau(\mu_\iota)}- x_{\tau(\nu_\iota)})^2 \right)- \frac{C_2}{t}(x_{m_\iota} - \chi^\iota)^2.
\end{equation*}
Thus we consider the various cases for the indicator in (\ref{IC: Shifted exponent}). \par{}
If $m_\iota= \tau(\mu_\iota)$ then it follows from $x_{m_\iota}\leq \chi_\iota \leq x_{\tau(\nu_\iota)}$ that $(x_{\tau(\mu_\iota)}- x_{\tau(\nu_\iota)})^2 \leq (x_{\tau(m_\iota)}- \chi^\iota)^2$, so that our desired bound is easily seen. \par{}
In the case that $m_\iota\neq \tau(\mu_\iota)$ and $y_{\tau^{-1}(m_\iota)} \geq x_{m_\iota}$, if we further assume $y_{\tau^{-1}(m_\iota)} \geq \chi^\iota$ then it follows
\begin{align*}
	&-(y_{\mu_\iota} - \chi^\iota)^2 - (y_{\tau^{-1}(m_\iota)} - x_{m_\iota})^2\\
	=& -(y_{\mu_\iota} - y_{\tau^{-1}(m_\iota)})^2 - (\chi^\iota - x_{m_\iota})^2 + (y_{\mu_\iota} - x_{m_\iota})(\chi^\iota - y_{\mu_\iota})\\
	\leq& -(y_{\mu_\iota} - y_{\tau^{-1}(m_\iota)})^2 - (\chi^\iota - x_{m_\iota})^2 \leq  - (\chi^\iota - x_{m_\iota})^2.
\end{align*}
Where the last line is true because $y\in\Weyl{n}$ and therefore $y_{\mu_\iota}\geq y_{\tau^{-1}(m_\iota)}$, so that our assumptions imply the last term is negative. If instead we have $x_{m_\iota} \leq y_{\tau^{-1}(m_\iota)} < \chi^\iota$ then $y\in \Weyl{n}$ implies that $y_{\nu_\iota} \leq y_{\tau^{-1}(m_\iota)}$ and thus $0> y_{\tau^{-1}(m_\iota)} - \chi^\iota \geq y_{\nu_\iota} - \chi^\iota$. Hence
\begin{align*}
	&-(y_{\nu_\iota} - \chi^\iota)^2 - (y_{\tau^{-1}(m_\iota)} - x_{m_\iota})^2\\
	\leq& -(y_{\tau^{-1}(m_\iota)} - \chi^\iota)^2 - (y_{\tau^{-1}(m_\iota)} - x_{m_\iota})^2\\
	=& -2(y_{\tau^{-1}(m_\iota)} - \frac{1}{2}(x_{m_\iota} + \chi^\iota))^2 - \frac{1}{2}(x_{m_\iota} - \chi^\iota)^2 \leq  - \frac{1}{2}(x_{m_\iota} - \chi^\iota)^2.
\end{align*}
Hence when $y_{\tau^{-1}(m_\iota)} \geq x_{m_\iota}$ we have the bound on (\ref{IC: Shifted exponent})
\begin{align}
	- \frac{1}{8t}(x_{m_\iota} - \chi^\iota)^2. \label{IC: shifted exponent for m}
\end{align}
If instead we have $m_\iota\neq \tau(\mu_\iota)$ and $y_{\tau^{-1}(m_\iota)} < x_{m_\iota}$ then we have $x,y \in \Weyl{n}$ and therefore $y_{\nu_\iota} \leq y_{\tau^{-1}(m_\iota)} < x_{m_\iota} \leq \chi^\iota$. Thus $-(y_{\nu_\iota} - \chi^\iota)^2 \leq -(x_{m_\iota} - \chi^\iota)^2$ so that (\ref{IC: shifted exponent for m}) expression is a bound on (\ref{IC: Shifted exponent}) for any $x,y \in \Weyl{n}$, as desired. Following the same steps for $l_\iota$ we get the analogous bound
\begin{align}
	&-\frac{1}{2t}\mathbbm{1}_{\{ \Gamma_{\mu_\iota ,x,y}\neq \R\}} (y_{\mu_\iota} - x_{\tau(\mu_\iota)})^2 -\frac{1}{2t}\mathbbm{1}_{\{ \Gamma_{\nu_\iota ,x,y}\neq \R\}} (y_{\nu_\iota} - x_{\tau(\nu_\iota)})^2\nonumber\\
	&\qquad -\frac{1}{2t}\mathbbm{1}_{\{ \Gamma_{\tau^{-1}(l_\iota) ,x,y}\neq \R, \ l_\iota\neq \tau(\nu_\iota)\}} (y_{\tau^{-1}(l_\iota)} - x_{l_\iota})^2  \nonumber\\
	\leq& -\frac{1}{4t}\left((y_{\mu_\iota} - \chi^\iota)^2 + (y_{\nu_\iota} - \chi^\iota)^2 + (x_{\tau(\mu_\iota)}- x_{\tau(\nu_\iota)})^2 \right)\nonumber\\
	&\qquad- \frac{1}{2t}\mathbbm{1}_{\{ y_{\tau^{-1}(l_\iota)} \geq x_{l_\iota}, \ l_\iota\neq \tau(\nu_\iota) \}}(y_{\tau^{-1}(m_\iota)} - x_{m_\iota})^2 \nonumber\\
	\leq& -\frac{1}{8t}(x_{l_\iota}- \chi^\iota)^2. \label{IC: shifted exponent for l}
\end{align}
Combining the bounds in (\ref{IC: Shifted exponent}), (\ref{IC: shifted exponent for m}), and (\ref{IC: shifted exponent for l}) we get the following bound, when $\nu_\iota> \mu_\iota$,
\begin{align}
		&-\frac{1}{2t}\mathbbm{1}_{\{ \Gamma_{\mu_\iota ,x,y}\neq \R\}} (y_{\mu_\iota} - x_{\tau(\mu_\iota)})^2 -\frac{1}{2t}\mathbbm{1}_{\{ \Gamma_{\nu_\iota ,x,y}\neq \R\}} (y_{\nu_\iota} - x_{\tau(\nu_\iota)})^2\nonumber\\
		& -\frac{1}{2t}\mathbbm{1}_{\{ \Gamma_{\tau^{-1}(m_\iota) ,x,y}\neq \R, \ m_\iota \neq \tau(\mu_\iota)\}} (y_{\tau^{-1}(m_\iota)} - x_{m_\iota})^2 -\frac{1}{2t}\mathbbm{1}_{\{ \Gamma_{\tau^{-1}(l_\iota) ,x,y}\neq \R, \ l_\iota \neq \tau(\nu_\iota)\}} (y_{\tau^{-1}(l_\iota)} - x_{l_\iota})^2\nonumber \\
		\leq& -\frac{1}{12t}\left((y_{\mu_\iota} - \chi^\iota)^2 + (y_{\nu_\iota} - \chi^\iota)^2 \right)- \frac{1}{24t}\left((x_{m_\iota} - \chi^\iota)^2 + (x_{l_\iota} - \chi^\iota)^2 \right) \nonumber\\
		\leq& -\frac{1}{12(\overline{\pi_{b_\iota}}- \underline{\pi_{a_\iota}})t}\sum_{\alpha = \underline{\pi_{a_\iota}}}^{\overline{\pi_{b_\iota}}} (y_\alpha - \chi^\iota)^2 - \frac{1}{24t}\left((x_{m_\iota} - \chi^\iota)^2 + (x_{l_\iota} - \chi^\iota)^2 \right).\label{IC: Exponent bound}
\end{align}
Where for the last line we have used that from by definition $\mu_\iota= \overline{\pi_{a_\iota}}$ and $\nu_\iota= \underline{\pi_{b_\iota}}$ and that under $\lambda^\pi$ we have that for any $\pi_j\in \pi$ if $\alpha, \beta \in \pi_j$ then $y_{\alpha}= y_\beta$ a.e. as well as having that $y\in \Weyl{n}$ and therefore $(y_{\mu_\iota}- \chi^\iota) \geq (y_\alpha- \chi_\iota) \geq (y_{\nu_\iota} - \chi_\iota)$ for all $\underline{\pi_{a_\iota}}\leq \alpha \leq \overline{\pi_{b_\iota}}$, and thus either $-(y_{\alpha}- \chi^\iota)^2 \leq -(y_{\mu_\iota}- \chi^\iota)^2$ or $-(y_{\alpha}- \chi^\iota)^2 \leq -(y_{\nu_\iota}- \chi^\iota)^2$.

Before we use this to get the bound on (\ref{IC: PreContourShift}) we need to deal with the second case: $\nu_\iota \leq \mu_\iota$.  \par{}

In the second case $\mu_\iota$ and $\nu_\iota$ are both in $\pi_\iota$, and therefore under $\lambda^\pi$ we have $y_{\mu_\iota} = y_{\nu_\iota}$ almost everywhere. Further since $\tau$ is increasing on every element of $\pi$ it follows that $m_{\iota}:= \tau(\mu_\iota) = \sup\{\tau(\alpha)| \ \nu_\iota \leq \alpha \leq \mu_\iota\}$ and $l_\iota:= \tau(\nu_\iota) = \inf\{\tau(\beta)| \ \nu_\iota \leq \beta \leq \mu_\iota \}$. Following the same steps as before if we assume both $y_{\nu_\iota}> x_{\tau(\nu_\iota)}$ and $y_{\mu_\iota}< x_{\tau(\mu_\iota)}$ then since $y_{\mu_\iota} = y_{\nu_\iota}$ it follows that $x_{\tau(\nu_\iota)}< x_{\tau(\mu_\iota)}$, which is a contradiction since $\tau(\nu_\iota)< \tau(\mu_\iota)$ and $x\in \Weyl{n}$. Thus at least one of $y_{\nu_\iota}\leq x_{\tau(\nu_\iota)}$ and $y_{\mu_\iota}\geq x_{\tau(\mu_\iota)}$ must hold for all $x, y \in \Weyl{n}$. With similar ideas to those used above we find
\begin{align}
	&-\frac{1}{2t}\mathbbm{1}_{\{ \Gamma_{\mu_\iota ,x,y}\neq \R\}} (y_{\mu_\iota} - x_{\tau(\mu_\iota)})^2 -\frac{1}{2t}\mathbbm{1}_{\{ \Gamma_{\nu_\iota ,x,y}\neq \R\}} (y_{\nu_\iota} - x_{\tau(\nu_\iota)})^2\nonumber\\
	\leq&-\frac{1}{4t}\left((y_{\mu_\iota} - \chi^\iota)^2 + (y_{\nu_\iota} - \chi^\iota)^2 + (x_{m_\iota}- x_{l_\iota})^2 \right)\nonumber \\
	\leq& -\frac{1}{12t}\left((y_{\mu_\iota} - \chi^\iota)^2 + (y_{\nu_\iota} - \chi^\iota)^2 \right)- \frac{1}{24t}\left((x_{m_\iota} - \chi^\iota)^2 + (x_{l_\iota} - \chi^\iota)^2 \right)\\
	\leq& -\frac{1}{12(\mu_\iota - \nu_\iota)t}\sum_{\alpha = \nu_\iota}^{\mu_\iota} (y_\alpha - \chi^\iota)^2 - \frac{1}{24t}\left((x_{m_\iota} - \chi^\iota)^2 + (x_{l_\iota} - \chi^\iota)^2 \right). \label{IC: Exponent bound 2}
\end{align}
Where the idea behind the bounds is similar, but this time we use $y_{\mu_\iota} = y_{\nu_\iota}$, and we used that $x_{l_\iota}\geq \chi^\iota \geq x_{m_\iota}$ for the second inequality. The constants appearing in the denominator have been chosen to be consistent with (\ref{IC: Exponent bound}), and so are not optimal. \par{}
Applying the bounds (\ref{IC: Exponent bound}) and (\ref{IC: Exponent bound 2}) to (\ref{IC: Contour shifted k integral}) leads to the following inequality
\begin{align}
	\left|\int_{\R^n} e^{-\frac{1}{2}t|k|^2+ik_\tau \cdot (y-x_\tau)} T^{\tau,\pi}(k) dk\right|\leq& e^{-\frac{1}{12nt}|y- \chi|^2} \left(\prod_{\pi_\iota \in \pi} e^{- \frac{1}{24nt}\left((x_{m_\iota} - \chi^\iota)^2 + (x_{l_\iota} - \chi^\iota)^2 \right)}\right)\nonumber \\
	& \quad \int_{\times_{\alpha=1}^n \Gamma_{\alpha, x, y}} e^{-\frac{1}{2}t\sum_{\alpha=1}^n \Re(k_{\tau(\alpha)})^2} \left| T^{\tau,\pi}(k) \right| dk. \label{IC: first k integral bound}
\end{align}
Where we have used $\mu_\iota- \nu_\iota, \overline{\pi_{b_\iota}}- \underline{\pi_{a_\iota}}<n$ for all $\iota$ to get the form of the Gaussian bound given above.
\end{proof}

We complete the proof of Proposition \ref{Prop: k integral bound} with the following lemma.
\begin{lemma}\label{Lemma: shifted k integral bound}
	There is a constant $C>0$, depending only on $\pi$ and $n$, such that
	\begin{align}
		\int_{\times_{\alpha=1}^n \Gamma_{\alpha, x, y}} e^{-\frac{1}{2}t\sum_{\alpha=1}^n \Re(k_{\tau(\alpha)})^2} \left| T^{\tau,\pi}(k) \right| dk \leq  C t^{-\frac{1}{2}|\pi|} |\log(t)|^{|\pi|}.
	\end{align}
\end{lemma}
\begin{proof}
Now we bound what the $k$ integral in the above expression, for which we need to collect some bounds on the factors appearing in the products (\ref{IC: Two products}). We need to make sure the bound covers the new contours, therefore it is sufficient to bound for $k\in E^{\tau, \pi}$. This can be done for the factors in the first product by bounding for all $h_a,h_b\geq 0$ and $k_a, k_b \in \R$
\begin{align}
	&\left| \frac{i\theta ( (k_a + ih_a) -  (k_b - i h_b)) + ( (k_a + ih_a))( (k_b -ih_b))}{i\theta ( (k_a + ih_a) -  (k_b - i h_b)) - ( (k_a + ih_a))( (k_b -ih_b))} \right| \nonumber\\
	=& \left| \frac{i \theta  (k_a - k_b) - \theta  (h_b + h_a) +i(k_b h_a - k_a h_b) + k_a k_b + h_a h_b }{i \theta  (k_a - k_b) - \theta  (h_b + h_a) -i(k_b h_a - k_a h_b) - k_a k_b - h_a h_b } \right| \nonumber\\
	=& \left( \tfrac{\theta^2 (k_a - k_b)^2 + (k_b h_a- k_a h_b )^2 -2\theta (k_b^2 h_a + k_a^2 h_b) + \theta^2 (h_b +h_a)^2 -2\theta h_a h_b(h_b + h_a) +(k_a k_b + h_a h_b)^2}{\theta^2 (k_a - k_b)^2 + (k_b h_a- k_a h_b )^2 +2\theta (k_b^2 h_a + k_a^2 h_b) + \theta^2 (h_b +h_a)^2 +2\theta h_a h_b (h_b +h_a) +(k_a k_b + h_a h_b)^2} \right)^{\frac{1}{2}}\nonumber \\
	\leq & 1, \ \text{ because } h_a,h_b\geq 0. \label{InitialCondition:FirstProductBound}
\end{align}
Here the $k$ variables are the real part of the integration variables, and the $h$ variables are the imaginary part. Hence we have that for all $k\in E^{\tau, \pi}$
\begin{align}
	&\int_{\times_{\alpha=1}^n \Gamma_{\alpha, x, y}} e^{-\frac{1}{2}t\sum_{\alpha=1}^n \Re(k_{\tau(\alpha)})^2} \left| T^{\tau,\pi}(k) \right| dk\nonumber\\
	\leq& \prod_{\pi_\iota \in \pi} \int_{\times_{\alpha \in \pi_\iota} \Gamma_{\alpha, x, y}} e^{-\frac{1}{2}t\sum_{\alpha\in \pi_\iota} \Re(k_{\tau(\alpha)})^2}\prod_{\substack{\alpha< \beta: \\ \alpha, \beta \in \pi_\iota}} \left| \tfrac{i\theta(k_{\tau(\beta)} -k_{\tau(\alpha)} )}{i\theta(k_{\tau(\beta)} - k_{\tau(\alpha)}) - k_{\tau(\alpha)} k_{\tau(\beta)}} \right|dk. \label{IC: Product of k integrals}
\end{align}
Similar to the previous argument it suffices to bound for $k_a, k_b\in \R$ and $h_a, h_b \geq 0$
\begin{align}
	&\left| \frac{i\theta ( (k_a + ih_a) -  (k_b - i h_b))}{i\theta ( (k_a + ih_a) -  (k_b - i h_b)) - ( (k_a + ih_a))( (k_b -ih_b))} \right| \nonumber\\
	=& \left( \tfrac{\theta^2 (k_a - k_b)^2 + \theta^2 (h_b + h_a)^2}{ \theta^2 (k_a - k_b)^2 + (k_b h_a- k_a h_b )^2 +2\theta(k_b^2 h_a + k_a^2 h_b) + \theta^2 (h_b +h_a)^2 +2\theta h_a h_b (h_b +h_a) +(k_a k_b + h_a h_b)^2 } \right)^{\frac{1}{2}} \nonumber \\
	\leq& \begin{cases}
		1, \\
		\theta \left( \frac{1}{|k_a|} + \frac{1}{|k_b|} \right) + \theta \left( \frac{|h_b + h_a|}{\left( (k_b h_a - k_ah_b)^2 + (k_a k_b + h_a h_b)^2 \right)^{1/2}} \right)
	\end{cases}\nonumber \\
	\leq&\begin{cases}
		1, \\
		2\theta \left( \frac{1}{|k_a|} + \frac{1}{|k_b|} \right).
	\end{cases} \label{InitialCondition:SecondProductBound}
\end{align}
Where the last line follows by expanding the brackets in the denominator, removing some non-negative terms, and then applying the triangle inequality. \par{}
Now we can parametrise the contour integrals in (\ref{IC: Product of k integrals}) as integrals over the real line (as the imaginary part of the contour is always constant) and, for convenience, relabel $k_{\tau(\alpha)}$ as $k_\alpha$. We then split each of these integral into regions where each $k_\alpha$ either satisfies $|k_\alpha|<\epsilon/\sqrt{t}$ or $|k_{\alpha}|\geq \epsilon/\sqrt{t}$, there are $2^n$ such regions. However we can simplify as follows, whenever we have $|k_\alpha|<\epsilon/\sqrt{t}$, all factors in the products depending on $k_\alpha$ can be bounded by $1$ using first line bound in (\ref{InitialCondition:SecondProductBound}), and so we can pull out the $k_\alpha$ integral and bound by $2\epsilon/\sqrt{t}$. If we relabel the remaining integration variables we find that what remains depends only on the number of $k_\alpha$ for which $|k_\alpha|\geq \epsilon/\sqrt{t}$. Bounding the product in these integrals using the second line bound from (\ref{InitialCondition:SecondProductBound}) gives us the following bound on (\ref{IC: Product of k integrals}).
\begin{align*}
	\prod_{\pi_\iota \in \pi}\sum_{j=1}^{|\pi_\iota|} \binom{|\pi_\iota|}{j} 2^{|\pi_\iota|-j + \binom{j}{2}}\theta^{\binom{j}{2}}(\epsilon/\sqrt{t})^{|\pi_\iota|-j} \int_{\substack{\R^j:\\ |k_\alpha|\geq \epsilon/\sqrt{t}, \ \forall \alpha}} e^{-\frac{1}{2}t|k|^2} \prod_{\alpha < \beta} \left( \frac{1}{|k_\alpha|} + \frac{1}{|k_\beta|} \right) dk.
\end{align*}
Rescaling the $k$ variables by $\frac{1}{\sqrt{t}}$ we see that this equals
\begin{align}
	\prod_{\pi_\iota \in \pi}\sum_{j=1}^{|\pi_\iota|} \binom{|\pi_\iota|}{j} 2^{|\pi_\iota|-j + \binom{j}{2}}\theta^{\binom{j}{2}}\epsilon^{|\pi_\iota|-j} t^{\frac{1}{2}\left(\binom{j}{2} - j\right)} \int_{\substack{\R^j:\\ |k_\alpha|\geq \epsilon, \ \forall \alpha}} e^{-\frac{1}{2}|k|^2} \prod_{\alpha < \beta} \left( \frac{1}{|k_\alpha|} + \frac{1}{|k_\beta|} \right) dk. \label{IC:kint}
\end{align}
Since the product runs through all pairs of $\alpha, \beta \in \{1,...,j\}$, upon expanding the brackets every term will involve at least $j-1$ of the $k_{\gamma}$, further at most one has exponent $-1$, with the rest having exponent at most $-2$. It is clear from repeated integration by parts,  that for each $y\neq 1$ there is some constant $C>0$
\begin{equation*}
	\int_{|x|\geq \epsilon} \frac{1}{|x|^y} e^{-\frac{1}{2}|x|^2} dx \leq C \epsilon^{1-y}, \quad \text{ when } \epsilon \in (0,1).
\end{equation*}
For $y=1$ we instead have that there is a constant $C>0$ such that
\begin{align*}
	\int_{|x|\geq \epsilon} \frac{1}{|x|} e^{-\frac{1}{2}|x|^2} dx \leq C |\log(\epsilon)|, \quad \text{ when } \epsilon \in (0,1).
\end{align*}
Hence, since the sum of all the powers of all the $k_\gamma$ in each term of the expanded brackets is $\binom{j}{2}$, and because the product runs through all pairs of indices so that in each term in the expansion there can be at most one $k_\gamma$ appearing with power $1$, there is some constant $C>0$, depending only on $n$ and $\pi$, such that for all $\epsilon \in (0,1)$ (\ref{IC:kint}) is bounded by
\begin{align*}
	\leq C\prod_{\pi_\iota \in \pi}\sum_{j=1}^{|\pi_\iota|} \epsilon^{|\pi_\iota|- j + j -1 - \binom{j}{2}} |\log(\epsilon)| t^{\frac{1}{2}\left(\binom{j}{2} - j\right)}.
\end{align*}
Which, if we set $\epsilon = \sqrt{t}$ is clearly bounded above by
\begin{align*}
	C t^{-\frac{1}{2}|\pi|} |\log(t)|^{|\pi|}.
\end{align*}
Which is the desired upper bound.
\end{proof}
\begin{proof}[Proof of Proposition \ref{Prop: k integral bound}]
	Combining the bounds from the above lemma and Lemma \ref{Lemma: intermediate k bound} proves the statement.
\end{proof}
We now apply this bound to complete the proof of the main proposition of the subsection.
\begin{proof}[Proof of Proposition \ref{Proposition: u satisfies initial condition}]
Proposition \ref{Prop: k integral bound} implies that (\ref{IC: PreContourShift}) is bounded above by
\begin{align}
	C t^{-\frac{1}{2}|\pi|}|\log(t)|^{|\pi|}& \sum_{\substack{\tau\in S_n: \\ \tau|_{\pi_\iota} \text{ is increasing } \forall \iota}}\nonumber\\
	& \int e^{\frac{1}{12nt}|y-\chi|^2}\prod_{\pi_\iota \in \pi} e^{-\frac{1}{24nt}\left((x_{m_\iota} - \chi^\iota)^2 + (x_{l_\iota} - \chi^\iota)^2 \right)} |f(y) - f(x)| \lambda^\pi(dy) \label{IC: Pre measure breakdown}
\end{align}
We can replace the function $f: \Weyl{n}\to \R$ with its symmetric extension $\overline{f}:\R^n\to \R$, that is the function $\overline{f}:\R^n\to \R$ such that for any $\sigma \in S_n$, $x\in \R^n$ we have $\overline{f}(x_\sigma) = \overline{f}(x)$ and $\overline{f}|_{\Weyl{n}} = f$, then (\ref{IC: Pre measure breakdown}) is bounded by
\begin{align*}	
	&C |\log(t)|^{|\pi|} \sum_{\substack{\tau\in S_n: \\ \tau|_{\pi_\iota} \text{ is increasing } \forall \iota}}\\
	&\quad \int_{\Weyl{|\pi|}} e^{-\frac{1}{12nt}|y|^2}\left(\prod_{\pi_\iota \in \pi} e^{-\frac{1}{24nt}\left((x_{m_\iota} - \chi^\iota)^2 + (x_{l_\iota} - \chi^\iota)^2 \right) } \right) |\overline{f}(\sqrt{t}\underline{y} + \chi) - \overline{f}(x_\tau)|dy.
\end{align*}
Where $\chi\in \R^n$ is defined by $\chi_\alpha := \chi^\iota$ when $\alpha\in \pi_\iota$, $\underline{y}$ is defined by $\underline{y}_\alpha=y_\iota$ for all $\alpha\in \pi_\iota$, and we have used $\overline{f}(x)=\overline{f}(x_\tau)$. We have also rewritten the integral with respect to $\lambda^\pi$ as an integral with respect to the Lebesgue measure. Since $f$ is a Lipschitz function, it is straightforward to show that $\overline{f}$ is also Lipschitz therefore the above expression is bounded by
\begin{align*}
	&C t^{-\frac{1}{2}|\pi|} |\log(t)|^{|\pi|} \sum_{\substack{\tau\in S_n: \\ \tau|_{\pi_\iota} \text{ is increasing } \forall \iota}}\\
	&\quad \int_{\Weyl{|\pi|}} e^{-\frac{1}{12nt}|y|^2} \left(\prod_{\pi_\iota \in \pi} e^{-\frac{1}{24nt}\left((x_{m_\iota} - \chi^\iota)^2 + (x_{l_\iota} - \chi^\iota)^2 \right)} \right) \left(\sqrt{t}|\underline{y}| + |\chi-x_\tau|\right) dy.
\end{align*}
The integrand is non negative, and $|\underline{y}| \leq |\pi||y|$, therefore this is bounded above (for a new constant $C$) by 
\begin{align}
	&C t^{-\frac{1}{2}|\pi|} |\log(t)|^{|\pi|} \sum_{\substack{\tau\in S_n: \\ \tau|_{\pi_\iota} \text{ is increasing } \forall \iota}}\nonumber\\
	&\quad \int_{\R^{|\pi|}} e^{-\frac{1}{12nt}|y|^2} \left(\prod_{\pi_\iota \in \pi} e^{-\frac{1}{24nt}\left((x_{m_\iota} - \chi^\iota)^2 + (x_{l_\iota} - \chi^\iota)^2 \right)} \right) \left(\sqrt{t}|\underline{y}| + |\chi-x_\tau|\right) dy\nonumber \\
	\leq& C|\log(t)|^{|\pi|} \sum_{\substack{\tau\in S_n: \\ \tau|_{\pi_\iota} \text{ is increasing } \forall \iota}} \\
	&\quad \int_{\R^{|\pi|}} (|y|+1)e^{-\frac{1}{12n}|y|^2}dy\left(\sqrt{t} + |\chi - x_\tau|e^{-\frac{1}{24nt}\sum_{\pi_\iota \in \pi}\left((x_{m_\iota} - \chi^\iota)^2 + (x_{l_\iota} - \chi^\iota)^2 \right)}\right)\label{IC: exp bound to do}.
\end{align}
Now we note that $|\chi - x_\tau| \leq \sum_{\alpha=1}^n |\chi_\alpha - x_{\tau(\alpha)}|$, but for all $\alpha\in [\mu_\iota, \nu_\iota]$ (or $[\nu_\iota, \mu_\iota]$) we have $x_{m_\iota}\leq x_{\tau(\alpha)}, \chi_\alpha \leq x_{l_\iota}$ (or $x_{l_\iota}\leq x_{\tau(\alpha)}, \chi_\alpha \leq x_{m_\iota}$). Hence either $|\chi_\alpha - x_{\tau(\alpha)}|\leq |\chi_\alpha - x_{m_\iota}|$ or $|\chi_\alpha - x_{\tau(\alpha)}|\leq |\chi_\alpha - x_{l_\iota}|$. Note that that for any $c>0$ and $x\in \R$ we have the inequality $|x|e^{-c|x|^2} \leq (2ec)^{-\frac{1}{2}}$. Hence (\ref{IC: exp bound to do}) is bounded by
\begin{align*}
	&C\sqrt{t}|\log(t)|^{|\pi|} \sum_{\substack{\tau\in S_n: \\ \tau|_{\pi_\iota} \text{ is increasing } \forall \iota}} \int_{\R^{|\pi|}} (|y|+1)e^{-\frac{1}{12n}|y|^2}dy \\
	\leq& C\sqrt{t}|\log(t)|^{|\pi|}.
\end{align*}
Where we have bounded the integral independently of $|\pi|$, and the constant $C$ has changed between lines. Summing over $\pi\in \Pi_n$, and using that since $\Pi_n$ is a finite set the constants $C$ in the above expression have a finite maximum, we get for a new constant $C>0$ depending only on $n$
\begin{align*}
	\sup_{x\in \Weyl{n}} |\int u_t(x,y) f(y) m^{(n)}_\theta(dy) - f(x)| \leq&  C\sqrt{t} \sum_{\pi \in\Pi_n} |\log(t)|^{|\pi|} \\
	\leq& C\sqrt{t} \log(t)^n \to 0, \ \text{ as } t \to 0.
\end{align*}
Where the last inequality is valid for $t< 1/e$. Hence we have the desired uniform convergence, and Proposition \ref{Proposition: u satisfies initial condition} is proven. 
\end{proof}

As a consequence of Proposition \ref{Proposition: u satisfies boundary conditions} and Proposition \ref{Proposition: u satisfies initial condition} we can apply Proposition \ref{Prop:TheBackwardsEquation} to our function $\int u_t(x,y) f(y) m^{(n)}_\theta (dy)$, to prove $\int u_{t-s}(Y_s, y) f(y) m^{(n)}_\theta (dy)$ is a local martingale. Suppose that $f\in C^\infty_c(\Weyl{n})$, i.e. $f$ has an extension to an open set $U$ containing $\Weyl{n}$ that is smooth and compactly supported. Then since $\int u_t(x,y) f(y) m^{(n)}_\theta (dy)$ converges uniformly to $f$ as $t\to 0$, and $f$ is bounded, there must be some $\epsilon>0$ such that $\int u_t(x,y) f(y) m^{(n)}_\theta (dy)$ is bounded for $t\in[0, \epsilon]$ and $x\in \Weyl{n}$. We also have $|\int u_t(x,y) f(y) m^{(n)}_\theta (dy)| \leq \frac{1}{(2\pi t)^{n/2}} \int |f(y)| m^{(n)}_\theta (dy)$ which is bounded for $t\in[\epsilon, \infty)$. Hence $\int u_t(x,y) f(y) m^{(n)}_\theta (dy)$ is bounded as a function of $(t, x) \in \R_{>0} \times \Weyl{n}$. It follows that $\int u_{t-s}(Y_s, y) f(y) m^{(n)}_\theta (dy)$ is a true martingale and thus $\E_x[f(Y_t)] = \int u_t(x,y) f(y) m^{(n)}_\theta (dy)$. In particular if $f(x)\geq 0$ for all $x\in \Weyl{n}$ then $\int u_t(x,y) f(y) m^{(n)}_\theta (dy) \geq 0$. Since this holds for every $f\in C^\infty_c(\Weyl{n})$ we have that for each $t>0, x\in \Weyl{n}$ $u_t(x,y) \geq 0$ $m^{(n)}_\theta$ almost everywhere. \par{}

Returning to the case where $f$ is merely bounded and Lipschitz,  we can use the non-negativity of $u_t(x,y)$, and Lemma \ref{Lemma: u integrates to one}, to get $|\int u_t(x,y) f(y) m^{(n)}_\theta(dy)| \leq \|f\|_\infty$. Hence the local martingale $\int u_{t-s}(Y_t, y) f(y) m^{(n)}_\theta (dy)$ is in fact a true martingale for $s\in [0,t]$, and so $\E_x[f(Y_t)] = \int u_t(x,y) f(y)m^{(n)}_\theta (dy)$. This proves Theorem \ref{transition probabilities}. \qed

As a consequence we can also prove the following.
\begin{theorem}
	$m^{(n)}_\theta$ is a stationary measure for $Y$, and $Y$ is reversible with respect to $m^{(n)}_\theta$.
\end{theorem}
\begin{proof}
	For $f$ a bounded, integrable, Lipschitz continuous function, we have for all $t>0$
	\begin{align}
		\frac{d}{dt} \int \E_x[f(Y_t)] m^{(n)}_\theta (dx) &=\frac{d}{dt} \int \int u_t(x,y) f(y) m^{(n)}_\theta (dx) m^{(n)}_\theta (dx) \\
		&=0.
	\end{align}
	With the second equality a consequence of Corollary \ref{StationarityOfm}, and Fubini's theorem. The necessary bounds to pass the derivatives through the integrals and then apply Fubini follow in the same way as Lemma \ref{uIsIntegrable}. The same bounds, together with the uniform convergence we just proved gives
	\begin{align}
		\lim_{t\to 0} \int \E_x[f(Y_t)] m^{(n)}_\theta (dx) = \int f(x) m^{(n)}_\theta (dx).
	\end{align}
	We can extend this to any $L^1(m^{(n)}_\theta)$ function by a density argument, proving that $m^{(n)}_\theta$ is the stationary measure for $Y$. \par{}
If $f$ and $g$ are bounded, Lipschitz continuous, and integrable; Fubini's theorem gives
\begin{align*}
	\int \E_x[f(Y_t)] g(x) m^{(n)}_\theta(dx) &= \int \int u_t(x,y) f(y)m^{(n)}_\theta (dy) g(x)m^{(n)}_\theta (dx) \\
	&= \int \int u_t(x,y) g(x) m^{(n)}_\theta (dx)f(y) m^{(n)}_\theta (dy) \\
	&= \int \E_y[g(Y_t)]f(y) m^{(n)}_\theta(dy).
\end{align*}
Where we have used the symmetry $u_t(x,y)= u_t(y, x)$ in the last line. Hence $Y$ is reversible with respect to $m^{(n)}_\theta$.
\end{proof}

\section{Stochastic Flows of Kernels}\label{Stochastic flows of kernels}
\subsection{Random Walks in Random Environments}\label{Section:RWRE}
We'll begin by introducing the discrete counterparts of Howitt-Warren flows and sticky Brownian motions: Random walks in space-time i.i.d. random environments and their $n$-point motions. A random walk in a random environment is simply a random walk whose transition probabilities are themselves random variables. We define the \textit{Random Environment} as a family of i.i.d $[0,1]$ valued random variables $\boldsymbol\omega=(\omega_{t,x})_{t,x \in \Z}$, with law and expectation $\p$, and $\E$ respectively. We then define a random walk running through realisation of the environment with transition probabilities:
\begin{gather*}
	P^{\boldsymbol\omega}(X(t+1)= x+1| \ X(t)=x)= \omega_{x,t};\\
	P^{\boldsymbol\omega}(X(t+1)= x-1| \ X(t)=x)= 1-\omega_{x,t}.
\end{gather*}
Where $P^{\boldsymbol\omega}$ denotes the law of the RWRE, and $E^{\boldsymbol\omega}$ its expectation, both of which depend on the environment. By considering the random transition probabilities $P^{\omega}(X_t=y| \ X_0=x)$ we can also consider this model as a random flow of a fluid, where the quantities describe how a point mass at $x$ is spread through the fluid at time $t$.\par

An important idea for studying such models are the \textit{n-point motions}, we run $n$ random walks independently through a sampling of the environment, and then average out the environment, this will break the particles' independence. That is, if $X(t)=(X^1(t),...,X^n(t))$ is the $n$-point motion, then
\begin{gather*}
	\p(X(t+1)=y|X(t)=x)= \E\left[\prod_{i=1}^{n}P^{\boldsymbol\omega}(X^i(t+1)=y_i|X(t)=x_i)\right].
\end{gather*}
Alternatively we can view the $n$-point motions as describing the behaviour of $n$ particles thrown into the fluid. Notice now that since the environment is i.i.d, the coordinate processes of the $n$-point motion behave independently when they are apart. However when they meet, they interact, in particular it's a simple consequence of Jensen's inequality that they are more likely to move in the same direction when together than when apart:
\begin{equation*}
	\E[\omega^n] + \E[(1-\omega)^n] \geq \E[\omega]^n + \E[1-\omega]^n,
\end{equation*}
$\omega$ being a copy of an environment variable. A group of particles situated at the same site, $x$, at time $t$ can break into at most two groups. The probability of a group of $n$ particles breaking into two groups of size $k$ and $l$, with the $k$ moving to $x+1$ and the $l$ to $x-1$, is
\begin{gather*}
	\E[\omega_{x,t}^k (1-\omega_{x,t})^l].
\end{gather*}
Hence the distribution of $\omega$ can be viewed as controlling the rate at which groups of particles break up, and the size of the groups they tend to break into. \par{}

If we take the diffusive scaling limit of these $n$-point motions in an environment having a fixed distribution, then the contribution of the environment is overcome in the limit, and we simply end up with independent Brownian motions (assuming the environment variables are mean $1/2$ so there is no drift). \par{}

It was shown by Howitt and Warren \cite{HowittWarren}, see also Schertzer, Sun and Swart \cite{FlowsinWebandNet}, that by changing the distribution of the $\omega$ as we take the diffusive scaling limit, we can obtain Brownian motions which still interact, specifically are sticky when they meet.

\begin{theorem}
Suppose $X(t)$ is the $n$-point motion of a RWRE, where the environment variables have law $\mu^{(\epsilon)}$ satisfying the following:
\begin{gather*}
	\frac{1}{\epsilon} \int_0^1 (1-2q) \mu^{(\epsilon)}(dq) \to \beta, \quad \text{as } \epsilon \to 0;\\
	\frac{1}{\epsilon} q(1-q)\mu^{\epsilon}(dq) \implies \nu(dq), \quad \text{as } \epsilon \to 0.
\end{gather*}
Then the laws of the processes $(\epsilon X(\epsilon^2 t))_{t\geq 0}$ converge weakly to the law of a solution to the Howitt-Warren martingale problem with drift $\beta$ and characteristic measure $\nu$.
\end{theorem}
In the special case of $\nu(dx)= \theta/2 dx$, where $dx$ is the Lebesgue measure the above result shows the solution to the Howitt-Warren martingale problem is the scaling limit of the Beta random walk in a random environment. That is choose $\mu^{(\epsilon)}(dq)=\frac{\Gamma(2 \theta \epsilon)}{\Gamma(\theta \epsilon) \Gamma(\theta \epsilon)} q^{\theta\epsilon-1}(1-q)^{\theta\epsilon-1}dq$, then for any function $C_b([0,1])$ the Dominated Convergence Theorem implies
\begin{align*}
	\frac{1}{\epsilon}\int_0^1 f(q) q(1-q) \mu^{(\epsilon)}(dq) &= \frac{\Gamma(2 \theta \epsilon)}{\epsilon\Gamma(\theta \epsilon) \Gamma(\theta \epsilon)}\int_0^1 f(q) q^{\theta\epsilon}(1-q)^{\theta\epsilon}dq \\
	&\to \frac{\theta}{2} \int_0^1 f(q) dq,
\end{align*}
using $\Gamma(x)= \frac{\Gamma(x+1)}{x}~ \frac{1}{x}$ as $x\to 0$. Hence $\frac{1}{\epsilon} q(1-q) \mu^{(\epsilon)} \Rightarrow \frac{\theta}{2} dx$; since we also have $\int_0^1 (1-2q) \mu^{(\epsilon)}(dq)=0$ for all $\epsilon>0$ the theorem implies the convergence of the $n$-point motions of the Beta random walk in a random environment to solutions of the Howitt-Warren martingale problem with characteristic measure $\frac{\theta}{2}\mathbbm{1}_{[0,1]}dx$ and zero drift. This is the key motivator for looking for exact solutions in the sticky Brownian motion case and was used by Barraquand and Rychnovsky in \cite{barraquand2019large} to find Fredholm determinant expressions in the sticky Brownian motions case by taking limits of those found for the Beta random walk in a random environment in \cite{Barraquand2017}.

\subsection{The Howitt-Warren process}

We first briefly introduce stochastic flows of kernels, these are essentially random transition probabilities $(K_{s,t}(x, dy))_{s\leq t}$, with the following additional assumptions: independent increments in the sense that for any $t_0<...,t_n$ the random kernels $K_{t_0,t_1}, ..., K_{t_{n-1}, t_n}$ are independent; stationarity, that is the law of $K_{s,t}$ depends only on $t-s$. They can be thought of as the continuum version of the random environment that is i.i.d. in space and time we considered in the previous section. 

The $n$-point motions of a stochastic flow of kernels are the family of Markov processes $(X_n)_{n=1}^\infty$ with $X_n$ taking values in $\R^n$ with transition probabilities
\begin{equation*}
	\p(X_{n}(t) \in E| \ X_{n}(s) =x) = \E\Big[\int_E \prod_{i=1}^{n} K_{s,t}(x_i, dy_i) \Big], \quad \text{for } x \in \R^n, \ E\in \mathcal{B}(\R^n).
\end{equation*}
Notice that this is very similar to the definition of the $n$-point motions in the RWRE case, with $K$ taking the place of the random transition probabilities.

Le Jan and Raimond \cite{lejan2004} have shown that any consistent family of Feller processes, are the $n$-point motions of some stochastic flow of kernels. Where a family of Feller processes $(X_n)_{n=1}^\infty$, $X_n:\R_{>0}\to \R^n$ is consistent if for any $k\leq n$ and any choice of $k$ coordinates from $X_n$: $(X^{i_1}_n,...,X^{i_k}_n)$ is equal in law to $X_k$. For a more complete introduction to stochastic flows of kernels we refer to \cite{lejan2004}. When the family of $n$-point motions, $(X_n)_{n=1}^\infty$, are sticky Brownian motions characterised by a Howitt-Warren martingale problem the resulting flow of kernels is called a Howitt-Warren flow. These flows have been studied extensively by Schertzer, Sun, and Swart \cite{FlowsinWebandNet}. \par{}

\begin{definition}\label{Def:HWFlow}
	The stochastic flow of kernels whose $n$-point motions solve the Howitt-Warren martingale problem (\ref{Def:HWMartProb}) with characteristic measure $\nu$ and drift $\beta$ is called the \textit{Howitt-Warren flow} with characteristic measure $\nu$ and drift $\beta$.
\end{definition}
 
Rather than look at the flow directly we want to consider the \textit{Howitt-Warren process} a measure valued process that describes how an initial mass is carried by the flow. In our case we are interested in the the case where all mass starts at the origin, thus we consider the Howitt-Warren process with initial condition $\delta_0$. That is for the Howitt-Warren flow $(K_{s,t})_{s\leq t}$ with characteristic measure $\nu$ and drift $\beta$ we define the Howitt-Warren process started from $\delta_0$ with characteristic measure $\nu$ and drift $\beta$ to be the process
\begin{equation}\label{Howitt-Warren process}
	 \rho_t(A):= K_{0,t}(0, A), \quad \text{for every Borel set } A\subset \R. 
\end{equation}

We have the following corollary of our main result, Theorem \ref{transition probabilities}, that allows us to study the Howitt-Warren process. If $f:\R^n \to \R$ is a symmetric function then $\E_x[f(X(t))]=\E_x[f(Y(t))]$ for all $x\in \Weyl{n}$. Hence
\begin{corollary}\label{Corollary: moments of the flow}
	If $f:\R^n\to \R$ is a symmetric function, and its restriction to $\Weyl{n}$ is a bounded, Lipschitz continuous function, then for a Howitt-Warren flow $(K_{s,t})_{s\leq t}$ with characteristic measure $\frac{\theta}{2} dx$ and drift zero we have
	\begin{equation*}
		\E\left[\int f(y) \prod_{i=1}^n K_{s,t}(x_i, dy_i)\right] = \int u_{t-s}(x,y) f(y) m^{(n)}_\theta(dy) \quad \text{for all } x\in \Weyl{n}.
	\end{equation*}
	From which it clearly follows that we have for the Howitt-Warren process started from $\delta_0$ with characteristic measure $\frac{\theta}{2}\mathbbm{1}_{[0,1]}$ and drift $0$ that
	\begin{equation}
		\E\left[ \int f(y) \rho^{\otimes n}_t(dy) \right] = \int u_t(0, y) f(y) m^{(n)}_\theta(dy).
	\end{equation}
\end{corollary}
This allows us to study the process directly, via $u$, which we will pursue further in the next subsection.\par{}

\subsection{Atoms of the Howitt-Warren process}
Schertzer, Swart, and Sun  proved \cite[Theorem 2.8]{FlowsinWebandNet} that any Howitt-Warren process is almost surely purely atomic for fixed times $t$. Thus, almost surely, we can write the Howitt-Warren process at time $t$ as a linear combination of delta measures $\rho_t(dy) = \sum_{i} w_i \delta_{y_i}(dy)$, where the $w_i$ and $x_i$ are both random. We can think of the collection of pairs $(y_i, w_i)$ as a point process on $\R\times \R_{>0}$. This point process has an associated intensity measure $\gamma_t$ on $\R\times \R_{>0}$ defined by
\begin{align*}
    \gamma_t(A_1\times A_2) = \E\left[ \sum_{i} \mathbbm{1}_{y_i\in A_1,\ w_i \in A_2}\right].
\end{align*}
We will use this intensity to study the behaviour of the weight of a single atom at a given point in space. See \cite{Determinantalpointprocesses} for an introduction to point processes. For any $n\in \N$ and $f:\R\to \R$ that is bounded and Lipschitz continuous we have the equalities
\begin{align}
    \int_{\R\times \R_{>0}} f(y) w^n \ \gamma_t(dy, dw) =& \E\left[ \sum_{i} f(y_i) w_i^n \right]\nonumber\\
    =& \E\left[ \int_{\mathbb{D}^n} f^{\otimes n}(y) \rho_t^{\otimes n}(dy) \right]\nonumber\\
    =& \int_{\mathbb{D}^n}  f^{\otimes n}(y) u^{(n)}_t(0, y) m^{(n)}_\theta(dy)\nonumber\\
    =& n^{-1} \theta^{1-n} \int_{\R} f(y) u^{(n)}_t(0, (y,...,y)) dy \label{Moments of the HW process}.
\end{align}
Above $\mathbb{D}^n:= \{(y,...,y)\in \R^n: \ y\in \R\}$ and we have written $u^{(n)}_t$ for the transition density $u_t$ on $\R^n$, which we do for the rest of the section to indicate the dependency on dimension. The first equality can be seen by approximating by simple functions, the second is direct from the definitions, the third is a consequence of Corollary \ref{Corollary: moments of the flow}, and the fourth from Definition \ref{ReferenceMeasureOfTheTdensity}.

Equality (\ref{Moments of the HW process}) also shows that the measure $\gamma_t(dy, dw)$ can be written in the form $\gamma_t(y, dw) dy$, and that we have for each $n\in \N$ and almost every $y\in \R$ the equality
\begin{align}\label{atomic moments}
    \int_{\R_{>0}} w^n \gamma_t(y, dw) = n^{-1} \theta^{1-n} u^{(n)}_t(0, (y,...,y)).
\end{align}
We will study the asymptotic behaviour of the measure $\gamma_t(y, dw)$ for certain choices of $y$. We can interpret $\gamma_t(y, dw)$ as describing the distribution of the size of an atom at $y$. However $\gamma_t(y, dw)$ is not a probability distribution; the measure of any neighbourhood of $w=0$ is infinite. Introducing size biasing and instead considering the measure $w \gamma_t(y, dw)$ we do get a finite measure which describes the size of an atom, picked at random from $\rho_t$ using the sizes of the atoms as probabilities and conditioning the chosen atom to be at $y$.

\begin{proposition}
For each $x\in \R$ we have as $t\to\infty$
	\begin{equation*}
		\sqrt{2\pi} t^{-\frac{1}{2}} e^{\frac{x^2}{2}}w\gamma_t\left( \sqrt{t} x, \frac{dw}{\sqrt{t}}\right) \Rightarrow (\theta \sqrt{2\pi})^{-1}  e^{-\frac{x^2}{2}}  e^{- \theta \sqrt{2\pi} e^{\frac{x^2}{2}}w} dw.
	\end{equation*}
	Where the right hand side is the exponential distribution with rate $\theta \sqrt{2\pi} e^{\frac{x^2}{2}}$.
\end{proposition}
\begin{proof}
	Note that the measure on the left hand side in the proposition has been normalised and is a probability measure, thus it is enough to show pointwise convergence of the moment generating functions on a neighbourhood of $0$. With Theorem \ref{transition probabilities} we can rewrite the expression for the moments derived in line (\ref{Moments of the HW process}) as follows
	\begin{align*}
		&\int_{\R>0} w^n \sqrt{2\pi}t^{-\frac{1}{2}} e^{\frac{x^2}{2}}w\gamma_t\left( \sqrt{t} x, \frac{dw}{\sqrt{t}}\right) = \sqrt{2\pi} e^{\frac{x^2}{2}} t^{\frac{n+1}{2}} \int_{\R_{>0}} w^{n+1} \gamma_t \left( \sqrt{t}x, dw \right) \\
		= & \sqrt{2\pi} e^{\frac{x^2}{2}} t^{\frac{n+1}{2}} \frac{u^{(n+1)}_t((0,...,0),\sqrt{t}(x,...,x))}{(n+1)\theta^n}\\
		= &\sqrt{2\pi} \frac{e^{\frac{x^2}{2}} t^{\frac{n+1}{2}}}{(n+1)\theta^n (2\pi)^{n+1}} \int_{\R^{n+1}} e^{-\frac{1}{2}t|k|^2-i\sqrt{t} k\cdot \underline{x}} \sum_{\sigma\in S_{n+1}} \prod_{\substack{\alpha<\beta:\\ \sigma(\beta)<\sigma(\alpha)}} \frac{i\theta(k_{\sigma(\alpha)} - k_{\sigma(\beta)})+ k_{\sigma(\beta)} k_{\sigma(\alpha)}}{i\theta (k_{\sigma(\alpha)}- k_{\sigma(\beta)})- k_{\sigma(\beta)} k_{\sigma(\alpha)}} dk\\
		=& \sqrt{2\pi} \frac{n! e^{\frac{x^2}{2}} t^{\frac{n+1}{2}}}{\theta^n (2\pi)^{n+1}} \int_{\R^{n+1}} e^{-\frac{1}{2}t|k|^2-i\sqrt{t} k\cdot \underline{x}} \prod_{\alpha<\beta} \frac{i\theta(k_\beta - k_\alpha)}{i\theta (k_\beta- k_\alpha)- k_\alpha k_\beta} dk \\
		=& \sqrt{2\pi} \frac{n! e^{\frac{x^2}{2}}}{\theta^n (2\pi)^{n+1}} \int_{\R^{n+1}} e^{-\frac{1}{2}|k|^2- k\cdot \underline{x}} \prod_{\alpha<\beta} \frac{i\theta(k_\beta - k_\alpha)}{i\theta (k_\beta- k_\alpha)- t^{-\frac{1}{2}} k_\alpha k_\beta} dk.
	\end{align*}
	To go from the third to the fourth line we have used the summation formula from Lemma \ref{Lemma: sum over permutations}. We can now write the moment generating function in terms of the moments.
	\begin{align*}
		& \sqrt{2\pi}t^{-\frac{1}{2}} e^{\frac{x^2}{2}}\int_{\R>0} e^{\lambda w } w \gamma_t\left(\sqrt{t}x, \frac{dw}{\sqrt{t}}\right)\\
		 =& \sum_{n=0}^\infty \sqrt{2\pi} \frac{\lambda^n e^{\frac{x^2}{2}}}{\theta^n (2\pi)^{n+1}} \int_{\R^{n+1}} e^{-\frac{1}{2}|k|^2- k\cdot \underline{x}} \prod_{\alpha<\beta} \frac{i\theta(k_\beta - k_\alpha)}{i\theta (k_\beta- k_\alpha)- t^{-\frac{1}{2}} k_\alpha k_\beta} dk.
	\end{align*}
To take $t\to \infty$ we want to apply the Dominated Convergence Theorem to pass the limit through both the sum and the integral. Similarly to what we have seen previously, line (\ref{InitialCondition:SecondProductBound}) to be precise, the modulus of the product within the integral is bounded above by $1$. With this bound we find that the modulus of the $n^{th}$ term of the series is bounded above by $\frac{\lambda^n e^{x^2/2}}{\theta^n}$ which is uniform in $t$ and summable for $|\lambda|<\theta$, and so we can take the limit $t\to \infty$ through the sum. Further the bound on the integral allows us to take the limit through the integral. Hence we have, for $|\lambda|<\theta$
	\begin{gather*}
		\lim\limits_{t\to \infty}\sum_{n=0}^\infty \sqrt{2\pi} \frac{\lambda^n e^{\frac{x^2}{2}}}{\theta^n (2\pi)^{n+1}}\int_{\R^{n+1}}e^{-\frac{1}{2}|k|^2-ik\cdot \underline{x}} \prod_{\alpha<\beta} \frac{i\theta(k_\beta - k_\alpha)}{i\theta (k_\beta- k_\alpha)- t^{-\frac{1}{2}}k_\alpha k_\beta} dk \\
		=\sum_{n=0}^\infty \sqrt{2\pi} \frac{\lambda^n e^{\frac{x^2}{2}}}{\theta^n (2\pi)^{n+1}}\int_{\R^{n+1}}e^{-\frac{1}{2}|k|^2-ik\cdot \underline{x}} dk = \sum_{n=0}^{\infty} \left(\frac{\lambda e^{-\frac{x^2}{2}}}{\theta \sqrt{2\pi}}\right)^n.
	\end{gather*}
	This is exactly the moment generating function of an exponential random variable with parameter $\theta \sqrt{2\pi} e^{x^2/2}$, and thus the statement is proved.
\end{proof}
We note that this result is analogous to Thiery and Le Doussal's result in \cite{ExactSolutionRWRE1dThieryLeDoussal}, where they found that the fluctuations of the transition probabilities of the Beta RWRE were Gamma distributed in the large $t$ limit. We also have the following Fredholm determinant formula analogous to formula (52) in \cite{ExactSolutionRWRE1dThieryLeDoussal}.
\begin{proposition}
    \begin{equation}
        1+ \sum_{n=1}^\infty \int_{\R_{>0}} \frac{(\lambda w)^n}{n! (n-1)!} \gamma_t(y, dw) = \theta \det\left( I + \frac{\lambda}{\theta 2\pi} K\right).
    \end{equation}
    Above, the determinant is a Fredholm determinant and $K$ is an integral operator on $L^2(\R)$ with kernel
    \begin{equation}
        K(x,y)= \frac{xy e^{-\frac{1}{4}t(x^2+ y^2)}}{i\theta(y-x) + xy}.
    \end{equation}
\end{proposition}
\begin{proof}
Equation (\ref{atomic moments}) and the summation formula in Lemma \ref{Lemma: sum over permutations} give the equality
	\begin{equation*}
	\int_{\R>0}w^n \gamma_t(y, dw)= \frac{(n-1)!}{\theta^{n-1} (2\pi)^{n}}\int_{\R^{n}} e^{-\frac{1}{2}t|k|^2-ik\cdot \underline{y}} \prod_{\alpha<\beta} \frac{i\theta(k_\beta - k_\alpha)}{i\theta (k_\beta- k_\alpha)- k_\alpha k_\beta} dk.
\end{equation*}
The proof is completed by the following identity, which is a consequence of the equalities (A.1) and (D.1) in \cite{ExactSolutionRWRE1dThieryLeDoussal}
\begin{align}
    \sum_{\sigma \in S_n} \prod_{\alpha< \beta} \frac{i \theta (k_{\sigma(\beta)} - k_{\sigma(\alpha)})}{i\theta(k_{\sigma(\beta)} - k_{\sigma(\alpha)}) - k_{\sigma(\alpha)} k_{\sigma(\beta)}} = n! \det_{1\leq \alpha, \beta\leq n} \left[ \frac{k_\beta k_\alpha}{i\theta(k_\beta - k_\alpha) + k_\alpha k_\beta} \right].
\end{align}
\end{proof}
It would be interesting to use the above formula to analyse the behaviour of $\gamma_t$ in the large deviation regime,  $\frac{y}{t}$ converges to a non zero number as $t\to \infty$, where we expect the appearance of GUE Tracey-Widom fluctuations. Unfortunately the above Fredholm determinant is not in an ideal form for asymptotic analysis. We would instead want an analogue of the conjectured formula (92) in \cite{ExactSolutionRWRE1dThieryLeDoussal}. In \cite{barraquand2019large} Barraquand and Rychnovsky instead considered the tails of the Howitt-Warren process, $\rho_t([tx, \infty])$, and derived a Fredholm determinant formula for the Laplace transform via a scaling limit from the Beta random walk in a random environment, with which they were able to prove the existence of GUE fluctuations. We make the following conjecture for the fluctuations of the individual atoms.
\begin{conjecture}
    If $X_{x,t}$ is a random variable on $\R$ with law $\sqrt{2\pi t} e^{-t\frac{x^2}{2}} w \gamma_t(tx, dw)$, then there are functions $J:\R\to \R$ and $\sigma:\R \to \R$ such that
    \begin{align}
        \lim_{t\to \infty}\p\left(\frac{\log(X_{x,t}) + J(x)t}{t^{1/3}\sigma(x)} <z \right) = F_{GUE}(z),
    \end{align}
    where $F_{GUE}$ is the cumulative function for the Tracey-Widom GUE distribution.
\end{conjecture}

\cleardoublepage
\bibliography{bibliography}{}
\bibliographystyle{plain}
\end{document}